\documentclass[11pt,twoside]{amsart}
\usepackage{amsmath,amsfonts,amscd,amsthm}
\usepackage{amssymb}
\usepackage[flushmargin]{footmisc}
\usepackage{color}                    
\usepackage{esvect}
\usepackage{mathrsfs}

\usepackage[colorlinks,
linkcolor=blue,
anchorcolor=blue,
citecolor=blue]{hyperref} 

\newtheorem{remark}{Remark}
\newcommand{\dbar}{\ensuremath{\overline\partial}}
\newcommand{\dbarstar}{\ensuremath{\overline\partial^*}}
\newcommand{\C}{\ensuremath{\mathbb{C}}}

\newcommand{\R}{\ensuremath{\mathbb{R}}}

\newcommand{\norm}[1]{\left\Vert#1\right\Vert}
\newcommand{\abs}[1]{\left\vert#1\right\vert}

\newcommand{\set}[1]{\left\{#1\right\}}
\newcommand{\To}{\rightarrow}

\newcommand{\cali}[1]{\mathscr{#1}}

\DeclareFontFamily{U}{mathx}{\hyphenchar\font45}
\DeclareFontShape{U}{mathx}{m}{n}{
      <5> <6> <7> <8> <9> <10>
      <10.95> <12> <14.4> <17.28> <20.74> <24.88>
      mathx10
      }{}
\DeclareSymbolFont{mathx}{U}{mathx}{m}{n}
\DeclareFontSubstitution{U}{mathx}{m}{n}
\DeclareMathAccent{\widecheck}{0}{mathx}{"71}
\DeclareMathAccent{\wideparen}{0}{mathx}{"75}

\def\eps{\varepsilon}

\def\omz{\Omega}

\makeatletter
\newcommand{\sumprime}{\if@display\sideset{}{'}\sum%
	\else\sum'\fi}
\makeatother

\newtheorem{thm}{Theorem}[section]
\newtheorem{prop}[thm]{Proposition}
\newtheorem{lem}[thm]{Lemma}

\theoremstyle{definition}
\newtheorem{defin}[thm]{Definition}

\theoremstyle{remark}

\numberwithin{equation}{section}

\providecommand\ufootnote[1]{{\let\thefootnote\relax\footnote[0]{#1}}}

\newcommand{\N}{\mathbb{N}}

\newcommand{\ol}{\overline}

\newcommand{\pa}{\partial}

 \DeclareMathOperator{\im}{Im}
\DeclareMathOperator{\ke}{Ker}

 \DeclareMathOperator{\Dom}{Dom}
  
\DeclareMathOperator{\Tr}{Tr} 

\begin{document}

\title[Heat kernel asymptotics]{Heat kernel asymptotics for Kohn Laplacians on CR manifolds}

\author{Chin-Yu Hsiao and Weixia Zhu}

\address{Institute of Mathematics, Academia Sinica, 6F, Astronomy-Mathematics Building, No.1, Sec.4, Roosevel} \email{chsiao@math.sinica.edu.tw; chinyu.hsiao@gmail.com}

\address{Fakult\"{a}t f\"{u}r Mathematik, Universit\"{a}t of Wien, Oskar-Morgenstern-Platz 1, 1090 Wien, Austria} \email{weixia.zhu@univie.ac.at; zhuvixia@gmail.com}

\begin{abstract} 
	Let $X$ be an abstract orientable (not necessarily compact)
	CR manifold of dimension $2n+1$, $n\geq1$, and let $L^k$ be the $k$-th tensor power of a CR complex line bundle $L$ over $X$. Suppose that condition $Y(q)$ holds at each point of $X$, we establish asymptotics of the heat kernel of Kohn Laplacian with values in $L^k$. As an application, we give a heat kernel proof of Morse  inequalities on compact CR manifolds. When $X$ admits a transversal CR $\mathbb R$-action, we also 
	establish  asymptotics of the $\mathbb R$-equivariant heat kernel of Kohn Laplacian with values in $L^k$. As an application, we get $\mathbb R$-equivariant Morse  inequalities on compact CR manifolds with transversal CR $\mathbb R$-action. 	
\bigskip


	\smallskip
	
\end{abstract}

\maketitle

\tableofcontents

\section{Introduction}\label{sec:intro}

Let $X$ be an orientable  compact CR manifold of dimension $2n+1$ and let $L^k$ be the $k$-th tensor power of a CR complex line bundle $L$ over $X$. The first author  and Marinescu established in \cite{HM12} the Morse inequalities on CR manifolds for $\dbar_b$ complex with values in $L^k$. Their approach bases on the study of asymptotic  behavior of the Szeg\H{o} kernels.  
In complex geometry, Demailly\cite{D91},  Bismut\cite{Bismut87} used the heat kernel asymptotics  for Kodaira Laplacians to establish  Morse inequalities on complex manifolds (see also~\cite[Section 1.6]{MM07}). It is a natural question that if we can also establish Morse inequalities on CR manifolds by using heat kernel asymptotics. On the other hand, the study of the heat kernel asymptotics is interesting in itself because it is closely related to some problems in CR geometry like  CR torsion problems, CR geometric quantization, $G$-equivariant CR Morse inequalities, etc. The purpose of this paper is to establish the heat kernel asymptotics  of Kohn Laplacian with values in $L^k$. 
When $X$ admits a transversal CR $\mathbb R$-action, we also establish  asymptotics of the $\mathbb R$-equivariant heat kernel of Kohn Laplacian with values in $L^k$. As an application, we get $\mathbb R$-equivariant Morse inequalities on compact CR manifolds with transversal CR $\mathbb R$-action. 
We will combine the  techniques of~\cite{HM12} and Ma-Marinescu~\cite[Section 1.6]{MM07} to get the heat kernel asymptotics of Kohn Laplacian on CR manifolds with values in $L^k$. 

We now formulate the main results. 
We refer the reader to Section~\ref{sec:prelim} for some notations and terminology used here. Let $(X,T^{1,0}X)$ be an abstract orientable (not necessarily compact)
CR manifold of dimension $2n+1$, $n\geq1$, and let $(L^k,h^{L^k})$ be the $k$-th tensor power of a CR complex line bundle $(L,h^L)$ over $X$, 
where $h^L$ denotes a Hermitian metric of $L$. Fix a Hermitian metric $\langle\,\cdot\mid\cdot\,\rangle$ on $\mathbb CTX$ so that 
$T^{1,0}X\perp T^{0,1}X$. Let $(\,\cdot\mid\cdot\,)_{h^{L^k}}$ be the $L^2$ inner product on $\Omega^{0,q}_c(X,L^k)$ induced by 
$\langle\,\cdot\mid\cdot\,\rangle$ and $h^{L^k}$. Let $L^2_{(0,q)}(X,L^k)$ be the completion of $\Omega^{0,q}_c(X,L^k)$ with respect to 
$(\,\cdot\mid\cdot\,)_{h^{L^k}}$. Let 
\[\Box^q_{b,k}: {\rm Dom\,}\Box^q_{b,k}\subset L^2_{(0,q)}(X,L^k)\To L^2_{(0,q)}(X,L^k)\]
be the Gaffney extension of Kohn Laplacian with values in $L^k$ (see \eqref{e-gue210228yyd}). Then, $\Box^q_{b,k}$ is a non-negative self-adjoint operator. 
Let 
\[e^{-t\Box^q_{b,k}}: L^2_{(0,q)}(X,L^k)\To L^2_{(0,q)}(X,L^k)\]
be the heat operator of $\Box^q_{b,k}$ and let 
\[e^{-t\Box^q_{b,k}}(x,y)\in\mathscr D'(\mathbb R_+\times X\times X,(T^{*0,q}X\otimes L^k)\boxtimes(T^{*0,q}X\otimes L^k)^*)\]
be the distribution kernel of $e^{-t\Box^q_{b,k}}$ with respect to $dv_X(x)$ (see the discussion after \eqref{e-gue210607yyd}), where $dv_X(x)$ is the volume form on $X$ 
induced by $\langle\,\cdot\mid\cdot\,\rangle$. Assume that $Y(q)$ holds (see Definition~\ref{d-gue210607yyd}, for the meaning of condition $Y(q)$). Then, 
\[e^{-t\Box^q_{b,k}}(x,y)\in\cali{C}^\infty(\mathbb R_+\times X\times X,(T^{*0,q}X\otimes L^k)\boxtimes(T^{*0,q}X\otimes L^k)^*).\]
We will use the canonical identification ${\rm End\,}(L^k)=\mathbb C$, especially 
\[e^{-t\Box^q_{b,k}}(x,x)\in\cali{C}^\infty(\mathbb R_+\times X\times X,T^{*0,q}X\boxtimes(T^{*0,q}X)^*)\]
(see \eqref{e-gue210303yydIII}). The main goal of this work is to study the asymptotic behavior $e^{-\frac{t}{k}\Box^q_{b,k}}(x,x)$ as $k\To+\infty$. 

To state our main results, we introduce some notations. Let $s$ be a local CR trivializing section of $L$ defined on an open set $D$ of $X$, $\abs{s}^2_{h^L}=e^{-\phi}$, 
$\phi\in\cali{C}^\infty(D,\mathbb R)$. For $x\in D$, let $\mathcal R_x^\phi$ be the Hermitian quadratic form on $T^{1,0}_xX$ induced by $\phi$ (see Definition~\ref{d-gue210524yyd}). 
Let 
\begin{equation}\label{e-gue210607yyds}
\begin{split}
&\dot{\mathcal{R}}^\phi_x: T^{1,0}_xX\To T^{1,0}_xX,\\
&\dot{\mathcal{L}}_x: T^{1,0}_xX\To T^{1,0}_xX,
\end{split}
\end{equation}
be the linear maps given by $\langle\,\dot{\mathcal{R}}^\phi_xU\mid V\,\rangle=\mathcal{R}^\phi_x(U,\ol V)$, 
$\langle\,\dot{\mathcal{L}}_xU\mid V\,\rangle=\mathcal{L}_x(U,\ol V)$, for all $U, V\in T^{1,0}_xX$, where $\mathcal{L}_x$ is the Levi form of $X$ at $x$ (see Definition~\ref{d-gue210607ycd}). For every $\eta\in\mathbb R$, let 
\[{\rm det\,}(\dot{\mathcal{R}}^\phi_x-2\eta\dot{\mathcal{L}}_x)=\mu_1(x)\cdots\mu_n(x),\]
where $\mu_j(x)$, $j=1,\ldots,n$, are the eigenvalues of $\dot{\mathcal{R}}^\phi_x-2\eta\dot{\mathcal{L}}_x$ with respect to $\langle\,\cdot\mid\cdot\,\rangle$, 
and put 
\begin{equation}\label{e-gue210607yydt}
\omega_x^\eta=\sum_{j,l=1}^n(\mathcal{R}^\phi_x-2\eta\mathcal L_x)(U_l,\ol U_j)\ol\omega^j\wedge(\ol\omega^l\wedge)^\star: T^{*0,q}_xX\To T^{*0,q}_xX,
\end{equation}
where $\{U_j\}_{j=1}^n$ is an orthonormal frame of $T^{1,0}_xX$ with dual frame $\{\omega^j\}_{j=1}^n\subset T^{*1,0}_xX$. 
It is straightforward to check that if $Y(q)$ holds, the map 
\begin{equation}\label{e-gue210718yydt}
\int_\R\dfrac{\det(\dot{\mathcal R}^\phi_x-2\eta\dot{\mathcal L}_x)}{\det\big(1-e^{-t(\dot{\mathcal R}^\phi_x-2\eta\dot{\mathcal L}_x)}\big)}e^{-t\omega_x^\eta}d\eta: T^{*0,q}_xX\To T^{*0,q}_xX
\end{equation}
is well-defined as a continuous linear map. 
It should be mentioned that the definition of $\mathcal R_x^{\phi}$ depends on the choice of local trivializations (local weight $\phi$).
From \cite[Proposition 4.2]{HM12}, it is easy to see that for every $x\in X$, the map \eqref{e-gue210718yydt} is independent of the choice of local weight $\phi$ and hence globally defined. Let $\mathscr L(T^{*0,q}_xX,T^{*0,q}_xX)$ denotes the linear map from $T^{*0,q}_xX$ to $T^{*0,q}_xX$. For $f(x)\in \mathscr L(T^{*0,q}_xX,T^{*0,q}_xX)$, let 
\begin{equation}\label{e-gue210607yydu}
\abs{f(x)}_{\mathscr L(T^{*0,q}_xX,T^{*0,q}_xX)}:=\sum^d_{j,\ell=1}\abs{\langle\,f(x)v_j(x)\mid v_\ell(x)\,\rangle}, 
\end{equation}
where $\set{v_j(x)}^d_{j=1}$ is an orthonormal basis of $T^{*0,q}_xX$ with respect to $\langle\,\cdot\mid\cdot\,\rangle$. 

 Our first main result is the following

\begin{thm}\label{main2}
	Let $X$ be an orientable (not necessarily compact)
CR manifold of dimension $2n-1$, $n\geq1$, and $(L^k,h^{L^k})$ be the $k$-th tensor power of a CR complex line bundle $(L,h^L)$ over $X$.
Given $q\in\set{0,1,\ldots,n}$. 
	Suppose that condition $Y(q)$ holds at each point of $X$. With the notations used above,  let $I\subset\mathbb R_+$ be a compact interval and let $K\Subset X$ be a compact set. Then, there is a constant $C>0$ independent of $k$ such that 
\begin{equation}\label{e-gue210607yydx}
\abs{e^{-t\Box^q_{b,k}}(x,x)}_{\mathscr L(T^{*0,q}_xX,T^{*0,q}_xX)}\leq Ck^{n+1},\ \ \mbox{for all $x\in K$ and $t\in I$}. 
\end{equation}

Moreover, for every $x\in X$, we have 
	\begin{equation}\label{asymoptptic}
	\begin{aligned}
	\lim_{k\to\infty}k^{-(n+1)}e^{-\frac{t}{k}\Box_{b,k}^q}(x,x)=&\frac{1}{(2\pi)^{n+1}}\int_\R\dfrac{\det(\dot{\mathcal R}^\phi_x-2\eta\dot{\mathcal L}_x)}{\det\big(1-e^{-t(\dot{\mathcal R}^\phi_x-2\eta\dot{\mathcal L}_x)}\big)} e^{-t\omega_x^\eta}d\eta.
	\end{aligned}
	\end{equation}
\end{thm} 

Consider the following $\dbar_b$-complex: 
$$
\cdots\longrightarrow\,\omz^{0,q-1}(X, L^k)\,\mathop{\longrightarrow}\limits^{\dbar_{b}}\omz^{0,q}(X, L^k) \,\mathop{\longrightarrow}\limits^{\dbar_{b}}\,\omz^{0,q+1}(X, L^k)\,\longrightarrow\cdots
$$
and  put 
\[
H^q_{b}(X,L^k):=\dfrac{\ke \dbar_{b}:\omz^{0,q}(X,L^k)\to \omz^{0,q+1}(X,L^k)}{\im \dbar_{b}: \omz^{0,q-1}(X,L^k)\to \omz^{0,q}(X,L^k)},\quad 0\le q\le n.\] 
Kohn(see \cite{K65,FK72}) showed that if $Y(q)$ holds and $X$ is compact, then 
\[
H^{q}_{b}(X,L^k)\cong{\rm Ker\,}\Box^q_{b,k},\ \ {\rm dim\,}H^{q}_{b}(X,L^k)={\rm dim\,}{\rm Ker\,}\Box^q_{b,k}<+\infty.\]
When $X$ is compact and $Y(q)$ holds on $X$, 
From Theorem~\ref{main2} as well as applying fundamental trace inequalities which depend on $k>0$(see, e.g., \cite[Theorem 1.4]{Bismut87}, \cite[Section 1.6]{MM07}), we get the following Hsiao-Marinescu's Morse inequalities on CR manifolds:

\begin{thm}[Hsiao-Marinescu]\label{thmmorse}
Let $X$ be an orientable compact
CR manifold with $\dim_\R X=2n+1$, $n\geq1$, and $(L^k,h^{L^k})$ be the $k$-th tensor power of a CR complex line bundle $(L,h^L)$ over $X$.
	 Suppose that condition $Y(j)$ holds at each point of $X$, for all $j=0,1,2,\cdots, q$. With the notations used above, as $k\to\infty$, we have 	
	 \begin{equation}\label{strong}
	\begin{split}
	&\sum_{j=0}^q(-1)^{q-j}\dim H_b^{j}(X,L^k)\\
	& \le\frac{k^{n+1}}{(2\pi)^{n+1}}\sum^q_{j=0}(-1)^{q-j}\int_X\int_{\R_x(j)}\left|\det(\dot{\mathcal R}^\phi_x-2\eta\dot{\mathcal L}_x)\right|d\eta dv_X(x)+o(k^{n+1}),
	\end{split}
	\end{equation}
	where 
	\begin{equation*}
	\begin{aligned}
	\R_x(j)=\{\eta&\in\R\mid \dot{\mathcal R}^\phi_x-2\eta\dot{\mathcal L}_x\,\text{has exactly $j$ negative eigenvalues} 
	\\&\text{and $n-j$ positive eigenvalues}\}.
	\end{aligned}
	\end{equation*}
	
	When $Y(q)$ holds, we have
	\begin{equation}\label{weak}
	\dim H_b^{q}(X,L^k) \le\frac{k^{n+1}}{(2\pi)^{n+1}}\int_X\int_{\R_x(q)}\left|\det(\dot{\mathcal R}^\phi_x-2\eta\dot{\mathcal L}_x)\right|d\eta dv_X(x)+o(k^{n+1}).
	\end{equation}
\end{thm} 

It is worth to mention that the function
\[x\in X\To \int_{\R_x(j)}\left|\det(\dot{\mathcal R}^\phi_x-2\eta\dot{\mathcal L}_x)\right|d\eta\in\mathbb R\]
is independent of the choice of local weight $\phi$.

As we see, to establish above theorems, we need the assumptions that the Levi form satisfies condition $Y(q)$. However, in some important problems in CR geometry, we may need to study things without any assumption of the Levi form. Therefore, we consider CR manifolds with group action. Let $X$ be an orientable  (not necessarily compact)
CR manifold with $\dim_\R X=2n+1$, $n\geq1$,. Suppose that $X$ admits a transversal and CR $\mathbb R$-action $\eta$, $\eta\in\mathbb R$: $\eta: X\to X$, $x\mapsto\eta\circ x$ (see Definition~\ref{d-gue210531yyd} for the meaning of transversal and CR $\mathbb R$-action).  Let $T\in\cali{C}^\infty(X, TX)$ be the infinitesimal generator of   the $\mathbb R$-action (see \eqref{e-gue150808}).  Suppose that $X$ admits a $\mathbb R$-invariant complete Hermitain metric $\langle\,\cdot\mid\cdot\,\rangle$ on $\mathbb CTX$ so that we have the orthogonal decomposition \[
\C TX=T^{1,0}X\oplus T^{0,1}X\oplus\{\lambda T: \lambda\in\C\}
\]
and $|T|^2=\langle\,T\mid T\,\rangle=1$.  Let $(L,h^L)$ be a rigid CR line bundle over $X$ with a $\mathbb R$-invariant Hermitian metric $h^L$ on $L$ (see Definition~\ref{Def:RigidCVB} for the meaning of rigid CR line bundles).  We will use the same notations as before.  Consider the operator
\[-iT: \Omega^{0,q}_c(X,L^k)\To\Omega^{0,q}_c(X,L^k)\]
and we extend $-iT$ to $L^2_{(0,q)}(X,L^k)$ in the standard way (see \eqref{e-gue210608yyd}). 
Fix $\delta>0$, let 
\[
\begin{split}
&L^2_{(0,q),\leq k\delta}(X,L^k):=E_{-iT}([-k\delta,k\delta]),\\
&\Omega^{0,q}_{\leq k\delta}(X,L^k):=\Omega^{0,q}(X,L^k)\bigcap L^2_{(0,q),\leq k\delta}(X,L^k),
\end{split}\]
where $E_{-iT}$ denotes the spectral measure of $-iT$. Let 
\[
Q_{X,\leq k\delta}: L^2_{(0,q)}(X,L^k)\To L^2_{(0,q),\leq k\delta}(X,L^k)\]
be the orthogonal projection with respect  to $(\,\cdot\mid\cdot\,)_{h^{L^k}}$. 

We have the following partial $\dbar_b$-complex: 
$$
\cdots\longrightarrow\,\omz^{0,q-1}_{\leq k\delta}(X, L^k)\,\mathop{\longrightarrow}\limits^{\dbar_{b}}\omz^{0,q}_{\leq k\delta}(X, L^k) \,\mathop{\longrightarrow}\limits^{\dbar_{b}}\,\omz^{0,q+1}_{\leq k\delta}(X, L^k)\,\longrightarrow\cdots
$$
and  put 
\[
H^q_{b,\leq k\delta}(X,L^k):=\dfrac{\ke \dbar_{b}:\omz^{0,q}_{\leq k\delta}(X,L^k)\to \omz^{0,q+1}_{\leq k\delta}(X,L^k)}{\im \dbar_{b}: \omz^{0,q-1}_{\leq k\delta}(X,L^k)\to \omz^{0,q}_{\leq k\delta}(X,L^k)},\quad 0\le q\le n.\]
Let 
\[
\begin{split}
&\Box^q_{b,k,\leq k\delta}: {\rm Dom\,}\Box^q_{b,k}\bigcap L^2_{(0,q),\leq k\delta}(X,L^k) \To L^2_{(0,q),\leq k\delta}(X,L^k),\\
&\Box^q_{b,k,\leq k\delta}u=\Box^q_{b,k}u,\ \ u\in {\rm Dom\,}\Box^q_{b,k}\bigcap L^2_{(0,q),\leq k\delta}(X,L^k).
\end{split}\]
(See also \eqref{e-gue210531yyde}.)
It is not difficult to check that without $Y(q)$ condition, 
\[H^q_{b,\leq k\delta}(X,L^k)\cong{\rm Ker\,}\Box^q_{b,k,\leq k\delta},\ \ {\rm dim\,}H^q_{b,\leq k\delta}(X,L^k)<+\infty.\]

Let
\[
e^{-t\Box^q_{b,k,\leq k\delta}}:=e^{-t\Box^q_{b,k}}\circ Q_{X,\leq k\delta}: L^2_{(0,q)}(X,L^k)\To{\rm Dom\,}\Box^q_{b,k,\leq k\delta},\ \ t>0,\]
and let 
\[e^{-t\Box^q_{b,k,\leq k\delta}}(x,y)\in\cali{C}^\infty(\mathbb R_+\times X\times X, (T^{*0,q}X\otimes L^k)\boxtimes(T^{*0,q}X\otimes L^k)^*)\]
 be the distribution kernel of $e^{-t\Box^q_{b,k,\leq k\delta}}$ with respect to $dv_X(x)$.  As before, we will use the canonical identification ${\rm End\,}(L^k)=\mathbb C$. Our second main result is the following

\begin{thm}\label{main3}
	Let $X$ be an orientable (not necessarily compact) CR manifold of dimension $2n-1$, $n\geq1$, with a transversal CR $\mathbb R$-action. Suppose that $X$ admits a $\mathbb R$-invariant complete Hermitian metric $\langle\,\cdot\mid\cdot\,\rangle$ on $\mathbb CTX$ so that we have the orthogonal decomposition \[
\C TX=T^{1,0}X\oplus T^{0,1}X\oplus\{\lambda T: \lambda\in\C\}
\]
and $|T|^2=\langle\,T\mid T\,\rangle=1$, where $T\in\cali{C}^\infty(X, TX)$ is the infinitesimal generator of the $\mathbb R$-action. Let $(L^k,h^{L^k})$ be the $k$-th tensor power of a rigid CR complex line bundle $(L,h^L)$ over $X$, where $h^L$ is a $\mathbb R$-invariant Hermitian metric  on $L$. Fix any $q\in\set{0,1,\ldots,n}$. With the notations used above, let $I\subset\mathbb R_+$ be a compact interval and let $K\Subset X$ be a compact set. Then, there is a constant $C>0$ independent of $k$ such that 
\begin{equation}\label{e-gue210608yyda}
\abs{e^{-t\Box^q_{b,k,\leq k\delta}}(x,x)}_{\mathscr L(T^{*0,q}_xX,T^{*0,q}_xX)}\leq Ck^{n+1},\ \ \mbox{for all $x\in K$ and $t\in I$}. 
\end{equation}

Moreover, for every $x\in X$, we have 
	\begin{equation}\label{asymoptpticS1}
	\begin{aligned}
	\lim_{k\to\infty}k^{-(n+1)}e^{-\frac{t}{k}\Box_{b,k,\leq k\delta}^q}(x,x)=&\frac{1}{(2\pi)^{n+1}}\int^\delta_{-\delta}\dfrac{\det(\dot{\mathcal R}^\phi_x-2\eta\dot{\mathcal L}_x)}{\det\big(1-e^{-t(\dot{\mathcal R}^\phi_x-2\eta\dot{\mathcal L}_x)}\big)} e^{-t\omega_x^\eta}d\eta.
	\end{aligned}
	\end{equation}
\end{thm} 

Note that in Theorem~\ref{main3}, we do not need any Levi curvature assumption. 
As an application of Theorem~\ref{main3}, we get the following Morse inequalities which generalize the results in \cite{HL18} for CR manifolds with $S^1$ action:

\begin{thm}\label{thms1}
With the same assumptions and notations used in Theorem~\ref{main3}, suppose that $X$ is compact. 
		 As $k\to\infty$, we have the following Morse inequalities: for every $q\in\set{0,1,\ldots,n}$, 
	\begin{equation}\label{strongS1}
	\begin{aligned}
	&\sum_{j=0}^q(-1)^{q-j}\dim H_{b,\leq k\delta}^{j}(X,L^k)\\
	& \le\frac{k^{n+1}}{(2\pi)^{n+1}}\sum^q_{j=0}(-1)^{q-j}\int_X\int_{\R_x(j)\bigcap[-\delta,\delta]}\left|\det(\dot{\mathcal R}^\phi_x-2\eta\dot{\mathcal L}_x)\right|d\eta dv_X(x)+o(k^{n+1})
	\end{aligned}
	\end{equation}
	with equality for $q=n$. 
	
	For every $q\in\set{0,1,\ldots,n}$, we have 
	\begin{equation}\label{weakS1}
	\dim H_{b,\leq k\delta}^{q}(X,L^k) \le\frac{k^{n+1}}{(2\pi)^{n+1}}\int_X\int_{\R_x(q)\bigcap[-\delta,\delta]}\left|\det(\dot{\mathcal R}^\phi_x-2\eta\dot{\mathcal L}_x)\right|d\eta dv_X(x)+o(k^{n+1}).
	\end{equation}
\end{thm}

Our paper is organized in the following way. In Section~\ref{sec:prelim}, 
we collect some standard notations, terminology, definitions, and statements 
we use throughout. 
In Section~\ref{estimate}, we introduce the Kohn Laplacian $\Box_{H_n,\Phi}^q$ on the Heisenberg group and find the relationship between the smooth kernel of heat operators $e^{-\frac{t}{k}\Box_{b,k}^q}$ and $e^{-t\Box_{H_n,\Phi}^q}$ in local. This process is achieved by comparing the scaled Kohn Laplacian $\Box_{\rho,(k)}^q$ and $\Box_{H_n,\Phi}^q$. We then compute $e^{-t\Box_{H_n,\Phi}^q}(0,0)$ explicitly and establish Theorem~\ref{main2} in Section~\ref{heatasymptotic}.  
In Section~\ref{morseinequality}, we obtain Morse inequalities on CR manifolds by using Theorem~\ref{main2}.
In Section~ \ref{S1}, we show that the same conclusions can be drawn when $X$ admits a transversal CR $\mathbb R$-action, without any assumption of the Levi form. 

Throughout this paper, we will use $C$($C_t$) to denote constants(the constants depend on $t$) which might not be the same in different appearances.

\section{Preliminary}\label{sec:prelim}

\subsection{Standard notations} \label{s-ssna}
We shall use the following notations: $\mathbb N=\{1,2,\ldots\}$
is the set of natural numbers, $\mathbb N_0=\mathbb N\cup\set{0}$, $\mathbb R$ 
is the set of real numbers, $\mathbb R_+:=\{x\in\mathbb R\mid x>0\}$,
$\overline{\mathbb R_+}:=\{x\in\mathbb R\mid x\geq0\}$. 
For a multiindex $\alpha=(\alpha_1,\ldots,\alpha_n)\in\mathbb N_0^n$
we denote by $\abs{\alpha}=\alpha_1+\ldots+\alpha_n$ its norm and by $l(\alpha)=n$ its length.
For $m\in\mathbb N$, write $\alpha\in\set{0,1,\ldots,m}^n$ if $\alpha_j\in\set{0,1,\ldots,m}$, 
$j=1,\ldots,n$. $\alpha$ is strictly increasing if $\alpha_1<\alpha_2<\ldots<\alpha_n$. For $x=(x_1,\ldots,x_n)$ we write
\[
\begin{split}
&x^\alpha=x_1^{\alpha_1}\ldots x^{\alpha_n}_n,\\
& \partial_{x_j}=\frac{\partial}{\partial x_j}\,,\quad
\partial^\alpha_x=\partial^{\alpha_1}_{x_1}\ldots\partial^{\alpha_n}_{x_n}=\frac{\partial^{\abs{\alpha}}}{\partial x^\alpha}\,,\\
&D_{x_j}=\frac{1}{i}\partial_{x_j}\,,\quad D^\alpha_x=D^{\alpha_1}_{x_1}\ldots D^{\alpha_n}_{x_n}\,.
\end{split}
\]
Let $z=(z_1,\ldots,z_n)$, $z_j=x_{2j-1}+ix_{2j}$, $j=1,\ldots,n$, be coordinates of $\mathbb C^n$.
We write
\[
\begin{split}
&z^\alpha=z_1^{\alpha_1}\ldots z^{\alpha_n}_n\,,\quad\ol z^\alpha=\ol z_1^{\alpha_1}\ldots\ol z^{\alpha_n}_n\,,\\
&\partial_{z_j}=\frac{\partial}{\partial z_j}=
\frac{1}{2}\Big(\frac{\partial}{\partial x_{2j-1}}-i\frac{\partial}{\partial x_{2j}}\Big)\,,\quad\partial_{\overline z_j}=
\frac{\partial}{\partial\overline z_j}=\frac{1}{2}\Big(\frac{\partial}{\partial x_{2j-1}}+i\frac{\partial}{\partial x_{2j}}\Big),\\
&\partial^\alpha_z=\partial^{\alpha_1}_{z_1}\ldots\partial^{\alpha_n}_{z_n}=\frac{\partial^{\abs{\alpha}}}{\partial z^\alpha}\,,\quad
\partial^\alpha_{\overline z}=\partial^{\alpha_1}_{\overline z_1}\ldots\partial^{\alpha_n}_{\overline z_n}=
\frac{\partial^{\abs{\alpha}}}{\partial\overline z^\alpha}\,.
\end{split}
\]
For $j, s\in\mathbb Z$, set $\delta_{j,s}=1$ if $j=s$, $\delta_{j,s}=0$ if $j\neq s$.

Let $W$ be a $\cali{C}^\infty$ paracompact manifold.
We let $TW$ and $T^*W$ denote the tangent bundle of $W$
and the cotangent bundle of $W$ respectively.
The complexified tangent bundle of $W$ and the complexified cotangent bundle of $W$ are denoted by $\mathbb C TW$
and $\mathbb C T^*W$, respectively. Write $\langle\,\cdot\,,\cdot\,\rangle$ to denote the pointwise
duality between $TW$ and $T^*W$.
We extend $\langle\,\cdot\,,\cdot\,\rangle$ bilinearly to $\mathbb C TW\times\mathbb C T^*W$.
Let $G$ be a $\cali{C}^\infty$ vector bundle over $W$. The fiber of $G$ at $x\in W$ will be denoted by $G_x$.
Let $E$ be another vector bundle over $W$. We write
$E\boxtimes G^*$ to denote the vector bundle over $W\times W$ with fiber over $(x, y)\in W\times W$
consisting of the linear maps from $G_y$ to $E_x$.  Let $Y\subset W$ be an open set. 
From now on, the spaces of distribution sections of $G$ over $Y$ and
smooth sections of $G$ over $Y$ will be denoted by $\mathscr D'(Y, G)$ and $\cali{C}^\infty(Y, G)$ respectively.
Let $\mathscr E'(Y, G)$ be the subspace of $\mathscr D'(Y, G)$ whose elements have compact support in $Y$ and 
let $\cali{C}^\infty_c(Y, G)$ be the subspace of $\cali{C}^\infty(Y, G)$ whose elements have compact support in $Y$.

\subsection{Set up}\label{s-gue210504yyd}

In this section, we introduce the basic notations on CR manifolds. It is worthwhile to mention that most of our notations follow the fundamental setting of \cite{HM12} and \cite{HL18}.

Let $(X,T^{1,0}X)$ be an orientable $2n+1$-dimensional (possibly non-compact) CR manifold, equipped with a 
Hermitian metric $\langle\,\cdot\,|\,\cdot\,\rangle$, so that $T^{1,0}X$ is orthogonal to $T^{0,1}X:=\ol{T^{1,0}X}$, $n\geq1$. Then there is a real non-vanishing vector field $T$ on $X$ which is pointwise orthogonal to $T^{1,0}X\oplus T^{0,1}X$, $\|T\|^2=\langle\,T\mid T\,\rangle=1$. We then have the following orthogonal decompositions:
\[
\C TX=T^{1,0}X\oplus T^{0,1}X\oplus\{\lambda T\mid\lambda\in\C\}
\]
and 
\[
\C T^*X=T^{*1,0}X\oplus T^{*0,1}X\oplus\{\lambda \omega_0\mid\lambda\in\C\},
\]
where 
$T^{*1,0}X$, $T^{*0,1}X$ and $\set{\lambda\omega_0\mid\lambda\in\mathbb C}$ are the dual bundles of $T^{1,0}X$,  $T^{0,1}X$ and $\set{\lambda T\mid\lambda\in\mathbb C}$, respectively, $\langle\,\omega_0\,,\,T\,\rangle=-1$. The Hermitian metric on $\C T^*X$ is induced by $\langle\,\cdot\,|\,\cdot\,\rangle$ by duality. Hereafter, we denote all these induced metrics by $\langle\,\cdot\,|\,\cdot\,\rangle$. Note the vector field $T$ here is first choosed locally, then by restriction to the coordinate transformations which preserve $T^{1,0}X$ and orientation, a local choice of sign in the direction of $T$ can extend $T$ to a global one.

\begin{defin}\label{d-gue210607ycd}
	For $x\in X$, the Hermitian quadratic form $\mathcal L_x$ on $T^{1,0}_xX$ defined by 
	\begin{equation}
	\mathcal L_x(U,\ol V)=\frac{1}{2i}\left\langle[\mathcal U,\ol{\mathcal V}](x),\omega_0(x)\right\rangle
	\end{equation} 
	is called the Levi form associated with $\omega_0$,
	where $U$, $V\in T^{1,0}_xX$, $\mathcal U$ and $\mathcal V$ are smooth sections on $T^{1,0}X$ such that $\mathcal U(x)=U$, $\mathcal V(x)=V$. Note that the Levi form $\mathcal L_x$ does not depends on the choices of $\mathcal U$ and $\mathcal V$.
\end{defin}

\noindent Let $x\in X$. Locally there is an orthonormal basis ${\mathcal U_1,\cdots,\mathcal U_n}$ of $T^{1,0}X$ such that $\mathcal L_x$ is diagonal in this basis, $\mathcal L_x(\mathcal U_j,\mathcal U_l)=\delta_{j,l}\lambda_j(x)$, $j, l=1,\ldots,n$. The entries $\{\lambda_1(x),\cdots,\lambda_n(x)\}$ are called the eigenvalues of the Levi form at $x\in X$.

\begin{defin}\label{d-gue210607yyd}
	Let $X$ be an oriented CR manifold of real dimension $2n+1$ with $n\ge1$. $X$ is said to satisfy condition $Y(q)$, $0\le q\le n$, if the Levi form has at least either $\max(n+1-q, q+1)$ eigenvalues of the same sign or $\min(n+1-q, q+1)$ pairs of eigenvalues of the opposite signs at every point on $X$.
\end{defin}

Define by $T^{*0,q}X:=\Lambda^qT^{*0,1}X$ the bundle of $(0,q)$-forms on $X$, $0\le q\le n$. For $u\in T^{*0,r}X$, let $(u\wedge)^*: T^{*0,q+r}X\to T^{*0,q}X$ be the adjoint of $u\wedge: T^{*0,q}X\to T^{*0,q+r}X$ with respect to $\langle\,\cdot\,|\,\cdot\,\rangle$, that is
\[
\langle\,u\wedge w\,|\,v\,\rangle=\langle\,w\,|\,(u\wedge)^*v\,\rangle
\]
for all $w\in T^{*0,q}X$ and $v\in T^{*0,q+r}X$. We write $|u|^2=\langle\,u\,|\,u\,\rangle$.

Let $D\subset X$ be an open set. We let $\omz^{0,q}(D)$ and $\omz_c^{0,q}(D)$ denote the space of smooth sections of $T^{*0,q}X$ over $D$ and the 
subspace of $\omz^{0,q}(D)$ whose elements have compact support in $D$, respectively. Let 
$$\dbar_b:\omz^{0,q}(X)\to \omz^{0,q+1}(X)$$
be the tangential Cauchy-Riemann operator. We say that $u\in\cali{C}^\infty(X)$ is a CR function if $\dbar_b u=0$. Moreover, if $L$ is a complex line bundle over $X$, $L$ is called a CR line bundle if its transition functions are CR.

Now, we let $(L,h^L)$ be a CR complex line bundle over $X$, where $h^L$ is the Hermitian fiber metric on $L$ and its local weight is $\phi$. Namely, let $s$ be a local CR trivializing section of $L$ on $D$, then locally,
\begin{equation}\label{trivia}
|s(x)|^2_{h^L}=e^{-\phi(x)},\quad x\in D.
\end{equation}

For $k>0$, let $L^k$ be the $k$-th tensor power of the line bundle $L$ over X, then $h^L$ induces a Hermitian fiber metric $h^{L^k}$ on $L^k$ and $s^k$ is a local CR trivializing section of $L^k$. For $D\subset X$, let $\omz^{0,q}(D,L^k)$ denote the space of smooth sections of $T^{*0,q}X\otimes L^k$ over $D$ and $\omz_c^{0,q}(D,L^k)$ be the subspace of $\omz^{0,q}(D,L^k)$ whose elements have compact support in $D$. For $u\in\omz^{0,q}(D,L^k)$, we sometimes write $|u(x)|^2_{h^{L^k}}$ simply as $|u|^2$.

Let $$\dbar_{b,k} \colon \omz^{0,q}(X,L^k)\to \omz^{0,q+1}(X,L^k)$$
 be the tangential Cauchy-Riemann operator acting on forms with values in $L^k$ such that
\begin{equation}\label{dbarbk}
\dbar_{b,k}(s^ku)=s^k\dbar_b u,
\end{equation}
where $s$ is a local CR trivializing section of $L$ on an open set $D$ and $u\in\omz^{0,q}(D)$. Then the $q$-th $\dbar_{b,k}$-cohomology is given by
\begin{equation}
H_b^{q}(X,L^k):=\ke \dbar_{b,k}/\im \dbar_{b,k},\quad 0\le q\le n.
\end{equation}
We sometimes simply write $\dbar_b$ to denote $\dbar_{b,k}$. 

Let $dv_X(x)$ be the volume form on $X$ induced by the fixed Hermitian metric on $\C TX$, then we get natural global $L^2$ inner products $(\,\cdot\,\mid\,\cdot\,)$ and $(\,\cdot\mid\cdot\,)_{h^{L^k}}$ on $\omz^{0,q}(X)$ and $\omz^{0,q}(X,L^k)$ respectively. Define by $L^2_{(0,q)}(X)$ the completion of $\omz_c^{0,q}(X)$ with respect to $(\,\cdot\mid\cdot\,)$ and  $L^2_{(0,q)}(X,L^k)$ the one of $\omz_c^{0,q}(X,L^k)$ with respect to $(\,\cdot\mid\cdot\,)_{h^{L^k}}$. The associated norms on $L^2_{(0,q)}(X)$ and  $L^2_{(0,q)}(X,L^k)$ are denoted by $\|\cdot\|$ and $\|\cdot\|_{h^{L^k}}$, respectively.

We extend $\dbar_{b,k}$ to $L^2_{(0,q)}(X,L^k)$: 
\[\dbar_{b,k}: {\rm Dom\,}\dbar_{b,k}\subset L^2_{(0,q)}(X,L^k)\To L^2_{(0,q+1)}(X,L^k),\]
where ${\rm Dom\,}\dbar_{b,k}=\set{u\in L^2_{(0,q)}(X,L^k)\mid \dbar_{b,k}u\in L^2_{(0,q+1)}(X,L^k)}$. 
Let 
$$\dbarstar_{b,k}:{\rm Dom\,}\dbarstar_{b,k}\subset L^2_{(0,q+1)}(X,L^k)\to L^2_{(0,q)}(X,L^k)$$
be the Hilbert space adjoint of $\dbar_{b,k}$ with respect to $(\,\cdot\mid\cdot\,)_{h^{L^k}}$. The (Gaffney extension) of Kohn Laplacian with values in $L^k$ is then given by
\begin{equation}\label{e-gue210228yyd}
\begin{split}
&\Box^q_{b,k}: {\rm Dom\,}\Box^q_{b,k}\subset L^2_{(0,q)}(X,L^k)\To L^2_{(0,q)}(X,L^k),\\
& {\rm Dom\,}\Box^q_{b,k}=\set{u\in{\rm Dom\,}\dbar_{b,k}\bigcap{\rm Dom\,}\dbarstar_{b,k} \mid \dbar_{b,k}u\in{\rm Dom\,}\dbarstar_{b,k}, \dbarstar_{b,k} u\in{\rm Dom\,}\dbar_{b,k}},\\
&\Box^q_{b,k}=\dbarstar_{b,k}\dbar_{b,k}+\dbar_{b,k}\dbarstar_{b,k}\ \ \mbox{on ${\rm Dom\,}\Box^q_{b,k}$}. 
\end{split}
\end{equation}

Let $s\in \N_0$, denote by $\|\cdot\|_s$ the standard Sobolev norm for sections of $L^k$ with respect to $(\,\cdot\mid\cdot\,)_{h^{L^k}}$ of order $s$. The following subelliptic estimates follow from Kohn's $L^2$ estimates\cite[Theorem 8.4.2]{K65}:

\begin{thm}\label{kohn}
Let $s\in\mathbb N_0$. There exists a constant $C_{s,k}>0$ such that 
	\begin{equation}\label{L2}
	\|u\|_{s+1}\le C_{s,k}\left(\|\Box_{b,k}^q u\|_s+\|u\|+\|T^{s+1}u\|\right),\ \ \mbox{for every $u\in\omz^{0,q}(X,L^k)$}. 
	\end{equation}
\end{thm} 

\begin{thm}\label{t-gue210228yyd}
Assume that $Y(q)$ holds on $X$. Let $s\in\mathbb N_0$. There exists a constant $C_{s,k}>0$ such that 
	\begin{equation}\label{e-gue210228yydI}
	\|u\|_{s+1}\le C_{s,k}\left(\|\Box_{b,k}^q u\|_s+\|u\|\right),\ \ \mbox{for every $u\in\omz^{0,q}(X,L^k)$}. 
	\end{equation}
\end{thm} 

From \eqref{e-gue210228yydI}, we see that  if $Y(q)$ holds, then $\Box_{b,k}^q$ is hypoelliptic, has compact resolvent and the strong Hodge decomposition holds. 
We refer the reader to \cite[Theorem 7.6]{K65}, \cite[Theorem 5.4.11-12]{FK72}, \cite[Corollary 8.4.7-8 ]{ChenShaw99} for relevant material. \\
Write 
\begin{equation}
\mathscr{H}_b^q(X,L^k)=\ke \Box_{b,k}^q.
\end{equation}
Then we have
\begin{equation}\label{e-gue210607yyd}
\dim\mathscr{H}_b^q(X,L^k)<\infty,\quad \mathscr{H}_b^q(X,L^k)\cong H_b^q(X,L^k)
\end{equation}
 if $Y(q)$ holds. 
 
 Since $\Box^q_{b,k}$ is non-negative and self-adjoint, the heat operator 
$e^{-t\Box^q_{b,k}}: L^2_{(0,q)}(X,L^k)\To{\rm Dom\,}\Box^q_{b,k}$, $t>0$, exists. Let 
\[A_k(t,x,y)=e^{-t\Box^q_{b,k}}(x,y)\in\mathscr D'(\mathbb R_+\times X\times X, (T^{*0,q}X\otimes L^k)\boxtimes(T^{*0,q}X\otimes L^k)^*)\]
 be the distribution kernel of $e^{-t\Box^q_{b,k}}$ with respect to $dv_X(x)$. 
 We also write $A_k(t)$ to denote $e^{-t\Box^q_{b,k}}$. 
 If $Y(q)$ holds, then 
 \[A_k(t,x,y)\in\cali{C}^\infty(\mathbb R_+\times X\times X, (T^{*0,q}X\otimes L^k)\boxtimes(T^{*0,q}X\otimes L^k)^*).\]
 One of the main goal of this work is to study the asymptotic behavior of 
 $A_k(\frac{t}{k},x,x)$ as $k\To+\infty$ when $Y(q)$ holds. Note that $A_k(\frac{t}{k})$ satisfies the following 
 \begin{equation}\label{e-gue210531ycda}
\begin{cases}
\left(\frac{\pa}{\pa t}+\frac{1}{k}\Box^q_{b,k}\right)A_k(\frac{t}{k})=0,\\
\lim_{t\To0}A_k(\frac{t}{k})=I\ \ \mbox{on $L^2_{(0,q)}(X,L^k)$}.
\end{cases}
\end{equation}

Now, assume that $Y(q)$ holds. We use the canonical identification ${\rm End\,}(L^k)=\mathbb C$, especially 
\begin{equation}\label{e-gue210303yyd}
A_k(\frac{t}{k},x,y)\in\cali{C}^\infty(\mathbb R_+\times X\times X, T^{*0,q}X\boxtimes(T^{*0,q}X)^*).
\end{equation}
Let $s$ be a local CR trivializing section of $L$ defined on an open set $D$ of $X$, $\abs{s}^2_{h^L}=e^{-\phi}$. We have the unitary identification: 
\[\begin{split}
U: L^2_{(0,q)}(D,L^k)&\To L^2_{(0,q)}(D),\\
s^k\otimes \hat u&\To e^{-\frac{k\phi}{2}} \hat u.
\end{split}\]
There exists $A_{k,s}(t,x,y)\in\cali{C}^\infty(\mathbb R_+\times D\times D, T^{*0,q}X\boxtimes(T^{*0,q}X)^*)$ such that 
\begin{equation}\label{e-gue210303yydI}
(e^{-t\Box^q_{b,k}}u)(x)=s^k(x)\otimes e^{\frac{k\phi(x)}{2}}\int_DA_{k,s}(t,x,y)e^{-\frac{k\phi(y)}{2}}\hat u(y)dv_X(y)\ \ \mbox{on $D$}, 
\end{equation}
for every $u=s^k\otimes\hat u\in\Omega^{0,q}_c(D,L^k)$, $\hat u\in\Omega^{0,q}_c(D)$. Note that 
\begin{equation}\label{e-gue210303yydII}
A_k(t,x,y)=s^k(x)\otimes A_{k,s}(t,x,y)e^{\frac{k\phi(x)-k\phi(y)}{2}}\otimes (s^k(y))^*\ \ \mbox{on $D$}. 
\end{equation}
In particular, 
\begin{equation}\label{e-gue210303yydIII}
\mbox{$e^{-t\Box^q_{b,k}}(x,x)=A_k(t,x,x)=A_{k,s}(t,x,x)$ on $D$}.
\end{equation}
We also notice that 
\begin{equation}\label{e-gue210303yydIV}
\mbox{$\mathop{\lim}\limits_{t\To0}\int A_{k,s}(t,x,y)\hat u(y)dv_X(y)=\hat u(x)$, for every $\hat u\in\Omega^{0,q}_c(D)$}. 
\end{equation}

\subsection{CR manifolds with $\mathbb R$-action}\label{s-gue210530yyd}

Now assume that $X$ admits a $\mathbb R$-action $\eta$, $\eta\in\mathbb R$: $\eta: X\to X$, $x\mapsto\eta\circ x$.  Let $T\in\cali{C}^\infty(X, TX)$ be the infinitesimal generator of   the $\mathbb R$-action which  is given by
\begin{equation}\label{e-gue150808}
(Tu)(x)=\frac{\partial}{\partial \eta}\left(u(\eta\circ x)\right)|_{\eta=0},\ \ u\in\cali{C}^\infty(X).
\end{equation}

\begin{defin}\label{d-gue210531yyd}
We say that the $\mathbb R$-action $\eta$ is CR if
$$[T, \cali{C}^\infty(X, T^{1,0}X)]\subset\cali{C}^\infty(X, T^{1,0}X)$$ and the $\mathbb R$-action is transversal if for each $x\in X$,
$$\mathbb C T(x)\oplus T_x^{1,0}(X)\oplus T_x^{0,1}X=\mathbb CT_xX.$$ 
\end{defin}

Assume that $(X, T^{1,0}X)$ is a CR manifold of dimension $2n+1$, $n\geq 1$, with a transversal CR $\mathbb R$-action $\eta$ and we let $T$ be the infinitesimal generator of the $\mathbb R$-action. Let $\omega_0\in\cali{C}^\infty(X,T^*X)$ be the global real one form determined by
\begin{equation}\label{e-gue150808I}
\begin{split}
&\langle\,\omega_0\,,\,u\,\rangle=0,\ \ \forall u\in T^{1,0}X\oplus T^{0,1}X,\\
&\langle\,\omega_0\,,\,T\,\rangle=-1.
\end{split}
\end{equation}
Suppose that $X$ admits a $\mathbb R$-invariant complete Hermitain metric $\langle\,\cdot\mid\cdot\,\rangle$ on $\mathbb CTX$ so that we have the orthogonal decomposition \[
\C TX=T^{1,0}X\oplus T^{0,1}X\oplus\{\lambda T\mid \lambda\in\C\}
\]
and $|T|^2=\langle\,T\mid T\,\rangle=1$.  

For $u\in\Omega^{0,q}(X)$, we define
\begin{equation}\label{e-gue150508faIIq}
Tu:=\frac{\partial}{\partial\eta}\bigr(\eta^*u\bigr)|_{\eta=0}\in\Omega^{0,q}(X),
\end{equation}
where $\eta^*: T^{*0,q}_{\eta\circ x}X\To T^{*0,q}_{x}X$ is the pull-back map of $\eta$. 
Since the $\mathbb R$-action is CR, as in the $S^1$-action case (see Section 2.4 in~\cite{HLM}), we have
\[T\dbar_b=\dbar_bT\ \ \mbox{on $\Omega^{0,q}(X)$}.\]

Let $(L,h^L)$ be a rigid CR line bundle over $X$ with a $\mathbb R$-invariant Hermitian metric $h^L$ on $L$. We refer the reader to Section 2.3 in~\cite{HHL20} for the definitions and terminology about rigid CR vector bundles. We recall the following (see Definition 2.9 in~\cite{HHL20})

\begin{defin}\label{Def:RigidCVB}
	A CR vector bundle $E$ of rank \(r\) over \(X\) with a CR bundle lift \(T^E\) of \(T\)  is called rigid CR (with respect to \(T^E\)) if for every point \(p\in X\) there exists an open neighborhood \(U\) around \(p\) and a CR frame \(\{f_1,\ldots,f_r\}\) of \(E|_U\) with \(T^E(f_j)=0\) for \(1\leq j\leq r\).
\end{defin}

A section $s\in \cali{C}^\infty(X, E)$ is called a rigid CR section if $T^E s=0$ and $\overline\partial_b s=0$. The frame $\{f_j\}_{j=1}^r$ in Definition~\ref{Def:RigidCVB}  is called a rigid CR frame of $E|_U$.
Note that any rigid CR vector bundle is locally CR trivializable. 

For every $k\in\mathbb N$, we use $T$ to denote the CR lifting $T^{L^k}$. For $u\in\Omega^{0,q}(X,L^k)$, we can define $Tu$ in the standard way (see Section 2.3 in~\cite{HHL20} ) and we have $Tu\in\Omega^{0,q}(X,L^k)$.  

Consider the operator
\[-iT: \Omega^{0,q}(X,L^k)\To\Omega^{0,q}(X,L^k)\]
and we extend $-iT$ to $L^2_{(0,q)}(X,L^k)$ space by
\begin{equation}\label{e-gue210608yyd}
\begin{split}
&-iT: {\rm Dom\,}(-iT)\subset L^2_{(0,q)}(X,L^k)\To L^2_{(0,q)}(X,L^k),\\
&{\rm Dom\,}(-iT)=\set{u\in L^2_{(0,q)}(X,L^k);\, -iTu\in L^2_{(0,q)}(X,L^k)}.
\end{split}\end{equation}
It is known that~\cite[Lemma 4.3]{HMW} $-iT$ is self-adjoint. 
  
Fix $\delta>0$, let 
\begin{equation}\label{e-gue210531yydb}
\begin{split}
&L^2_{(0,q),\leq k\delta}(X,L^k):=E_{-iT}([-k\delta,k\delta]),\\
&\Omega^{0,q}_{\leq k\delta}(X,L^k):=\Omega^{0,q}(X,L^k)\bigcap L^2_{(0,q),\leq k\delta}(X,L^k),
\end{split}
\end{equation}
where $E_{-iT}$ denotes the spectral measure of $-iT$. Let 
\begin{equation}\label{e-gue210531yydc}
Q_{X,\leq k\delta}: L^2_{(0,q)}(X,L^k)\To L^2_{(0,q),\leq k\delta}(X,L^k)
\end{equation}
be the orthogonal projection with respect  to $(\,\cdot\mid\cdot\,)_{h^{L^k}}$. We have the following partial $\dbar_{b,k}$-complex: 
$$
\cdots\longrightarrow\,\omz^{0,q-1}_{\leq k\delta}(X, L^k)\,\mathop{\longrightarrow}\limits^{\dbar_{b,k}}\omz^{0,q}_{\leq k\delta}(X, L^k) \,\mathop{\longrightarrow}\limits^{\dbar_{b,k}}\,\omz^{0,q+1}_{\leq k\delta}(X, L^k)\,\longrightarrow\cdots
$$
and  put 
\begin{equation}\label{e-gue210531yydd}
H^q_{b,\leq k\delta}(X,L^k):=\dfrac{\ke \dbar_{b,k}:\omz^{0,q}_{\leq k\delta}(X,L^k)\to \omz^{0,q+1}_{\leq k\delta}(X,L^k)}{\im \dbar_{b,k}: \omz^{0,q-1}_{\leq k\delta}(X,L^k)\to \omz^{0,q}_{\leq k\delta}(X,L^k)},\quad 0\le q\le n.
\end{equation}
Let 
\begin{equation}\label{e-gue210531yyde}
\begin{split}
&\Box^q_{b,k,\leq k\delta}: {\rm Dom\,}\Box^q_{b,k,\leq k\delta}\subset L^2_{(0,q),\leq k\delta}(X,L^k)\To L^2_{(0,q),\leq k\delta}(X,L^k),\\
& {\rm Dom\,}\Box^q_{b,k,\leq k\delta}={\rm Dom\,}\Box^q_{b,k}\bigcap L^2_{(0,q),\leq k\delta}(X,L^k),\\
&\Box^q_{b,k,\leq k\delta}=\dbarstar_{b,k}\dbar_{b,k}+\dbar_{b,k}\dbarstar_{b,k}\ \ \mbox{on ${\rm Dom\,}\Box^q_{b,k,\leq k\delta}$}. 
\end{split}
\end{equation}
It is not difficult to check that without $Y(q)$ condition, 
\begin{equation}\label{e-gue210531yydf}
H^q_{b,\leq k\delta}(X,L^k)\cong{\rm Ker\,}\Box^q_{b,k,\leq k\delta},\ \ {\rm dim\,}H^q_{b,\leq k\delta}(X,L^k)<+\infty.
\end{equation}

Let
\begin{equation}\label{e-gue210531yydg}
e^{-t\Box^q_{b,k,\leq k\delta}}:=e^{-t\Box^q_{b,k}}\circ Q_{X,\leq k\delta}: L^2_{(0,q)}(X,L^k)\To{\rm Dom\,}\Box^q_{b,k,\leq k\delta},\ \ t>0,
\end{equation} 
and let 
\[A_{k,\delta}(t,x,y)=e^{-t\Box^q_{b,k,\le k\delta}}(x,y)\in\cali{C}^\infty(\mathbb R_+\times X\times X, (T^{*0,q}X\otimes L^k)\boxtimes(T^{*0,q}X\otimes L^k)^*)\]
 be the distribution kernel of $e^{-t\Box^q_{b,k,\leq k\delta}}$ with respect to $dv_X(x)$. We also write $A_{k,\delta}(t)$ to denote $e^{-t\Box^q_{b,k,\leq k\delta}}$. One of the main goal of this work is to study the asymptotic behavior of 
 $A_{k,\delta}(\frac{t}{k},x,x)$ as $k\To+\infty$. Note that $A_{k,\delta}(\frac{t}{k})$ satisfies the following 
 \begin{equation}
\begin{cases}
\left(\frac{\pa}{\pa t}+\frac{1}{k}\Box^q_{b,k}\right)A_{k,\delta}(\frac{t}{k})=0,\\
\lim_{t\To0}A_{k,\delta}(t)=Q_{X,\leq k\delta}\ \ \mbox{on $L^2_{(0,q)}(X,L^k)$}.
\end{cases}
\end{equation}

As before, we use the canonical identification ${\rm End\,}(L^k)=\mathbb C$, especially 
\begin{equation}\label{e-gue210303yydab}
A_{k,\delta}(\frac{t}{k},x,y)\in\cali{C}^\infty(\mathbb R_+\times X\times X, T^{*0,q}X\boxtimes(T^{*0,q}X)^*).
\end{equation}
Let $s$ be a local CR rigid trivializing section of $L$ defined on an open set $D$ of $X$. We have the unitary identification: 
\[\begin{split}
U: L^2_{(0,q)}(D,L^k)&\To L^2_{(0,q)}(D),\\
s^k\otimes \hat u&\To e^{-\frac{k\phi}{2}} \hat u.
\end{split}\]
There exists $A_{k,s,\delta}(t,x,y)\in\cali{C}^\infty(\mathbb R_+\times D\times D, T^{*0,q}X\boxtimes(T^{*0,q}X)^*)$ such that 
\begin{equation}\label{e-gue210303yydIa}
(e^{-t\Box^q_{b,k,\leq k\delta}}u)(x)=s^k(x)\otimes e^{\frac{k\phi(x)}{2}}\int_DA_{k,s,\delta}(t,x,y)e^{-\frac{k\phi(y)}{2}}\hat u(y)dv_X(y)\ \ \mbox{on $D$}, 
\end{equation}
for every $u=s^k\otimes\hat u\in\Omega^{0,q}_c(D,L^k)$, $\hat u\in\Omega^{0,q}_c(D)$. Note that 
\begin{equation}\label{e-gue210303yydIIa}
A_{k,\delta}(t,x,y)=s^k(x)\otimes A_{k,s,\delta}(t,x,y)e^{\frac{k\phi(x)-k\phi(y)}{2}}\otimes (s^k(y))^*\ \ \mbox{on $D$}. 
\end{equation}
In particular, 
\begin{equation}\label{e-gue210303yydIIIa}
\mbox{$e^{-t\Box^q_{b,k,\leq k\delta}}(x,x)=A_{k,\delta}(t,x,x)=A_{k,s,\delta}(t,x,x)$ on $D$}.
\end{equation}

\section{The estimate of the heat kernel distribution}\label{estimate}

We shall first set up the local coordinates. For a given point $p\in X$, let $U_1,\cdots,U_n$ be a smooth orthonormal frame for $(1,0)$ vector fields in a neighborhood of $p$,
 for which $\mathcal L_p(U_j,\ol U_l)=\delta_{j,l}\lambda_j(p)$, $1\le j,l\le n$. Let $\omega^1,\cdots,\omega^n$ be an orthonormal basis for $(1,0)$ forms which is dual to the basis $U_1,\cdots,U_n$. Let $s$ be a local CR trivializing section of $L$ on an open neighborhood  $D$ of $p$ with local weight $\phi$, i.e. $|s(x)|^2_{h^L}=e^{-\phi(x)}$ on $D$.
 
We take local coordinates $x=(x_1,\ldots,x_{2n+1})=(z,\theta)=(z_1,\cdots,z_n,\theta)$ on an open set $D$ of $p$, such that 
\begin{equation}\label{local1}
 (z(p),\theta(p))=0,\quad \omega_0(p)=d\theta,\quad z_j=x_{2j-1}+ix_{2j},\quad \frac{\pa}{\pa z_j}=\frac{1}{2}\left(\frac{\pa}{\pa x_{2j-1}}-i\frac{\pa}{\pa x_{2j}}\right),
\end{equation}
\begin{equation}\label{local2}
\begin{aligned}
&\langle\,\frac{\partial}{\partial x_j}\,\mid\,\frac{\partial}{\partial x_l}\,\rangle=2\delta_{j,l}+O(|x|), \quad
\langle\,\frac{\partial}{\partial x_j}\,\mid\,\frac{\partial}{\partial\theta}\,\rangle=O(|x|),\\
 &\left|\frac{\partial}{\partial \theta}\right|^2=1+O(|x|),\quad  \left|\frac{\pa}{\pa z_j}\right|^2=1+O(|x|),
 \end{aligned}
\end{equation}
for $j,l=1,\cdots,2n$.
 Moreover,
\begin{equation}\label{U}
   U_j=\frac{\pa}{\pa z_j}-i\lambda_j\ol z_j\frac{\pa}{\pa \theta}-c_j\theta\frac{\pa}{\pa \theta}+\sum^{2n+1}_{s=1}r_{j,s}(z,\theta)\frac{\partial}{\partial x_s},
\end{equation}
where $r_{j,s}(z,\theta)\in\cali{C}^\infty(D)$, $r_{j,s}(z,\theta)=O(|(z,\theta)|^2)$, $j=1,\ldots,n$, $s=1,\ldots,2n+1$, 
and
\begin{equation}\label{phi}
\begin{aligned}
\phi(z,\theta)=&\beta\theta+\sum_{j=1}^n(a_jz_j+\ol a_j\ol z_j)+\sum_{j,l=1}^{n}(a_{j,l}z_jz_l+\ol a_{j,l}\ol z_j\ol z_l)+\sum_{j,l=1}^{n}\mu_{j,l}z_j\ol z_l\\
&+O(|\theta|^2+|z||\theta|+|(z,\theta)|^3)
\end{aligned}
\end{equation}
where $\beta\in\R$, $c_j$, $a_j$, $a_{j,l}$, $\mu_{j,l}\in\C$, and $\lambda_j$ are the eigenvalues of $\mathcal L_p$, $j, l=1,\ldots,n$. We then work with this local coordinates and we identify $D$ with some open set in $\R^{2n+1}$.

\subsection{The  unitary identification}
\

\vspace{0.25cm}
For $u,v\in \omz_c^{0,q}(D)$, let
$$
(u\mid v)_{k\phi,D}=(u\mid v)_{k\phi}:=\int_{D}\langle\,u\mid v\,\rangle e^{-k\phi}dv_X(x).
$$
$(\,\cdot\mid\cdot\,)_{k\phi}$ is called the inner product on $\omz_c^{0,q}(D)$ with weight $k\phi$. Denote by $L_{(0,q)}^2(D,k\phi)$ the completion of $\omz_c^{0,q}(D)$ with respect to $(\,\cdot\mid\cdot\,)_{k\phi}$.
Let $\dbar_b^{*,k\phi}$ be the formal adjoint of $\dbar_b$ with respect to $(\,\cdot\mid\cdot\,)_{k\phi}$. From $\dbar_b=\sum_{j=1}^{n}\bigr(\ol\omega^j\wedge \ol U_j+(\dbar_b \ol\omega^j)\wedge (\ol\omega^j\wedge)^*\bigr)$, we can check that
\begin{equation}
\dbar^{*,k\phi}_{b}=\sum_{j=1}^{n}\bigr((\ol\omega^j\wedge)^*(- U_j+kU_j\phi+\beta_j(z,\theta))+\ol\omega^j\wedge (\dbar_b \ol\omega^j\wedge)^*\bigr),
\end{equation}
where $\beta_j(z,\theta)$ is a smooth function on $D$, independent of $k$, for every $j=1,\ldots,n$. Set
\begin{equation}\label{e-gue210531yydu}
\Box_{b,k\phi}^q=\dbar^{*,k\phi}_{b}\dbar_{b}+\dbar_{b}\dbar^{*,k\phi}_{b}: \Omega^{0,q}(D)\To\Omega^{0,q}(D).
\end{equation}
For $u\in\omz^{0,q}(D,L^k)$, there exists a $\check u\in \omz^{0,q}(D)$ such that 
$u=s^k\check u$ on $D$. Recall that $\dbar_{b,k}(s^k \check u)=s^k\dbar_b \check u$ and we have
\begin{equation}\label{e-gue210303yyda}
\Box_{b,k}^qu=s^k\Box_{b,k\phi}^q\check u.
\end{equation}
From now on, on $D$, we identify $u$ with $\check u$ and $\Box_{b,k}^q$ with $\Box_{b,k\phi}^q$.

Set
\begin{equation}\label{rho}
\rho(z)=\sum_{j=1}^na_jz_j+\sum_{j,l=1}^{n}a_{j,l}z_jz_l
\end{equation}
and 
\begin{equation}\label{phi0}
\begin{aligned}
\phi_0(z,\theta)&=\phi(z,\theta)-\rho(z)-\ol{\rho(z)}\\
&=\beta\theta+\sum_{j,l=1}^{n}\mu_{j,l}z_j\ol z_l+O(|\theta|^2+|z||\theta|+|(z,\theta)|^3).
\end{aligned}
\end{equation}
We consider the following unitary identification:
$$
\begin{cases}
L^2_{(0,q)}(D,k\phi)&\leftrightarrow L^2_{(0,q)}(D,k\phi_0)\\
u&\to\tilde{u}=e^{-k\rho}u\\
u=e^{k\rho}\tilde{u}&\leftarrow\tilde{u}\\
\end{cases}
$$
For $u\in\omz^{0,q}(D)\cap L^2_{(0,q)}(D,k\phi_0)$,
let $\dbar_{\rho}=\dbar_b+k(\dbar_b\rho)\wedge$, then we have 
\begin{equation}\label{dbarrho}
\dbar_\rho\tilde u=e^{-k\rho}\dbar_{b}u=\widetilde{\dbar_bu}.
\end{equation}
Let $\dbar_\rho^{*,k\phi_0}$ be the formal adjoint of $\dbar_{\rho}$ with respect to $(\,\cdot\mid\cdot\,)_{k\phi_0}$. We can check that 
\begin{equation}
\begin{aligned}
\dbar_\rho=\sum_{j=1}^{n}\bigr(\ol\omega^j\wedge (\ol U_j+k\ol U_j\rho)+(\dbar_b \ol\omega^j)\wedge (\ol\omega^j\wedge)^*\bigr)
\end{aligned}
\end{equation}
and
\begin{equation}\label{dbarstarrho}
\begin{aligned}
\dbar_\rho^{*,k\phi_0}=\sum_{j=1}^{n}\bigr((\ol\omega^j\wedge)^*\left(-U_j+kU_j\phi_0+kU_j\ol\rho+\alpha_j(z,\theta)\right)+\ol\omega^j\wedge (\dbar_b \ol\omega^j\wedge)^*\bigr),
\end{aligned}
\end{equation}\label{dbarstarkphi0}
where $\alpha_j(z,\theta)$ is a smooth function on $D$, independent of $k$, for every $j=1,\ldots,n$. For abbreviation, we write $\ol U_{\rho,j}=\ol U_j+k\ol U_j\rho$ and $\ol U_{\rho,j}^{*,k\phi_0}=-U_j+kU_j\phi_0+kU_j\ol\rho+\alpha_j(z,\theta)$, $j=1,\ldots,n$. 
Put
$$
\Box^q_{\rho.k\phi_0}=\dbar^{*,k\phi_0}_{\rho}\dbar_{\rho}+\dbar_{\rho}\dbar^{*,k\phi_0}_{\rho}: \Omega^{0,q}(D)\To\Omega^{0,q}(D).
$$ 
Thorough a straightforward computation, we have
\begin{equation}
\begin{aligned}
\Box^q_{\rho.k\phi_0}=&\sum_{j=1}^{n}\ol U_{\rho,j}^{*,k\phi_0}\ol U_{\rho,j}
+\sum_{j,l=1}^{n}\ol\omega^j\wedge(\ol\omega^l\wedge)^*\big[\ol U_{\rho,j}\,,\,\ol U_{\rho,l}^{*,k\phi_0}\big] \\
&+\sum_{j=1}^n\Big(\epsilon_j(z,\theta)\ol U_{\rho,j}+s_j(z,\theta)\ol U_{\rho,j}^{*,k\phi_0}\Big)+\gamma(z,\theta),
\end{aligned}
\end{equation}
where $\epsilon_j$, $s_j$ and $\gamma$ are smooth functions on $D$ and independent of $k$, $j=1,\ldots,n$. We can check that
\begin{equation}\label{e-gue210516yyd}
\Box_{b,k\phi}^qu=e^{k\rho}\Box_{\rho, k\phi_0}^q(e^{-k\rho}u),\ \ u\in\Omega^{0,q}(D).
\end{equation} 

Until further notice, we assume that $Y(q)$ holds on $X$. Let 
\begin{equation}\label{e-gue210305yyd}
\begin{split}
&A_{k\phi}(\frac{t}{k},x,y):=e^{\frac{k\phi(x)}{2}}A_{k,s}(\frac{t}{k},x,y)e^{-\frac{k\phi(y)}{2}},\\
&A_{k\phi_0}(\frac{t}{k},x,y):=e^{-k\rho(x)}A_{k\phi}(\frac{t}{k},x,y)e^{k\rho(y)}=e^{-k\rho(x)+\frac{k\phi(x)}{2}}A_{k,s}(\frac{t}{k},x,y)e^{k\rho(y)-\frac{k\phi(y)}{2}},
\end{split}
\end{equation}
where $A_{k,s}(\frac{t}{k},x,y)$ is as in \eqref{e-gue210303yydII}. For $t>0$, let 
\[A_{k\phi_0}(t): \mathscr E'(D,T^{*0,q}X)\To\Omega^{0,q}(D) \]
be the continuous operator with distribution kernel $A_{k\phi_0}(t,x,y)$ with respect to $dv_X$. Note that 
\begin{equation}\label{e-gue210305yydI}
A_{k\phi_0}(t)u=\int A_{k\phi_0}(t,x,y)u(y)dv_X(y),\ \ u\in\Omega^{0,q}_c(D).
\end{equation}
It is easy to check that 
\begin{equation}\label{e-gue210305yydII}
\begin{cases}
&A'_{k\phi_0}(t)+\Box^q_{\rho,k\phi_0}A_{k\phi_0}(t)=0,\\
&\lim_{t\To0^+}A_{k\phi_0}(t)=I.
\end{cases}
\end{equation}


Notice that the reason we take the term $\rho(z)$ out of $\phi(z,\theta)$ is that  the term $kU_j\ol\rho$ in it will approach to zero as $k\To+\infty$ after $\Box_{\rho,k\phi_0}^q$ been scaled. Now let us introduce the scaling technique.

\subsection{The scaling technique} 
\

\vspace{0.25cm}

Let $B_r:=\{(z,\theta)=(x_1,\cdots x_{2n},\theta)\in\R^{2n+1}\mid |x_j|<r,|\theta|<r,j=1,\cdots 2n\}$. Let 
\[\begin{split}
F_k: \mathbb R^{2n+1}&\To\mathbb R^{2n+1},\\
(z,\theta)&\To(\frac{z}{\sqrt{k}},\frac{\theta}{k}),
\end{split}\]
 be the scaling map, then we can choose sufficient large $k$ such that $F_k(B_{\log k})\subset D$. For $(z,\theta)\in B_{\log k}$, let 
\begin{equation}
F_k^*T^{*0,q}_{z,\theta}X:=\left\{\sumprime_{|J|=q}a_J\ol\omega^J\left(\frac{z}{\sqrt k},\frac{\theta}{k}\right)\,\bigg|\, a_J\in\C\right\},
\end{equation}
where the sum takes over all $q$-tuple of integers between $1$ and $n$, the prime means that we sum over only strictly increasing multiindices, 
$\ol\omega^J=\ol\omega^{j_1}\wedge\cdots\wedge\ol\omega^{j_q}$, $1\leq j_1<\cdots<j_q\leq n$.
Let $F^*_kT^{*0,q}X$ be the vector bundle over $B_{\log k}$ with fiber $F^*_kT^{*0,q}_{x}X$, $x=(z,\theta)\in B_{\log k}$.
Define by $F_k^*\omz^{0,q}(B_r)$ the space of smooth sections of $F_k^*T^{*0,q}X$ over $B_r$. Let 
\[F^*_k\Omega^{0,q}_c(B_r)=\set{u\in F^*_k\Omega^{0,q}(B_r)\mid {\rm supp\,}u\Subset B_r}.\]
For $u=\sumprime_{|J|=q}u_J\ol\omega^J\in \omz^{0,q}(F_k(B_{\log k}))$, we define the scaled form $F_k^*u$ by
\begin{equation}
F_k^*u:=\sumprime_{|J|=q}u_J\left(\frac{z}{\sqrt k},\frac{\theta}{k}\right)\ol\omega^J\left(\frac{z}{\sqrt k},\frac{\theta}{k}\right)\in F_k^*\omz^{0,q}(B_{\log k}).
\end{equation}

Let $\dbar_{\rho,(k)}:F_k^*\omz^{0,q}(B_{\log k})\to F_k^*\omz^{0,q+1}(B_{\log k})$ be the scaled differential operator, which is given by
\begin{equation}\label{dbarrhok}
\begin{aligned}
\dbar_{\rho,(k)}&=\sum_{j=1}^{n}\ol\omega^j\left(\frac{z}{\sqrt k},\frac{\theta}{k}\right)\wedge \left(\ol U_{j,(k)}+\sqrt kF_k^*(\ol U_j\rho)\right)\\
&+\frac{1}{\sqrt k}(\dbar_b \ol\omega^j)\left(\frac{z}{\sqrt k},\frac{\theta}{k}\right)\wedge \left(\ol\omega^j\left(\frac{z}{\sqrt k},\frac{\theta}{k}\right)\wedge\right)^*,
\end{aligned}
\end{equation}
where 
\begin{equation}\label{staru1}
\begin{split}
\ol U_{j,(k)}=&\frac{\pa}{\pa \ol z_j}+i\lambda_j z_j\frac{\pa}{\pa \theta}-\frac{1}{\sqrt k}\bar c_j\theta\frac{\pa}{\pa \theta}\\
&\quad+\sum^{2n}_{s=1}r_{j,s}(\frac{z}{\sqrt{k}},\frac{\theta}{k})\frac{\partial}{\partial x_s}+\sqrt{k}r_{j,2n+1}(\frac{z}{\sqrt{k}},\frac{\theta}{k})\frac{\partial}{\partial\theta},
\end{split}
\end{equation}
where $r_{j,s}(z,\theta)\in\cali{C}^\infty(D)$, $j=1,\ldots,n$, $s=1,\ldots,2n+1$, are as in \eqref{U}. For every $j=1,\ldots,n$, put 
\begin{equation}\label{e-gue210305ycd}
\begin{split}
&\varepsilon_{j,(k)}:=\sum^{2n}_{s=1}r_{j,s}(\frac{z}{\sqrt{k}},\frac{\theta}{k})\frac{\partial}{\partial x_s}+\sqrt{k}r_{j,2n+1}(\frac{z}{\sqrt{k}},\frac{\theta}{k})\frac{\partial}{\partial\theta},\\
&\overline\varepsilon_{j,(k)}:=\sum^{2n}_{s=1}\overline{r_{j,s}}(\frac{z}{\sqrt{k}},\frac{\theta}{k})\frac{\partial}{\partial x_s}+\sqrt{k}\overline{r_{j,2n+1}}(\frac{z}{\sqrt{k}},\frac{\theta}{k})\frac{\partial}{\partial\theta}.
\end{split}
\end{equation}
Note that 
\begin{equation}\label{staru2}
\sqrt kF_k^*(\ol U_j\rho)=\overline\varepsilon_{j,(k)}\left(\sum_{j=1}^na_jz_j+\frac{1}{\sqrt k}\sum_{j,l=1}^{n}a_{j,l}z_jz_l\right).
\end{equation}
Compare \eqref{dbarrho} and \eqref{dbarrhok}, one can check that
\begin{equation}\label{k1}
\dbar_{\rho,(k)}(F_k^*u)=\frac{1}{\sqrt k}F_k^*(\dbar_{\rho}u),\quad u\in\omz^{0,q}(F_k(B_{\log k})).
\end{equation} 

Fix $q\in\set{0,1,\ldots,n}$. With the notations used in the beginning of Section~\ref{estimate}, 
\[\set{\overline\omega^J(x)\mid J=(j_1,\ldots,j_q), 1\leq j_1<\cdots<j_q\leq n}\] 
is an orthonormal frame for $T^{*0,q}_xX$, for every $x\in D$. Let $\langle\,\cdot\mid\cdot\,\rangle_{F^*_k}$ be the Hermitian metric of $F^*_kT^{*0,q}X$ on $B_{\log k}$ such that $\set{\overline\omega^J(F_kx)\mid J=(j_1,\ldots,j_q, 1\leq j_1<\cdots<j_q\leq n}$ is an orthonormal frame at every $x\in B_{\log k}$. On $D$, write 
$dv_X(x)=m(x)dx_1\cdots dx_{2n+1}$, $m(x)\in\cali{C}^\infty(D)$. Let $(\,\cdot\mid\cdot\,)_{kF^*_k\phi_0}$ be the $L^2$ inner product on $F^*_k\Omega^{0,q}_c(B_{\log k})$ given by 
\[(\,u\mid v\,)_{kF^*_k\phi_0}=\int\langle\,u\mid v\,\rangle_{F^*_k}e^{-kF^*_k\phi_0}m(F_kx)dx_1\cdots dx_{2n+1},\ \ u, v\in F^*_k\Omega^{0,q}_c(B_{\log k}).\]

Let $\dbar_{\rho,(k)}^{*,kF_k^*\phi_0}$ be the formal adjoint of $\dbar_{\rho,(k)}$ with respect to $(\,\cdot\mid\cdot\,)_{kF_k^*\phi_0}$, then
\begin{equation}\label{dbarstarrhok}
\begin{aligned}
\dbar_{\rho,(k)}^{*,kF_k^*\phi_0}=&\sum_{j=1}^{n}\left(\ol\omega^j\left(\frac{z}{\sqrt k},\frac{\theta}{k}\right)\wedge\right)^*\left(-U_{j,(k)}+\sqrt kF_k^*(U_j\phi_0)+\sqrt kF_k^*(U_j\ol\rho)+\frac{1}{\sqrt k}F_k^*\alpha_j\right)\\
&+\frac{1}{\sqrt k}\ol\omega^j\left(\frac{z}{\sqrt k},\frac{\theta}{k}\right)\wedge \left((\dbar_b \ol\omega^j)\left(\frac{z}{\sqrt k},\frac{\theta}{k}\right)\wedge\right)^*,
\end{aligned}
\end{equation}
where $\alpha_j$ is a smooth function as in \eqref{dbarstarrho}. Compare with \eqref{dbarstarrho}, we also have
\begin{equation}\label{k2}
\dbar^{*,kF_k^*\phi_0}_{\rho,(k)}(F_k^*u)=\frac{1}{\sqrt k}F_k^*(\dbar^{*,k\phi_0}_{\rho}u),\quad u\in\Omega^{0,q+1}(F_k(B_{\log k})).
\end{equation}
Now we define the scaled Kohn-Laplacian:
\begin{equation}\label{e-gue210614yydI}
\Box^q_{\rho,(k)}=\dbar_{\rho,(k)}^{*,kF_k^*\phi_0}\dbar_{\rho,(k)}+\dbar_{\rho,(k)}\dbar_{\rho,(k)}^{*,kF_k^*\phi_0}
\end{equation}
on $F_k^*\Omega^{0,q}(B_{\log k})$. It follows from \eqref{k1} and \eqref{k2} immediately that
\begin{equation}\label{k(k)}
\Box_{\rho,(k)}^q(F_k^*u)=\frac{1}{k}F_k^*(\Box_{\rho,k\phi_0}^qu),\quad u\in\Omega^{0,q}(F_k(B_{\log k})). 
\end{equation} 

Let 
\begin{equation}\label{e-gue210325yyd}
\begin{split}
&A_{(k)}(t,x,y)\\
&:=k^{-(n+1)}A_{k\phi_0}(\frac{t}{k},F_kx,F_ky)\in\cali{C}^\infty(\mathbb R_+\times B_{\log k}\times B_{\log k}, F^*_kT^{*0,q}X\boxtimes(F^*_kT^{*0,q}X)^*),
\end{split}
\end{equation}
where $F_kx=(\frac{z}{\sqrt{k}},\frac{\theta}{k})$, $F_ky=(\frac{w}{\sqrt{k}},\frac{\xi}{k})$, $y=(w,\xi)\in\mathbb C^n\times\mathbb R$, $A_{k\phi_0}(t,x,y)$ is as in
\eqref{e-gue210305yyd}. Let 
\[A_{(k)}(t): F^*_k\Omega^{0,q}_c(B_{\log k})\To F^*_k\Omega^{0,q}(B_{\log k})\] 
be the continuous operator given by 
\begin{equation}\label{e-gue210325yydI}
(A_{(k)}(t)u)(x)=\int A_{(k)}(t,x,y)u(y)m(F_ky)dy,\ \ u\in F^*_k\Omega^{0,q}_c(B_{\log k}).
\end{equation}

Observe that \eqref{staru1} and \eqref{staru2} show us
\begin{equation}\label{staru3}
\ol U_{j,(k)}+\sqrt kF_k^*(\ol U_j\rho)=\frac{\pa}{\pa \ol z_j}+i\lambda_j z_j\frac{\pa}{\pa \theta}+\epsilon_kZ_{j,k},\ \ j=1,\ldots,n,
\end{equation}
on $B_{\log k}$, where $\epsilon_k$ is a sequence tending to zero as $k\to\infty$ and $Z_{j,k}$ is a first order differential operator and all the 
derivatives of the coefficients of $Z_{j,k}$ are uniformly bounded in $k$ on $B_{\log k}$, $j=1,\ldots,n$. 
In the same manner, from \eqref{phi0}, we see that 
\begin{equation}\label{staru4}
\begin{aligned}
&-U_{j,(k)}+\sqrt kF_k^*(U_j\phi_0)+\sqrt k(U_j\ol\rho)+\frac{1}{\sqrt k}F_k^*\alpha_j\\
&=-\frac{\pa}{\pa z_j}+i\lambda_j\ol z_j\frac{\pa}{\pa \theta}-i\lambda_j\ol z_j\beta+\sum_{l=1}^{n}\mu_{l,j}\bar z_l+\delta_kV_{j,k},\ \ j=1,\ldots,n, 
\end{aligned}
\end{equation}
on $B_{\log k}$, where $\delta_k$ is a sequence tending to zero as $k\to\infty$ and $V_{j,k}$ is a first order differential operator and all the 
derivatives of the coefficients of $V_{j,k}$ are uniformly bounded in $k$ on $B_{\log k}$, $j=1,\ldots,n$.  
The following proposition is fairly straightforward follows from \eqref{dbarrhok}, \eqref{dbarstarrhok}, \eqref{staru3} and \eqref{staru4}: 

\begin{prop}
We have that
\begin{equation}\label{boxrhok}
\begin{aligned}
\Box_{\rho,(k)}^{q}=&\sum_{j=1}^{n}\left(-\frac{\pa}{\pa z_j}+i\lambda_j\ol z_j\frac{\pa}{\pa \theta}-i\lambda_j\ol z_j\beta+\sum_{l=1}^{n}\mu_{l,j}\bar z_l\right)\left(\frac{\pa}{\pa \ol z_j}+i\lambda_j z_j\frac{\pa}{\pa \theta}\right)\\
&+\sum_{j,l=1}^{n}\ol\omega^j\left(\frac{z}{\sqrt k},\frac{\theta}{k}\right)\wedge \left( \ol\omega^l\left(\frac{z}{\sqrt k},\frac{\theta}{k}\right)\wedge\right)^*
\left(2i\lambda_j \delta_{j,l}\frac{\pa}{\pa \theta}+\mu_{j,l}-i\lambda_j \delta_{j,l}\beta\right)\\
&+\varepsilon_kP_k
\end{aligned}
\end{equation}
on $B_{\log k}$, where $\varepsilon_k$ is a sequence tending to zero as $k\to\infty$ and $P_k$ is a second order differential operator and all the 
derivatives of the coefficients of $P_k$ are uniformly bounded in $k$ on $B_{\log k}$.
\end{prop}

We pause and introduce some notations. 
For $s\in\mathbb N_0$ and $B\subset B_{\log k}$, let 
\[W^s_{kF_k^*\phi_0}(B, F_k^*T^{*0,q}X)\] 
be the Sobolev space of order $s$ define on the sections of $F_k^*T^{*0,q}X$ over $B$ with respect to $(\,\cdot\mid\cdot\,)_{kF_k^*\phi_0}$. Let $x=(z,\theta)\in B$, the Sobolev norm of $u=\sumprime_{|J|=q}u_J\ol\omega^J\left(\frac{z}{\sqrt k},\frac{\theta}{k}\right)\in W^s_{kF_k^*\phi_0}(B, F_k^*T^{*0,q}X)$ is given by
\begin{equation}
\begin{aligned}
\|u\|^2_{kF_k^*\phi_0,s,B}=\sum_{|\alpha|\le s}\sumprime_{|J|=q}\int_{B}|\pa_{x}^\alpha u_J|^2e^{-kF_k^*\phi_0}m(F_kx)dx_1\cdots dx_{2n+1}.
\end{aligned}
\end{equation}
Let us write $\|\cdot\|_{kF_k^*\phi_0,0,B}$ as $\|\cdot\|_{kF_k^*\phi_0,B}$ to simplify the notation. 
Let 
\[\begin{split}
&W^s_{{\rm c\,},kF^*_k\phi_0}=\set{u\in W^s_{kF_k^*\phi_0}(B, F_k^*T^{*0,q}X)\mid{\rm supp\,}u\Subset B},\\
&W^s_{{\rm loc\,},kF^*_k\phi_0}=\set{u\in\mathscr D'(B,F^*_kT^{*0,q}X)\mid\mbox{$\chi u\in W^s_{kF_k^*\phi_0}(B, F_k^*T^{*0,q}X)$, for every $\chi\in\cali{C}^\infty_c(B)$}}.
\end{split}\]
Write 
\[\begin{split}
&L^2_{{\rm c\,},kF^*_k\phi_0}(B, F_k^*T^{*0,q}X):=W^0_{{\rm c\,},kF^*_k\phi_0}(B, F_k^*T^{*0,q}X),\\ 
&L^2_{{\rm loc\,},kF^*_k\phi_0}(B, F_k^*T^{*0,q}X):=W^0_{{\rm loc\,},kF^*_k\phi_0}(B, F_k^*T^{*0,q}X).\end{split}\]

For $s\in\mathbb Z$, $s<0$, define $W^s_{{\rm c\,},kF^*_k\phi_0}(B, F_k^*T^{*0,q}X)$ and $W^s_{{\rm loc\,},kF^*_k\phi_0}(B, F_k^*T^{*0,q}X)$ as the dual spaces of $W^{-s}_{{\rm loc\,},kF^*_k\phi_0}(B, F_k^*T^{*0,q}X)$ and $W^{-s}_{{\rm c\,},kF^*_k\phi_0}(B, F_k^*T^{*0,q}X)$ with respect to $(\,\cdot\mid\cdot\,)_{kF^*_k\phi_0}$ respectively. For $s\in\mathbb Z$, $s<0$, we define $\|\cdot\|_{kF_k^*\phi_0,s,B}$ 
the Sobolev norm for the space $W^s_{{\rm c\,},kF^*_k\phi_0}(B, F_k^*T^{*0,q}X)$ in the standard way. Similarly, for $s\in\mathbb Z$, we define $W^{s}_{{\rm c\,},kF^*_k\phi_0}(B_{\log k}, F^*_kT^{*0,q}X)$, $W^{s}_{{\rm loc\,},kF^*_k\phi_0}(B_{\log k}, F^*_kT^{*0,q}X)$ in the same way. 
We also define $L^2_{{\rm c\,}}(F_kB_{\log k},T^{*0,q}X)$ in the similar way (with $\phi_0=0$). 

Kohn\cite{K65} proved that if $Y(q)$ holds, then $\Box_{\rho,(k)}^q$ is hypoelliptic with loss of one derivative. The following proposition is a direct consequence of Theorem \ref{kohn} and Kohn's poof. 

\begin{prop}\label{p-gue210511yyd}
Recall that we work with the assumption that $Y(q)$ holds on $X$. Let $s\in\mathbb N_0$. 
	For $r>0$ with $B_{2r}\subset B_{\log k}$, there exists a constant $C_{r,s}>0$ which is independent of $k$ such that for all $u\in F_k^*\omz^{0,q}(B_{\log k})$, we have
	\begin{equation}\label{subellipticest}
	\begin{aligned}
	\|u\|^2_{kF_k^*\phi_0,s+1,B_r}\le C_{r,s}\left(\|u\|^2_{kF_k^*\phi_0,B_{2r}}+\|\Box_{\rho,(k)}^qu\|^2_{kF_k^*\phi_0,s,B_{2r}}\right).
	\end{aligned}
	\end{equation}
\end{prop}

\begin{remark}\label{constant}
	Go over Kohn's proof, we see that the constant $C_{r,s}$ only depends on the derivatives of all the coefficients of $\Box_{\rho,(k)}^q$, $kF_k^*\phi_0$ and $m(F_kx)$ on 
	$B_{\log k}$. From \eqref{boxrhok}, it is clear that all the derivatives of the coefficients of $\Box_{\rho,(k)}^q$, $kF_k^*\phi_0$ and $m(F_kx)$ are uniformly bounded in $k$ in $\cali{C}^\infty$ topology on $B_{\log k}$. Hence $C_{r,s}$ can be taken to be indenpendent of $k$.
\end{remark}

Our next objective is to show that if $Y(q)$ holds, $A_{(k)}(t,x,y)$ is uniformly bounded in $k$ in $\cali{C}^\infty$-topology for some fixed $t>0$. For this purpose, we need to apply the spectral theorem\cite{Davies95} for self-adjoint operator. We reproduce it as follows for the reader’s convenience.

\begin{thm}[spectral theorem for self-adjoint operator]\label{spectralthm}
	Let $A: {\rm Dom\,}A\subset\mathbb H\To\mathbb H$ be a self-adjoint operator on a Hilbert space $\mathbb H$. Then there exists a finite regular Borel measure $\mu$ on $\sigma(A)\times\N$ and a unitary transformation 
	$$
	U:\mathbb H\to L^2(\sigma(A)\times\N, d\mu),
	$$
	$U$ is one to one, onto, such that
	$$\Dom A=\left\{U^{-1}f\mid f\in\Dom M_s\right\},$$ 
	where $\sigma(A)$ is the spectrum of $A$, $M_s\varphi(s,n)=s\varphi(s,n)$, $(s,n)\in\sigma(A)\times\mathbb N$, is the multiplication operator by $s$ on $L^2(\sigma(A)\times\N, d\mu)$, 
	\[{\rm Dom\,}M_s=\set{\varphi(s,n)\in L^2(\sigma(A)\times\N, d\mu)\mid s\varphi(s,n)\in L^2(\sigma(A)\times\N, d\mu)}.\]
	Furthermore, $UAU^{-1}=M_s$ on $U({\rm Dom\,}A)$. 
\end{thm}

Theorem \ref{spectralthm} says that when studying spectral property of a self-adjoint operator it suffices to study the multiplication operator $M_s: {\rm Dom\,}M_s\subset L^2(\sigma(A)\times\N, d\mu)\to L^2(\sigma(A)\times\N, d\mu)$. Roughly speaking, we can identify 
\begin{equation}
\begin{aligned}
\mathbb H\cong L^2(\sigma(A)\times\N, d\mu)
\end{aligned}
\end{equation}
and
\begin{equation}
\begin{aligned}
A\cong M_s: {\rm Dom\,}M_s\subset L^2(\sigma(A)\times\N, d\mu)&\to L^2(\sigma(A)\times\N, d\mu),\\
\varphi(s,n)&\mapsto s\varphi(s,n).
\end{aligned}
\end{equation} 

As before, let $\Box^q_{b,k}: {\rm Dom\,}\Box^q_{b,k}\subset L^2_{(0,q)}(X,L^k)\To L^2_{(0,q)}(X,L^k)$ be the Kohn Laplacian given by 
\eqref{e-gue210228yyd}. Since $\Box^q_{b,k}$ is self-adjoint, by Theorem~\ref{spectralthm}, we identify $L^2_{(0,q)}(X,L^k)$ with $L^2(\mathbb S\times\mathbb N,d\mu)$ and $\Box^q_{b,k}$ with $M_s$, where $\mathbb S$ denotes the spectrum of $\Box^q_{b,k}$. Then, 
\[\begin{split}
e^{-t\Box^q_{b,k}}=e^{-ts}: L^2(\mathbb S\times\mathbb N,d\mu)&\To L^2(\mathbb S\times\mathbb N,d\mu),\\
\varphi(s,n)&\To e^{-ts}\varphi(s,n).
\end{split}\]

\begin{lem}\label{l-gue210331yyd}
Fix $t>0$. For every $N\in\mathbb N_0$, we have 
\begin{equation}\label{e-gue210331yyd}
\norm{e^{-\frac{t}{k}\Box^q_{b,k}}(\Box^q_{b,k})^Nu}_{h^{L^k}}=\norm{(\Box^q_{b,k})^Ne^{-\frac{t}{k}\Box^q_{b,k}}u}_{h^{L^k}}
\leq(1+\frac{k^N}{t^N})C_N\norm{u}_{h^{L^k}}, 
\end{equation}
for every $u\in\Omega^{0,q}(X,L^k)$, where $C_N>0$ is a constant independent of $k$ and $t$. 
\end{lem} 

\begin{proof}
Fix $t>0$ and $N\in\mathbb N_0$. For $u=\varphi(s,n)\in\Omega^{0,q}(X,L^k)\subset L^2(\mathbb S\times\mathbb N,d\mu)$, we have 
\[\begin{split}
&\|e^{-\frac{t}{k}\Box^q_{b,k}}(\Box^q_{b,k})^Nu\|^2_{h^{L^k}}=\|(\Box^q_{b,k})^Ne^{-\frac{t}{k}\Box^q_{b,k}}u\|^2_{h^{L^k}}\\
&=\int\abs{e^{-\frac{t}{k}s}s^N\varphi(s,n)}^2d\mu=\int_{\set{0\leq s<1}}\abs{e^{-\frac{t}{k}s}s^N\varphi(s,n)}^2d\mu+
\int_{\set{s\geq1}}\abs{e^{-\frac{t}{k}s}s^N\varphi(s,n)}^2d\mu\\
&\leq\int_{\set{0\leq s<1}}\abs{\varphi(s,n)}^2d\mu+
\hat C_N\int_{\set{s\geq1}}\abs{\frac{k^N}{t^Ns^N}s^N\varphi(s,n)}^2d\mu\\
&\leq(1+\hat C_N\frac{k^{2N}}{t^{2N}})\int\abs{\varphi(s,n)}^2d\mu,
\end{split}\]
where $\hat C_N>0$ is a constant independent of $k$ and $t$. The lemma follows. 
\end{proof}

Let $L^2_{kF^*_k\phi_0}(B_{\log k},F^*_kT^{*0,q}X)$ be the completion of $F^*_k\Omega^{0,q}_c(B_{\log k})$ with respect to the $L^2$ inner product 
$(\,\cdot\mid\cdot\,)_{kF^*_k\phi_0}$. We extend $\dbar_{\rho,(k)}$ to $L^2_{kF^*_k\phi_0}(B_{\log k},F^*_kT^{*0,q}X)$: 
\[\dbar_{\rho,(k)}: {\rm Dom\,}\dbar_{\rho,(k)}\subset L^2_{kF^*_k\phi_0}(B_{\log k},F^*_kT^{*0,q}X)\To L^2_{kF^*_k\phi_0}(B_{\log k},F^*_kT^{*0,q+1}X),\]
where $ {\rm Dom\,}\dbar_{\rho,(k)}=\set{u\in L^2_{kF^*_k\phi_0}(B_{\log k},F^*_kT^{*0,q}X)\mid\dbar_{\rho,(k)}u\in L^2_{kF^*_k\phi_0}(B_{\log k},F^*_kT^{*0,q+1}X)}$. 
We also write 
\[\dbar^{*,kF^*_k\phi_0}_{\rho,(k)}: {\rm Dom\,}\dbar^{*,kF^*k\phi_0}_{\rho,(k)}\subset L^2_{kF^*_k\phi_0}(B_{\log k},F^*_kT^{*0,q+1}X)\To L^2_{kF^*_k\phi_0}(B_{\log k},F^*_kT^{*0,q}X)\]
to denote the $L^2$ adjoint of $\dbar_{\rho,(k)}$. Let 
\[\begin{split}
&\Box^q_{\rho,(k)}+I=\dbar^{*,kF^*_k\phi_0}_{\rho,(k)}\dbar_{\rho,(k)}+\dbar_{\rho,(k)}\dbar^{*,kF^*_k\phi_0}_{\rho,(k)}+I\\
&: {\rm Dom\,}(\Box^q_{\rho,(k)}+I)\subset L^2_{kF^*_k\phi_0}(B_{\log k},F^*_kT^{*0,q}X)\To L^2_{kF^*_k\phi_0}(B_{\log k},F^*_kT^{*0,q}X),\\
&{\rm Dom\,}(\Box^q_{\rho,(k)}+I)=\{u\in L^2_{kF^*_k\phi_0}(B_{\log k},F^*_kT^{*0,q}X)\mid 
u\in{\rm Dom\,}\dbar_{\rho,(k)}\bigcap{\rm Dom\,}\dbar^{*,kF^*_k\phi_0}_{\rho,(k)}, \\
&\quad\quad\dbar_{\rho,(k)}u\in{\rm Dom\,}\dbar^{*,kF^*_k\phi_0}_{\rho,(k)}, 
\dbar^{*,kF^*_k\phi_0}_{\rho,(k)}u\in{\rm Dom\,}\dbar_{\rho,(k)}\}.
\end{split}\] 
It is clear that $\Box^q_{\rho,(k)}$ has closed range. Let 
\[N^q_{\rho,(k)}: L^2_{kF^*_k\phi_0}(B_{\log k}, F^*_kT^{*0,q}X)\To{\rm Dom\,}(\Box^q_{\rho,(k)}+I)\]
be the inverse of $\Box^q_{\rho,(k)}+I$. We have 
\begin{equation}\label{e-gue210331ycd}
\begin{split}
\mbox{$(\Box^q_{\rho,(k)}+I)N^q_{\rho,(k)}=I$ on $L^2_{kF^*_k\phi_0}(B_{\log k}, F^*_kT^{*0,q}X)$},\\
\mbox{$N^q_{\rho,(k)}(\Box^q_{\rho,(k)}+I)=I$ on ${\rm Dom\,}(\Box^q_{\rho,(k)}+I)$},
\end{split}
\end{equation}
and 
\begin{equation}\label{e-gue210331ycdI}
\mbox{$\|N^q_{\rho,(k)}u\|_{kF^*_k\phi_0,B_{\log k}}\leq\|u\|_{kF^*_k\phi_0,B_{\log k}}$, for every $u\in L^2_{kF^*_k\phi_0}(B_{\log k}, F^*_kT^{*0,q}X)$}. 
\end{equation}

We introduce some notations. Fix $s_1, s_2\in\mathbb Z$. Let 
\[A_k: W^{s_1}_{{\rm c\,},kF^*_k\phi_0}(B_{\log k}, F^*_kT^{*0,q}X)\To W^{s_2}_{{\rm loc\,},kF^*_k\phi_0}(B_{\log k}, F^*_kT^{*0,q}X)\]
be a continuous operator. We write
\[A_k=O(k^{n_0}): W^{s_1}_{{\rm c\,},kF^*_k\phi_0}(B_{\log k}, F^*_kT^{*0,q}X)\To W^{s_2}_{{\rm loc\,},kF^*_k\phi_0}(B_{\log k}, F^*_kT^{*0,q}X),\ \ n_0\in\mathbb R,\]
if for every $\chi, \chi_1, \chi_2\in\cali{C}^\infty_c(B_{\log k})$, $\chi_2=1$ on ${\rm supp\,}\chi_1$, $\chi_1=1$ on ${\rm supp\,}\chi$, $\chi$, $\chi_1$ and $\chi_2$ are independent of $k$, there is a consnat $C>0$ independent of $k$ such that 
\[\begin{split}
&\mbox{$\|\chi_1A_k\chi u\|_{kF^*_k\phi_0,s_2,B_{\log k}}\leq Ck^{n_0}\|\chi_2u\|_{kF^*_k\phi_0,s_1,B_{\log k}}$,}\\
&\mbox{for every $u\in W^{s_1}_{{\rm c\,},kF^*_k\phi_0}(B_{\log k}, F^*_kT^{*0,q}X)$}.\end{split}\] 
The following follows from Proposition~\ref{p-gue210511yyd}

\begin{prop}\label{p-gue210331yyd}
Recall that we work with the assumption that $Y(q)$ holds on $X$. With the notations used above, for every $\ell\in\mathbb N$, $(N^q_{\rho,(k)})^\ell$ can be extended continuously to 
\[(N^q_{\rho,(k)})^\ell=O(k^0): 
W^{-\ell}_{{\rm c\,},kF^*_k\phi_0}(B_{\log k}, F^*_kT^{*0,q}X)\To L^2_{kF^*_k\phi_0}(B_{\log k}, F^*_kT^{*0,q}X).\]
\end{prop} 

We can now prove 

\begin{prop}\label{uniformly}
Let $I\subset\mathbb R_+$ be a compact set. 
Assume that $Y(q)$ holds on $X$. Let $\ell\in\mathbb N_0$ and $r>0$. Then, there is a constant $C_{\ell,r}>0$ independent of $k$ and $t$ such that 
\[\|A_{(k)}(t,x,y)\|_{\cali{C}^\ell(I\times B_r\times B_r, F^*_kT^{*0,q}X\boxtimes(F^*_kT^{*0,q}X)^*)}\leq C_{\ell,r}.\]
\end{prop}

\begin{proof}
Fix $r>0$ and $\ell\in\mathbb N$. Consider $(\Box^q_{\rho,(k)})^\ell A_{(k)}(t)$. Let $u\in F^*_k\Omega^{0,q}_c(B_r)$. Let $v\in\Omega^{0,q}_c(F_k(B_r))$ such that $F^*_kv=u$ on $B_r$. On $D$, we identify $\Box^q_{b,k}$ with $\Box^q_{b,k\phi}$ and sections of $L^k$ with functions. From \eqref{e-gue210516yyd}, \eqref{e-gue210305yyd}, \eqref{k(k)} and \eqref{e-gue210325yyd}, we have on $B_r$, 
\begin{equation}\label{e-gue210410yyd}
\begin{split}
&(\Box^q_{\rho,(k)})^{\ell}A_{(k)}(t)u=(\Box^q_{\rho,(k)})^{\ell}(F^*_kA_{k\phi_0}(\frac{t}{k})v)\\
&=\frac{1}{k^\ell}F^*_k\Bigr((\Box^q_{\rho,k\phi_0})^\ell A_{k\phi_0}(\frac{t}{k})v\Bigr)=\frac{1}{k^\ell}F^*_k\Bigr(e^{-k\rho}(\Box^q_{b,k\phi})^\ell e^{k\rho}A_{k\phi_0}(\frac{t}{k})v\Bigr)\\
&=\frac{1}{k^\ell}F^*_k\Bigr(e^{-k\rho}(\Box^q_{b,k\phi})^\ell A_{k\phi}(\frac{t}{k})e^{k\rho}v\Bigr).
\end{split}
\end{equation} 
From \eqref{e-gue210410yyd} and Lemma~\ref{l-gue210331yyd}, we have 
\begin{equation}\label{e-gue210412yyd}
\begin{split}
&\|(\Box^q_{\rho,(k)})^\ell A_{(k)}(t)u\|_{kF^*_k\phi_0,B_r}=\|k^{-\ell}F^*_k\Bigr(e^{-k\rho}(\Box^q_{b,k\phi})^\ell A_{k\phi}(\frac{t}{k})e^{k\rho}v\Bigr)\|_{kF^*_k\phi_0,B_r}\\
&=\|k^{-\ell+n+1}(\Box^q_{b,k\phi})^\ell A_{k\phi}(\frac{t}{k})e^{k\rho}v\|_{k\phi,F_k(B_r)}
\leq k^{-\ell+n+1}\|(\Box^q_{b,k})^\ell e^{-\frac{t}{k}\Box^q_{b,k}}(e^{k\rho}v)\|_{h^{L^k}}\\
&\leq\frac{k^{n+1}}{t^\ell}C_\ell\|e^{k\rho}v\|_{h^{L^k}}=\frac{k^{n+1}}{t^\ell}C_\ell\|e^{k\rho}v\|_{k\phi,F_k(B_r)}=\frac{k^{n+1}}{t^\ell}C_\ell\|v\|_{k\phi_0,F_k(B_r)}
=\frac{C_\ell}{t^\ell}\|u\|_{kF^*_k\phi_0,B_r}\\
&\leq\hat C_\ell\|u\|_{kF^*_k\phi_0,B_r},
\end{split}
\end{equation}
for every $t\in I$, where $C_\ell>0$, $\hat C_\ell>0$ are constants independent of $k$ an $t$. 
From Proposition~\ref{p-gue210511yyd} and \eqref{e-gue210412yyd}, we conclude that 
\begin{equation}\label{e-gue210412yydI}
A_{(k)}(t)=O(k^0): L^2_{{\rm c\,},kF^*_k\phi_0}(B_r,F^*_kT^{*0,q}X)\To W^s_{{\rm loc\,},kF^*_k\phi_0}(B_r,F^*_kT^{*0,q}X)
\end{equation}
is continuous, for every $s\in\mathbb Z$ and the continuity is uniformly in $t\in I$.  

We can repeat the proof of \eqref{e-gue210412yydI} with minor change and deduce that for every $\ell_1\in\mathbb N_0$, 
\begin{equation}\label{e-gue210412yydII}
A_{(k)}(t)(\Box^q_{\rho,(k)})^{\ell_1}=O(k^0): L^2_{{\rm c\,},kF^*_k\phi_0}(B_r,F^*_kT^{*0,q}X)\To W^s_{{\rm loc\,},kF^*_k\phi_0}(B_r,F^*_kT^{*0,q}X)
\end{equation}
is continuous, for every $s\in\mathbb Z$ and the continuity is uniformly in $t\in I$. By taking adjoint of \eqref{e-gue210412yydII}, we conclude that for every $\ell_1\in\mathbb N_0$, 
\begin{equation}\label{e-gue210412ycd}
(\Box^q_{\rho,(k)})^{\ell_1}A_{(k)}(t)=O(k^0): W^{-s}_{{\rm c\,},kF^*_k\phi_0}(B_r,F^*_kT^{*0,q}X)\To L^2_{{\rm loc\,},kF^*_k\phi_0}(B_r,F^*_kT^{*0,q}X)
\end{equation}
is continuous, for every $s\in\mathbb Z$ and the continuity is uniformly in $t\in I$. Let $\chi\in\cali{C}^\infty_c(B_r,\mathbb R)$, $\chi\geq0$. Consider $A_{(k)}(t)\chi(\Box^q_{\rho,(k)})^\ell$. Note that 
\[(\Box^q_{\rho,(k)})^\ell=O(k^0): L^2_{{\rm c\,},kF^*_k\phi_0}(B_r,F^*_kT^{*0,q}X)\To W^{-2\ell}_{{\rm c\,},kF^*_k\phi_0}(B_r,F^*_kT^{*0,q}X)\]
is continuous. From this observation and \eqref{e-gue210412ycd}, we deduce that for every $\ell_1\in\mathbb N_0$, we have 
\begin{equation}\label{e-gue210412ycdI}
\begin{split}
&(\Box^q_{\rho,(k)})^{\ell_1}A_{(k)}(t)\chi(\Box^q_{\rho,(k)}+I)^\ell\\
&=O(k^0): L^2_{{\rm c\,},kF^*_k\phi_0}(B_r,F^*_kT^{*0,q}X)\To 
L^2_{{\rm loc\,},kF^*_k\phi_0}(B_r,F^*_kT^{*0,q}X)
\end{split}
\end{equation}
is continuous and the continuity is uniformly in $t\in I$. From Proposition~\ref{p-gue210511yyd} and \eqref{e-gue210412ycdI}, we get 
\begin{equation}\label{e-gue210412ycdII}
A_{(k)}(t)\chi(\Box^q_{\rho,(k)}+I)^\ell=O(k^0): L^2_{{\rm c\,},kF^*_k\phi_0}(B_r,F^*_kT^{*0,q}X)\To W^{s}_{{\rm loc\,},kF^*_k\phi_0}(B_r,F^*_kT^{*0,q}X)
\end{equation}
is continuous, for every $s\in\mathbb Z$ and the continuity is uniformly in $t\in I$. Note that 
\begin{equation}\label{e-gue210412ycdIII}
\mbox{$A_{(k)}(t)\chi=A_{(k)}(t)\chi(\Box^q_{\rho,(k)}+I)^\ell(N^q_{\rho,(k)})^\ell$ on $L^2_{{\rm c\,},kF^*_k\phi_0}(B_r,F^*_kT^{*0,q}X)$},
\end{equation}
where $N^q_{\rho,(k)}$ is as in \eqref{e-gue210331ycd}. Let $u\in W^{-\ell}_{{\rm c\,},kF^*_k\phi_0}(B_r,F^*_kT^{*0,q}X)$. Let 
\[\begin{split}
&\mbox{$u_j\in L^2_{{\rm c\,},kF^*_k\phi_0}(B_r,F^*_kT^{*0,q}X)$, $j=1,2,\ldots$},\\
&\mbox{ $u_j\To u$ in $W^{-\ell}_{{\rm c\,},kF^*_k\phi_0}(B_r,F^*_kT^{*0,q}X)$
as $j\To+\infty$}.\end{split}\] 
From Proposition~\ref{p-gue210331yyd}, \eqref{e-gue210412ycdII} and \eqref{e-gue210412ycdIII}, we have for every $s\in\mathbb Z$, 
\[\begin{split}
&\|A_{(k)}(t)\chi(u_j-u_i)\|_{kF^*_k\phi_0,s,B_r}=\|A_{(k)}(t)\chi(\Box^q_{\rho,(k)}+I)^\ell(N^q_{\rho,(k)})^\ell(u_j-u_i)\|_{kF^*_k\phi_0,s,B_r}\\
&\leq C\|(N^q_{\rho,(k)})^\ell(u_j-u_i)\|_{kF^*_k\phi_0,B_{2r}}\To0
\end{split}\]
as $j, i\To+\infty$, where $C>0$ is a constant independent of $k$ and $t\in I$. Hence, $A_{(k)}(t)$ can be extended continuously to $W^{-\ell}_{{\rm c\,},kF^*_k\phi_0}(B_r,F^*_kT^{*0,q}X)$ and 
\begin{equation}\label{e-gue210412yyda}
A_{(k)}(t)\chi=O(k^0): W^{-\ell}_{{\rm c\,},kF^*_k\phi_0}(B_r,F^*_kT^{*0,q}X)\To W^s_{{\rm loc\,},kF^*_k\phi_0}(B_r,F^*_kT^{*0,q}X), 
\end{equation}
is continuous, for every $s\in\mathbb Z$ and the continuity is uniformly in $t\in I$. From \eqref{e-gue210412yyda}, Sobolev embedding theorem and notice that 
$A'_{(k)}(t)=-\Box^q_{\rho,(k)}A_{(k)}(t)$, the proposition follows. 
\end{proof}

For $f(x)\in \mathscr L(T^{*0,q}_xX,T^{*0,q}_xX)$, let 
\begin{equation}\label{e-gue210522yyd}
\abs{f(x)}_{\mathscr L(T^{*0,q}_xX,T^{*0,q}_xX)}:=\sum^d_{j,\ell=1}\abs{\langle\,f(x)v_j(x)\mid v_\ell(x)\,\rangle}, 
\end{equation}
where $\set{v_j(x)}^d_{j=1}$ is an orthonormal basis of $T^{*0,q}_xX$ with respect to $\langle\,\cdot\mid\cdot\,\rangle$. 

\begin{prop}\label{p-gue210522yyd}
Assume that $Y(q)$ holds on $X$. Let $I\subset\mathbb R_+$ be a compact interval and let $K\Subset X$ be a compact set. Then, there is a constant $C>0$ independent of $k$ such that 
\begin{equation}\label{e-gue210522yydv}
\abs{e^{-t\Box^q_{b,k}}(x,x)}_{\mathscr L(T^{*0,q}_xX,T^{*0,q}_xX)}\leq Ck^{n+1},\ \ \mbox{for all $x\in K$ and $t\in I$}. 
\end{equation}
\end{prop} 

\begin{proof}
Fix $p\in X$. We take local coordinates $x=(x_1,\ldots,x_{2n+1})=(z,\theta)=(z_1,\cdots,z_n,\theta)$ on an open set $D$ of $p$, such that \eqref{local1} hold. 
From Proposition~\ref{uniformly}, we see that on $I$, 
\[\abs{e^{-t\Box^q_{b,k}}(p,p)}_{\mathscr L(T^{*0,q}_pX,T^{*0,q}_pX)}\leq C_pk^{n+1},\]
where $C_p>0$ is a constant independent of $k$. Moreover, from the proof of Proposition~\ref{uniformly}, it is straightforward to see that 
the constant $C_p$ only depends on the upper bounds of derivatives of order $\leq s$ of $\phi$, the given volume form, and the coefficients of $\Box^q_{\rho,(k)}$ on $B_r$, for some $s\in\mathbb N$. Since $K$ is compact, $C_p$ can be taken to be independent of $p$. The proposition follows. 
\end{proof}

\subsection{The Heisenberg group $H_n$} 
\

\vspace{0.25cm}
  We pause and introduce some notations about Heisenberg group. We identify $\mathbb R^{2n+1}$
with the Heisenberg group $H_n:=\C^{n}\times\R$. Using the same setting as in \eqref{local1} and \eqref{local2} to denote by $x=(z,\theta)$ the coordinates of $H_n$, $z=(z_1,\cdots,z_n)\in \C^n$, $\theta\in \R$, $x=(x_1,\ldots,x_{2n+1})$, $z_j=x_{2j-1}+ix_{2j}$, $j=1,\ldots,n$, $\theta=x_{2n+1}$. Let 
\[T^{1,0}H_n={\rm span\,}\set{\frac{\partial}{\partial z_j}-i\lambda_j\overline z_j\frac{\partial}{\partial\theta}\mid j=1,\ldots,n}.\]
Then, $(H_n,T^{1,0}H_n)$ is a CR manifold of dimension $2n+1$. Let $\langle\,\cdot\mid\cdot\,\rangle_{H_n}$ be the Hermitian metric on 
$\mathbb CTH_n$ so that 
\begin{equation}
\begin{aligned}
\left\{U_{j,H_n}, \ol U_{j,H_n}, T=-\frac{\pa}{\pa \theta}\mid j=1,\cdots n\right\}
\end{aligned}
\end{equation}
is an orthonormal basis for $\C TH_n$, where
\begin{equation}
\begin{aligned}
U_{j,H_n}=\frac{\pa}{\pa z_j}-i\lambda_j\ol z_j\frac{\pa}{\pa \theta},\ \ j=1,\ldots,n.
\end{aligned}
\end{equation}
Then 
\begin{equation}
\begin{aligned}
\left\{dz_j, d\ol z_j, \omega_0=d\theta+\sum_{j=1}^{n}\left(i\lambda_j\ol z_jdz_j-i\lambda_jz_jd\ol z_j\right), j=1,\cdots, n\right\}
\end{aligned}
\end{equation}
is an orthonormal basis of the complexified cotangent bundle. The Levi form $\mathcal{L}_p$ of $H_n$ at $p\in H_n$ is given by $\mathcal L_p=\sum_{j=1}^{n}\lambda_jdz_j\wedge d\ol z_j$. The tangential Cauchy-Riemann operator $\dbar_{b,H_n}$ on $H_n$ is given by
\begin{equation}\label{dbarbh}
\begin{aligned}
\dbar_{b,H_n}=\sum_{j=1}^{n}d\ol z_j\wedge\ol U_{j,H_n}:\omz^{0,q}(H_n)\to \omz^{0,q+1}(H_n).
\end{aligned}
\end{equation}
Let $\Phi=\beta\theta+\sum_{j,l=1}^{n}\mu_{j,l}z_j\ol z_l$,  where $\beta$ and $\mu_{j,l}$ are the same as in \eqref{phi0}. It is easy to check that
\begin{equation}
\begin{aligned}
\sup_{(z,\theta)\in B_{\log k}}\left|kF_k^*\phi_0-\Phi\right|\to 0,\,\, \,\text{as}\,\, k\to\infty.
\end{aligned}
\end{equation}  
The Hermitian metric on $\mathbb CTH_n$ induces a Hermitian metric $\langle\,\cdot\mid\cdot\,\rangle_{H_n}$ on $T^{*0,q}H_n$. 
Let $(\,\cdot\mid\cdot\,)_{\Phi}$ be the  inner product on $\omz_c^{0,q}(H_n)$ with weight $\Phi$. Namely,
\begin{equation}\label{e-gue210506yyd}
\begin{aligned}
(\,u\mid v\,)_{\Phi}=\int_{H_n}\langle\,u\mid v\,\rangle_{H_n}e^{-\Phi}dv(z)d\theta,\quad u,v\in\omz_c^{0,q}(H_n),
\end{aligned}
\end{equation}
where $dv(z)=2^n dx_1\cdots dx_{2n}$, $z_j=x_{2j-1}+ix_{2j}$, $j=1,\cdots,n$.  Let $L^2_{(0,q)}(H_n,\Phi)$ be the completion of $\omz_c^{0,q}(H_n)$ with respect to $(\,\cdot\mid\cdot\,)_{\Phi}$. Let $\|\cdot\|_{\Phi}$ be the corresponding norm. Let $W\subset H_n$ be an open set. For $u\in L^2_{(0,q)}(H_n,\Phi)$, let 
\[\|u\|^2_{\Phi,W}:=\int_W\langle\,u\mid u\,\rangle_{H_n}e^{-\Phi}dv(z)d\theta.\]

We extend $\dbar_{b,H_n}$ to $L^2_{(0,q)}(H_n,\Phi)$: 
\[\dbar_{b,H_n}: {\rm Dom\,}\dbar_{b,H_n}\subset L^2_{(0,q)}(H_n,\Phi)\To L^2_{(0,q+1)}(H_n,\Phi),\]
where ${\rm Dom\,}\dbar_{b,H_n}=\set{u\in L^2_{(0,q)}(H_n,\Phi)\mid \dbar_{b,H_n}u\in L^2_{(0,q+1)}(H_n,\Phi)}$. 
Let 
$$\dbarstar_{b,H_n}:{\rm Dom\,}\dbarstar_{b,H_n}\subset L^2_{(0,q+1)}(H_n,\Phi)\to L^2_{(0,q)}(H_n,\Phi)$$
be the Hilbert space adjoint of $\dbar_{b,H_n}$ with respect to $(\,\cdot\mid\cdot\,)_{\Phi}$. The (Gaffney extension) of Kohn Laplacian $\Box^q_{H_n,\Phi}$ is given by
\[\begin{split}
&\Box^q_{H_n,\Phi}: {\rm Dom\,}\Box^q_{H_n,\Phi}\subset L^2_{(0,q)}(H_n,\Phi)\To L^2_{(0,q)}(H_n,\Phi),\\
& {\rm Dom\,}\Box^q_{H_n,\Phi}=\{u\in L^2_{(0,q)}(H_n,\Phi)\mid u\in{\rm Dom\,}\dbar_{b,H_n}\bigcap{\rm Dom\,}\dbarstar_{b,H_n},\\
&\quad\dbar_{b,H_n}u\in{\rm Dom\,}\dbarstar_{b,H_n}, \dbarstar_{b,H_n} u\in{\rm Dom\,}\dbar_{b,H_n}\},\\
&\Box^q_{H_n,\Phi}=\dbarstar_{b,H_n}\dbar_{b,H_n}+\dbar_{b,H_n}\dbarstar_{b,H_n}\ \ \mbox{on ${\rm Dom\,}\Box^q_{H_n,\Phi}$}. 
\end{split}\]

On $\Omega^{0,q}_c(H_n)$, a direct calculation shows that
\begin{equation}\label{boxhn}
\begin{aligned}
\Box^q_{H_n,\Phi}=&\sum_{j=1}^{n}\left(-\frac{\pa}{\pa z_j}+i\lambda_j\ol z_j\frac{\pa}{\pa \theta}-i\beta\lambda_j\ol z_j+\sum_{l=1}^{n}\mu_{l,j}\bar z_l\right)\left(\frac{\pa}{\pa \ol z_j}+i\lambda_j z_j\frac{\pa}{\pa \theta}\right)\\
&+\sum_{j,l=1}^{n}d\ol z_j\wedge \left(d\ol z_l\wedge\right)^*
\left(2i\lambda_j \delta_{j,l}\frac{\pa}{\pa \theta}+\mu_{j,l}-i\beta\lambda_j \delta_{j,l}\right),
\end{aligned}
\end{equation}
where $\left(d\ol z_l\wedge\right)^*$ is the adjoint of $d\ol z_l\wedge$ with respect to $\langle\,\cdot\mid\cdot\,\rangle_{H_n}$, $l=1,\ldots,n$. 

Let $e^{-t\Box^q_{H_n,\Phi}}$, $t>0$, be the heat operator for $\Box^q_{H_n,\Phi}$. Let 
\[e^{-t\Box^q_{H_n,\Phi}}(x,y)\in\mathscr D'(\mathbb R_+\times H_n\times H_n,T^{*0,q}H_n\boxtimes(T^{*0,q}H_n)^*)\]
be the distribution kernel of $e^{-t\Box^q_{H_n,\Phi}}$ with respect to $(\,\cdot\mid\cdot\,)_\Phi$. If $Y(q)$ holds on $H_n$, then
\[e^{-t\Box^q_{H_n,\Phi}}(x,y)\in\cali{C}^\infty(\mathbb R_+\times H_n\times H_n,T^{*0,q}H_n\boxtimes(T^{*0,q}H_n)^*).\]
Note that 
\[(e^{-t\Box^q_{H_n,\Phi}}u)(x)=\int e^{-t\Box^q_{H_n,\Phi}}(x,y)u(y)dv_{H_n}(y),\ \ u\in\Omega^{0,q}_c(H_n),\]
where $dv_{H_n}(y)=dv(z)d\theta$. We need

\begin{lem}\label{l-gue210501yyd}
Let $u\in\Omega^{0,q}(H_n)\bigcap L^2_{(0,q)}(H_n,\Phi)$ with $\Box^q_{H_n,\Phi}u\in L^2_{(0,q)}(H_n,\Phi)$. Then $u\in{\rm Dom\,}\Box^q_{H_n,\Phi}$.
\end{lem}

\begin{proof}
Let $\chi\in\cali{C}^\infty_c(H_n,[0,1])$, $\chi=1$ near $0\in H_n$. For every $M>0$, let 
$\chi_M(x):=\chi(\frac{x}{M})$. It is straightforward to check that 
\begin{equation}\label{e-gue210503yyd}
\begin{split}
&(\,\chi^2_M\dbar_{b,H_n}u\mid\dbar_{b,H_n}u\,)_\Phi+(\,\chi^2_M\dbar^*_{b,H_n}u\mid\dbar^*_{b,H_n}u\,)_\Phi\\
&\leq(\,\chi^2_M\Box^q_{H_n,\Phi}u\mid u\,)_{\Phi}+
C\Bigr(\|\chi_M\dbar_{b,H_n}u\|_\Phi\|u\|_{\Phi}+\|\chi_M\dbar^*_{b,H_n}u\|_\Phi\|u\|_\Phi\Bigr),
\end{split}
\end{equation}
where $C>0$ is a constant. From \eqref{e-gue210503yyd} and let $M\To+\infty$, we conclude that 
\begin{equation}\label{e-gue210503yydI}
\dbar_{b,H_n}u\in L^2_{(0,q+1)}(H_n,\Phi),\ \ \dbar^*_{b,H_n}u\in L^2_{(0,q-1)}(H_n,\Phi).
\end{equation}
From \eqref{e-gue210503yydI}, we see that $u\in{\rm Dom\,}\dbar_{b,H_n}$. Moreover, from \eqref{e-gue210503yydI} and Friedrich's lemma, we can check that $u\in{\rm Dom\,}\dbar^*_{b,H_n}$. 

From $\Box^q_{H_n,\Phi}u\in L^2_{(0,q)}(H_n,\Phi)$ and by using similar argument as \eqref{e-gue210503yyd}, it is straightforward to check that 
$\dbar^*_{b,H_n}\dbar_{b,H_n}u\in L^2_{(0,q)}(H_n,\Phi)$, $\dbar_{b,H_n}\dbar^*_{b,H_n}u\in L^2_{(0,q)}(H_n,\Phi)$.
From $\dbar^*_{b,H_n}\dbar_{b,H_n}u\in L^2_{(0,q)}(H_n,\Phi)$, $\dbar_{b,H_n}\dbar^*_{b,H_n}u\in L^2_{(0,q)}(H_n,\Phi)$, we conclude that $\dbar_{b,H_n}u\in{\rm Dom\,}\dbar^*_{b,H_n}$,  $\dbar^*_{b,H_n}u\in{\rm Dom\,}\dbar_{b,H_n}$. Hence, $u\in{\rm Dom\,}\Box^q_{H_n,\Phi}$. 
\end{proof} 

For any bounded open set $W\subset H_n$, we now identify $e^{-t\Box^q_{H_n,\Phi}}(x,y)$ and $A_{(k)}(t,x,y)$ as elements in 
$\cali{C}^\infty(\mathbb R_+\times W\times W, \Lambda^*(\mathbb CT^*H_n)\boxtimes(\Lambda^*(\mathbb CT^*H_n))^*)$, 
where $\Lambda^*(\mathbb CT^*H_n):=\sum^{2n+1}_{j=0}\Lambda^j(\mathbb CT^*H_n)$.

\begin{thm}\label{t-gue210503yyd}
Assume that $Y(q)$ holds. Let $I\subset\mathbb R_+$ be a bounded interval. 
Let $r>0$. We have 
\[\lim_{k\To+\infty}A_{(k)}(t,x,y)=e^{-t\Box^q_{H_n,\Phi}}(x,y)\]
in $\cali{C}^\infty(I\times B_r\times B_r,\Lambda^*(\mathbb CT^*H_n)\boxtimes(\Lambda^*(\mathbb CT^*H_n))^*)$ topology.
\end{thm} 

\begin{proof}
From Proposition~\ref{uniformly} and the Cantor diagonal argument, we can find a subsequence $\{k_1<k_2<\cdots\}$ of $\mathbb N$, $\lim_{j\To+\infty}k_j=+\infty$, such that 
\[\lim_{j\To+\infty}A_{(k_j)}(t,x,y)=P(t,x,y)\]
locally uniformly on $\mathbb R_+\times H_n\times H_n$ in $\cali{C}^\infty$ topology, where 
$P(t,x,y)\in\cali{C}^\infty(\mathbb R_+\times H_n\times H_n,T^{*0,q}H_n\boxtimes(T^{*0,q}H_n)^*)$. Let 
\[P(t): \Omega^{0,q}_c(H_n)\To\Omega^{0,q}(H_n)\]
be the operator given by 
\[(P(t)u)(x)=\int P(t,x,y)u(y)dv_{H_n}(y),\ \ u\in\Omega^{0,q}_c(H_n).\]
We claim that 
\begin{equation}\label{e-gue210503ycd}
\lim_{t\To0}P(t)u=u,\ \ \mbox{for every $u\in\Omega^{0,q}_c(H_n)$}, 
\end{equation}
\begin{equation}\label{e-gue210503ycdI}
P'(t)u+\Box^q_{H_n,\Phi}P(t)u=0,\ \ \mbox{for every $u\in\Omega^{0,q}_c(H_n)$ and $t>0$}, 
\end{equation}
\begin{equation}\label{e-gue210503ycdII}
P(t)u\in{\rm Dom\,}\Box^q_{H_n,\Phi},\ \ \mbox{for every $u\in\Omega^{0,q}_c(H_n)$ and $t>0$}.
\end{equation}
From
\begin{equation}\label{e-gue210504yyd}
A'_{(k)}(t)+\Box^q_{\rho,(k)}A_{(k)}(t)=0\ \ \mbox{on $B_{\log k}$}, 
\end{equation}
\eqref{boxrhok} and let $k\To+\infty$ in \eqref{e-gue210504yyd}, we get \eqref{e-gue210503ycdI}. Now, let $u\in\Omega^{0,q}_c(H_n)$. 
Write $u=\sum'_{\abs{J}=q}u_J(z,\theta)d\overline z^J$, $u_J\in\cali{C}^\infty_c(H_n)$, $J=(j_1,\ldots,j_q)$, $1\leq j_1<\cdots<j_q\leq n$, 
$d\overline z^J=d\overline z_{j_1}\wedge\cdots\wedge d\overline z_{j_q}$, where $\sum'$ means that the summation is over strictly increasing indices. Let 
\begin{equation}\label{e-gue210504yydbc}
u_k:=\sideset{}{'}\sum_{\abs{J}=q}u_J(z,\theta)\overline\omega^J(\frac{z}{\sqrt{k}},\frac{\theta}{k}),
\end{equation}
where $\set{\omega^1,\ldots,\omega^n}$ are as in the beginning of Section~\ref{estimate}. Fix $r>0$. For every $\ell\in\mathbb N_0$ and $t>0$, we have 
\begin{equation}\label{e-gue210504yydI}
\|(\Box^q_{H_n,\Phi})^\ell P(t)u\|_{\Phi,B_r}=\lim_{j\To+\infty}\|(\Box^q_{\rho,(k_j)})^\ell A_{(k_j)}(t)u_{k_j}\|_{k_jF^*_{k_j}\phi_0,B_r}.
\end{equation}
From \eqref{e-gue210412yyd} and \eqref{e-gue210504yydI}, we conclude that there is a constant $C_\ell>0$ independent of $t$ and $r$ such that 
\begin{equation}\label{e-gue210504yydII}
\|(\Box^q_{H_n,\Phi})^\ell P(t)u\|_{\Phi,B_r}\leq\frac{C_\ell}{t^\ell}\|u\|_{\Phi,B_r}. 
\end{equation}
Take $r\gg1$ so that ${\rm supp\,}u\subset B_r$. From \eqref{e-gue210504yydII}, we get 
$\|(\Box^q_{H_n,\Phi})^\ell P(t)u\|_{\Phi,B_r}\leq\frac{C_\ell}{t^\ell}\|u\|_\Phi$, for every $r\gg1$. Let $r\To+\infty$, we get 
\begin{equation}\label{e-gue210504yydIII}
\|(\Box^q_{H_n,\Phi})^\ell P(t)u\|_\Phi\leq\frac{C_\ell}{t^\ell}\|u\|_\Phi.
\end{equation}
From Lemma~\ref{l-gue210501yyd} and \eqref{e-gue210504yydIII}, we get  \eqref{e-gue210503ycdII}. 

We now prove \eqref{e-gue210503ycd}. Let $u\in\Omega^{0,q}_c(H_n)$ and let $u_k\in F^*_k\Omega^{0,q}_c(B_{\log k})$ be as in \eqref{e-gue210504yydbc}. We have for every $t>0$, 
\begin{equation}\label{e-gue210504yydb}
A_{(k)}(t)u_k-u_k=\int^t_0A'_{(k)}(s)u_kds=-\int^t_0\Box^q_{\rho,(k)}(A_{(k)}(s)u_k)ds.
\end{equation}
From \eqref{e-gue210412yyd}, there is a constant $C>0$ independent of $t$ and $k$ such that 
\begin{equation}\label{e-gue210504yydc}
\|A'_{(k)}(s)u_k\|_{kF^*_k\phi_0,B_r}\leq C\|u_k\|_{kF^*_k\phi_0,B_r}\leq\hat C\ \ \mbox{if $r\gg1$},
\end{equation}
where $\hat C>0$ is a constant independent of $k$ and $t$. From \eqref{e-gue210504yydc}, we can apply Lebesgue dominate theorem and get 
\[\begin{split}
&P(t)u-u=\lim_{j\To+\infty}\int^t_0A'_{(k_j)}(s)u_{k_j}ds=\int^t_0\lim_{j\To+\infty}A'_{(k_j)}(s)u_{k_j}ds\\
&=\int^t_0P'(s)uds=P(t)u-\lim_{t\To0+}P(t)u.
\end{split}\]
We get \eqref{e-gue210503ycd}. 

We now prove that $P(t)=e^{-t\Box^q_{H_n,\Phi}}$. Let $u, v\in\Omega^{0,q}_c(H_n)$. From \eqref{e-gue210503ycd}, we have 
\begin{equation}\label{e-gue210504ycd}
\begin{split}
&(\,u\mid e^{-t\Box^q_{H_n,\Phi}}v\,)_\Phi-(\,P(t)u\mid v\,)_\Phi\\
&=\int^t_0\frac{\partial}{\partial s}\Bigr((\,P(t-s)u\mid e^{-s\Box^q_{H_n,\Phi}}v\,)_\Phi\Bigr)ds\\
&=\int^t_0(\,-P'(t-s)u\mid e^{-s\Box^q_{H_n,\Phi}}v\,)_\Phi ds+\int^t_0(\,P(t-s)u\mid-\Box^q_{H_n,\Phi}(e^{-s\Box^q_{H_n,\Phi}}v)\,)_\Phi ds.
\end{split}
\end{equation}
Form \eqref{e-gue210503ycdII}, we see that $P(t-s)u\in{\rm Dom\,}\Box^q_{H_n,\Phi}$ and hence 
\begin{equation}\label{e-gue210504ycdI}
(\,P(t-s)u\mid-\Box^q_{H_n,\Phi}(e^{-s\Box^q_{H_n,\Phi}}v)\,)_\Phi =(\,-\Box^q_{H_n,\Phi}P(t-s)u\mid e^{-s\Box^q_{H_n,\Phi}}v\,)_\Phi. 
\end{equation}
From \eqref{e-gue210503ycdI}, \eqref{e-gue210504ycd} and \eqref{e-gue210504ycdI}, we deduce that 
\[(\,u\mid e^{-t\Box^q_{H_n,\Phi}}v\,)_\Phi-(\,P(t)u\mid v\,)_\Phi=0.\]
Since $(\,u\mid e^{-t\Box^q_{H_n,\Phi}}v\,)_\Phi=(\,e^{-t\Box^q_{H_n,\Phi}}u\mid v\,)_\Phi$, we conclude that $P(t)=e^{-t\Box^q_{H_n,\Phi}}$. 

We have proved that there is a subsequence $\set{k_1<k_2<\cdots}$, $\lim_{j+\infty}k_j=+\infty$, such that 
$\lim_{j\To+\infty}A_{(k_j)}(t,x,y)=e^{-t\Box^q_{H_n,\Phi}}(x,y)$ locally uniformly on $\mathbb R_+\times H_n\times H_n$ in $\cali{C}^\infty$ topology. Moreover, for any subsequence  $\set{\hat k_1<\hat k_2<\cdots}$, $\lim_{j+\infty}\hat k_j=+\infty$, we can repeat the procedure above and deduce that there is a subsequence 
$\set{\hat k_{j_1}<\hat k_{j_2}<\cdots}$, $\lim_{s+\infty}\hat k_{j_s}=+\infty$, such that $\lim_{s\To+\infty}A_{(\hat k_{j_s})}(t,x,y)=e^{-t\Box^q_{H_n,\Phi}}(x,y)$ locally uniformly on $\mathbb R_+\times H_n\times H_n$ in $\cali{C}^\infty$ topology. Hence, $\lim_{k\To+\infty}A_{(k)}(t,x,y)=e^{-t\Box^q_{H_n,\Phi}}(x,y)$ locally uniformly on $\mathbb R_+\times H_n\times H_n$ in $\cali{C}^\infty$ topology. 
\end{proof}

\section{Asymptotics of the heat kernels}\label{heatasymptotic}

Our main objective of this section is to compute $e^{-t\Box^q_{H_n,\Phi}}(x,y)$ and apply it to prove Theorem \ref{main2}.

\subsection{The heat kernel on the Heisenberg group $H_n$} 
\

\vspace{0.25cm}

Until further notice, we do not assume that $Y(q)$ holds.  
Consider $\mathbb C^n$. Let $\langle\,\cdot\mid\cdot\,\rangle_{\mathbb C^n}$ be the Hermitian inner product on $T^{*0,q}\mathbb C^n$ so that 
\[\set{d\ol z^J\mid J=(j_1,\ldots,j_q), 1\leq j_1<\cdots<j_q\le n}\]
is an orthonormal basis for $T^{*0,q}\mathbb C^n$, where $T^{*0,q}\mathbb C^n$ denotes the bundle of $(0,q)$ forms of $\mathbb C^n$. 
Let $\Omega^{0,q}(\mathbb C^n)$ denote the space of smooth $(0,q)$ forms of $\mathbb C^n$ and put $\Omega^{0,q}_c(\mathbb C^n):=\Omega^{0,q}(\mathbb C^n)\bigcap\mathscr E'(\mathbb C^n,T^{*0,q}\mathbb C^n)$. Let $(\,\cdot\mid\cdot\,)_{\mathbb C^n}$ be the $L^2$ inner product on $\Omega^{0,q}_c(\mathbb C^n)$ 
induced by $\langle\,\cdot\mid\cdot\,\rangle_{\mathbb C^n}$ and let $L^2_{(0,q)}(\mathbb C^n)$ be the completion of $\Omega^{0,q}_c(\mathbb C^n)$ with respect to 
$(\,\cdot\mid\cdot\,)_{\mathbb C^n}$. 

For $\eta\in\mathbb R$, $j, l=1,\ldots,n$, put
\begin{equation}\label{Phieta}
\Phi_{\eta}=-2\sum_{j=1}^{n}\eta\lambda_j|z_j|^2+\sum_{j,l}^{n}\mu_{j,l}\ol z_jz_l,
\end{equation}
where $\lambda_j$ and $\mu_{j,l}$ are as in \eqref{U} and \eqref{phi0}. In particular, $\Phi_0=\sum^n_{j,l=1}\mu_{j,l}\ol z_jz_l$. 
Let
$$
\dbar_{\eta}=\sum_{j=1}^nd\ol z_j\wedge\left(\frac{\pa}{\pa\ol z_j}+\frac{1}{2}\frac{\pa\Phi_\eta}{\pa\ol z_j}\right): \Omega^{0,q}(\mathbb C^n)\To\Omega^{0,q}(\mathbb C^n).
$$
Let 
\[\dbarstar_\eta:  \Omega^{0,q+1}(\mathbb C^n)\To\Omega^{0,q}(\mathbb C^n)\]
be the formal adjoint of $\dbar_\eta$ with respect to $(\,\cdot\mid\cdot\,)_{\mathbb C^n}$. It is easy to check that 
$$
\dbarstar_{\eta}=\sum_{j=1}^n(d\ol z_j\wedge)^*\left(-\frac{\pa}{\pa z_j}+\frac{1}{2}\frac{\pa\Phi_\eta}{\pa z_j}\right),
$$
where $(d\ol z_j\wedge)^*$ is the adjoint of $d\ol z_j\wedge$ with respect to $\langle\,\cdot\mid\cdot\,\rangle_{\mathbb C^n}$, $j=1,\ldots,n$. Let 
\[\Box_\eta:=\dbar_{\eta}\dbarstar_\eta+\dbarstar_\eta\dbar_\eta: {\rm Dom\,}\Box_\eta\subset L^2_{(0,q)}(\mathbb C^n)\To L^2_{(0,q)}(\mathbb C^n),\]
where ${\rm Dom\,}\Box_\eta=\set{u\in L^2_{(0,q)}(\mathbb C^n)\mid\Box_\eta u\in L^2_{(0,q)}(\mathbb C^n)}$. It is not difficult to see that $\Box_\eta$ is a non-negative self-adjoint operator. On $\Omega^{0,q}(\mathbb C^n)$, we can check that 
\begin{equation}\label{boxeta2}
\begin{aligned}
\Box_\eta=\sum_{j=1}^{n}\left(-\frac{\pa}{\pa z_j}+\frac{1}{2}\frac{\pa\Phi_\eta}{\pa z_j}\right)\left(\frac{\pa}{\pa\ol z_j}+\frac{1}{2}\frac{\pa\Phi_\eta}{\pa\ol z_j}\right)
+\sum_{j,l=1}^{n}d\ol z_j\wedge \left(d\ol z_l\wedge\right)^*
\frac{\pa^2\Phi_{\eta}}{\pa\ol z_j\pa z_l}.
\end{aligned}
\end{equation}
Let $\mathcal{S}^{0,q}(H_n)$ be the space of Schwartz test $(0,q)$ forms. Let 
\begin{equation}\label{e-gue210518yydI}
\begin{split}
G: \Omega^{0,q}_c(H_n)&\To\mathcal{S}^{0,q}(H_n),\\
u(z,\theta)&\To\int_{\mathbb R}u(z,\theta)e^{-\frac{\beta(\theta-i\lambda\abs{z}^2)}{2}-i\theta\eta-\frac{\Phi_0(z)}{2}}d\theta,
\end{split}
\end{equation}
where $\lambda\abs{z}^2:=\sum^n_{j=1}\lambda_j\abs{z_j}^2$ and we also use $\eta$ to denote $\theta$. 
The following follows from some straightforward calculation, we omit the details. 

\begin{lem}\label{l-gue210518yyd}
Let $u\in\Omega^{0,q}_c(H_n)$. We have 
\[G(\Box^q_{H_n,\Phi}u)=\Box_\eta(Gu).\]
\end{lem}

We will construct the heat kernel for $\Box^q_{H_n,\Phi}$ by using Lemma~\ref{l-gue210518yyd} and some tricks. We recall now Mehler's formula. 
Consider $\mathbb C^n$. We will use the same notations as before. Let 
\[A: T^{1,0}\mathbb C^n\To T^{1,0}\mathbb C^n\]
be an invertible self-adjoint complex endomorphism, where $T^{1,0}\mathbb C^n$ is the holomorphic tangent bundle of $\mathbb C^n$. 
We extend $A$ to $\mathbb CT\mathbb C^n$ by defining $A\ol v:=-\ol {Av}$, for any $v\in T^{1,0}\mathbb C^n$, 
then $iA$ induces an anti-symmetric endomorphism on $T\mathbb C^n$. We will denote by ${\rm det\,}$ the determinant on $T^{1,0}\mathbb C^n$ with respect to 
$\langle\,\cdot\mid\cdot\,\rangle_{\mathbb C^n}$. Write $z_j=x_{2j-1}+ix_{2j}$, $j=1,\ldots,n$. Let $e_j:=\frac{1}{\sqrt{2}}\frac{\partial}{\partial x_j}$, $j=1,\ldots,2n$. 
Then, $\set{e_j}^{2n}_{j=1}$ is an orthonormal basis of $T\mathbb C^n$. For every $x\in\mathbb C^n$, we will identify $x$ with the vector field $\sum^{2n}_{j=1}x_j\frac{\partial}{\partial x_j}$. Let
\[\mathcal H=\mathcal H(x):=-\sum_{j}^{2n}\left(e_j+\frac{1}{2}\langle\,Ax\mid e_j\,\rangle_{\mathbb C^n}\right)^2-\Tr_{T^{1,0}\C^n}[A].\]
Then, $\mathcal H(x)$ is a second order P.D.E.. Put 
\[\mathcal H: {\rm Dom\,}\mathcal H\subset L^2_{(0,q)}(\mathbb C^n)\To L^2_{(0,q)}(\mathbb C^n),\]
where ${\rm Dom\,}\mathcal H=\set{u\in L^2_{(0,q)}(\mathbb C^n)\mid\mathcal{H}u\in L^2_{(0,q)}(\mathbb C^n)}$. It is straightforward to check that 
$\mathcal H$ is self-adjoint. Let $e^{-t\mathcal H}$ be the heat operator of $\mathcal H$ and let 
\[e^{-t\mathcal H}(x,y)\in\cali{C}^\infty(\mathbb R_+\times\mathbb C^n\times\mathbb C^n, T^{*0,q}\mathbb C^n\boxtimes(T^{*0,q}\mathbb C^n)^*)\]
be the distribution kernel of $e^{-t\mathcal H}$ with respect to $(\,\cdot\mid\cdot\,)_{\mathbb C^n}$. Note that 
\[(e^{-t\mathcal H}u)(x)=\int_{\mathbb C^n}e^{-t\mathcal{H}}(x,y)u(y)dv(y),\ \ u\in L^2_{(0,q)}(\mathbb C^n),\]
where $dv(y)=2^ndy_1\cdots dy_{2n}$. We have the following Mehler's formula (see \cite[Appendix E.2]{MM07}): 

\begin{thm}\label{t-gue210522yyda}
With the notations used above, we have
	\begin{equation}\label{mehler}
	\begin{aligned}
	e^{-t\mathcal H}(x,y)=&\frac{1}{(2\pi)^n}\dfrac{\det A}{\det\left(1-e^{-2tA}\right)}
	\exp\bigg\{-\frac{1}{2}\Big\langle\,\frac{A/2}{\tanh(tA)}x\mid x\,\Big\rangle_{\mathbb C^n}\\
	&-\frac{1}{2}\Big\langle\,\frac{A/2}{\tanh(tA)}y\mid y\,\Big\rangle_{\mathbb C^n}
	+\Big\langle\,\frac{A/2}{\sinh(tA)}e^{tA}x\mid y\,\Big\rangle_{\mathbb C^n}\bigg\}.
	\end{aligned}
	\end{equation}
\end{thm}

Let $e^{-t\Box_\eta}$ be the heat operator of $\Box_\eta$ and let $e^{-t\Box_\eta}(x,y)\in\cali{C}^\infty(\mathbb R_+\times\mathbb C^n\times\mathbb C^n, T^{*0,q}\mathbb C^n\boxtimes(T^{*0,q}\mathbb C^n)^*)$ be the distribution kernel of $e^{-t\Box_\eta}$. Let $\dot{R}^\eta:T^{1,0}\C^n\To T^{1,0}\C^n$ be the linear map defined by
\begin{equation}\label{Reta}
\langle\,\dot{R}^\eta U\mid\ol V\,\rangle_{\mathbb C^n}=\pa\dbar\Phi_{\eta}(U,\ol V),\quad U,V\in T^{1,0}\C^n
\end{equation}
and let
\begin{equation}\label{omega}
\omega^\eta_{\mathbb C^n}:=\sum_{j,l=1}^n\frac{\pa^2\Phi_{\eta}}{\pa\ol z_j\pa z_l}d\ol z_j\wedge (d\ol z_l\wedge)^*,
\end{equation}
where $\Phi_\eta$ is as in \eqref{Phieta}. From\eqref{boxeta2}, we can rewrite $\Box_\eta$ as  
\begin{equation}\label{boxeta2real}
\begin{aligned}
\Box_\eta=\frac{1}{2}\left[-\sum_{j=1}^{2n}\left(e_j+\frac{1}{2}\langle\,\dot R^\eta x\mid e_j\,\rangle_{\mathbb C^n}\right)^2-\Tr \dot R^\eta \right]+\omega^\eta_{\mathbb C^n}.
\end{aligned}
\end{equation}
From \eqref{mehler} and \eqref{boxeta2real}, we deduce that 
\begin{equation}\label{e-gue210521yyd}
\begin{aligned}
e^{-t\Box_\eta}(x,y)=&\frac{1}{(2\pi)^n}\dfrac{\det\dot R^\eta}{\det\big(1-e^{-t \dot R^\eta}\big)}
\exp\bigg\{-t\omega^\eta_{\mathbb C^n}-\frac{1}{2}\Big\langle\,\frac{\dot R^\eta/2}{\tanh(t\dot R^\eta/2)}x\mid x\,\Big\rangle_{\mathbb C^n}\\
&-\frac{1}{2}\Big\langle\,\frac{\dot R^\eta/2}{\tanh(t\dot R^\eta/2)}y\mid y\,\Big\rangle_{\mathbb C^n}
+\Big\langle\,\frac{\dot R^\eta/2}{\sinh(t\dot R^\eta/2)}e^{t\dot R^\eta/2}x\mid y\,\Big\rangle_{\mathbb C^n}\bigg\}.
\end{aligned}
\end{equation}

We introduce some notations. As before, on $H_n$, let $T_{H_n}=-\frac{\partial}{\partial\theta}$ and consider 
\[-iT_{H_n}: {\rm Dom\,}(-iT_{H_n})\subset L^2_{(0,q)}(H_n,\Phi)\To L^2_{(0,q)}(H_n,\Phi),\]
where ${\rm Dom\,}(-iT_{H_n})=\set{u\in L^2_{(0,q)}(H_n,\Phi)\mid -iT_{H_n}u\in L^2_{(0,q)}(H_n,\Phi)}$. It is not difficult to see that $-iT_{H_n}$ is self-adjoint. For $\delta_1<\delta_2$, $\delta_1, \delta_2\in\mathbb R$, let 
\begin{equation}\label{e-gue210604yydIx}
Q_{[\delta_1,\delta_2]}: L^2_{(0,q)}(H_n)\To E_{-iT_{H_n}}([\delta_1,\delta_2])
\end{equation}
be the orthogonal projection with respect to $(\,\cdot\mid\cdot\,)_{H_n}$, where $E_{-iT_{H_n}}([\delta_1,\delta_2])$ denotes the spectral measure of $-iT_{H_n}$. 
It was shown in~\cite[Lemma 4.7]{HMW} that 
\begin{equation}\label{e-gue210521yydI}
(Q_{[\delta_1,\delta_2]}u)(x)=\frac{1}{(2\pi)^{2n+1}}
\int e^{i<x-y,\eta>}1_{[\delta_1,\delta_2]}(\eta_{2n+1})
u(y)dyd\eta\in\Omega^{0,q}(H_n)\bigcap L^2_{(0,q)}(H_n,\Phi),
\end{equation}
for every $u\in\Omega^{0,q}_c(H_n)$, where $1_{[\delta_1,\delta_2]}(\eta_{2n+1})=1$ if $\eta_{2n+1}\in[\delta_1,\delta_2]$, 
$1_{[\delta_1,\delta_2]}(\eta_{2n+1})=0$ if $\eta_{2n+1}\notin[\delta_1,\delta_2]$. For $\delta>0$, 
let 
\begin{equation}\label{e-gue210604yydu}
e^{-t\Box^{q,\delta}_{H_n,\Phi}}:=e^{-t\Box^q_{H_n,\Phi}}\circ Q_{[-\delta,\delta]}: L^2_{(0,q)}(H_n,\Phi)\To L^2_{(0,q)}(H_n,\Phi)
\end{equation}
and 
let $e^{-t\Box^{q,\delta}_{H_n,\Phi}}(x,y)\in\mathscr D'(\mathbb R_+\times H_n\times H_n, T^{*0,q}H_n\boxtimes(T^{*0,q}H_n)^*)$ be the 
distribution kernel of $e^{-t\Box^{q,\delta}_{H_n,\Phi}}$. It is straightforward to check that 
\begin{equation}\label{e-gue210521yydII}
(\frac{\partial}{\partial t}+\Box^q_{H_n,\Phi})e^{-t\Box^{q,\delta}_{H_n,\Phi}}=0\ \ \mbox{on $L^2_{(0,q)}(H_n,\Phi)$},
\end{equation}
\begin{equation}\label{e-gue210521yydIII}
\lim_{t\To0}e^{-t\Box^{q,\delta}_{H_n,\Phi}}=Q_{[-\delta,\delta]}\ \ \mbox{on $L^2_{(0,q)}(H_n,\Phi)$},
\end{equation}
\begin{equation}\label{e-gue210521ycd}
\lim_{\delta\To\infty}e^{-t\Box^{q,\delta}_{H_n,\Phi}}u=e^{-t\Box^{q}_{H_n,\Phi}}u,\ \ \mbox{for every $u\in L^2_{(0,q)}(H_n,\Phi)$}, 
\end{equation}
\begin{equation}\label{e-gue210521ycdI}
e^{-t\Box^{q,\delta}_{H_n,\Phi}}u\in E_{-iT_{H_n}}([-\delta,\delta]),\ \ \mbox{for every $u\in L^2_{(0,q)}(H_n,\Phi)$ and every $t>0$}. 
\end{equation}

We come back to our situation. For every $\delta>0$ and $t>0$, put 
\begin{equation}\label{e-gue210521ycdII}
\begin{split}
&P_\delta(t,x,y)\\
&:=\frac{1}{2\pi}\int e^{i<x_{2n+1}-y_{2n+1},\eta>+\frac{\beta}{2}\bigr((x_{2n+1}-y_{2n+1})+i\lambda(-\abs{z}^2+\abs{w}^2)\bigr)+\frac{-\Phi_0(w)+\Phi_0(z)}{2}}\\
&\quad\quad\quad\quad\times e^{-t\Box_\eta}(z,w)1_{[-\delta,\delta]}(\eta)d\eta\in\cali{C}^\infty(\mathbb R_+\times H_n\times H_n,T^{*0,q}H_n\boxtimes(T^{*0,q}H_n)^*),
\end{split}
\end{equation}
where $z=(x_1,\ldots,x_{2n})$, $w=(y_1,\ldots,y_{2n})$. Let
\[P_\delta(t): \Omega^{0,q}_c(H_n)\To\Omega^{0,q}(H_n)\]
be the continuous operator given by 
\begin{equation}\label{e-gue210521ycdIII}
(P_\delta(t)u)(x)=\int P_\delta(t,x,y)u(y)dv_{H_n}(y),\ \ u\in\Omega^{0,q}_c(H_n).
\end{equation} 
From Parserval's formula, it is not difficult to see that 
\begin{equation}\label{e-gue210524yyd}
P_\delta(t)u\in\Omega^{0,q}(H_n)\bigcap L^2_{(0,q)}(H_n,\Phi),\ \ \mbox{for every $u\in\Omega^{0,q}_c(H_n)$}. 
\end{equation} 
From Lemma~\ref{l-gue210518yyd} and \eqref{e-gue210521yydI}, we can check that 
\begin{equation}\label{e-gue210521yyda}
\begin{split}
&P'_\delta(t)u+\Box^{q}_{H_n,\Phi}P_\delta(t)u=0,\ \ \mbox{for every $u\in\Omega^{0,q}_c(H_n)$},\\
&\lim_{t\To0}P_\delta(t)u=Q_{[-\delta,\delta]}u, \ \ \mbox{for every $u\in\Omega^{0,q}_c(H_n)$},\\
&P_\delta(t)u\in E_{-iT_{H_n}}([-\delta,\delta]), \  \ \ \mbox{for every $u\in\Omega^{0,q}_c(H_n)$}.
\end{split}
\end{equation}
From \eqref{e-gue210521yydII}, \eqref{e-gue210521yydIII}, \eqref{e-gue210521ycdI} and \eqref{e-gue210521yyda}, 
we can repeat the procedure in the final part of the proof of Theorem~\ref{t-gue210503yyd} and deduce that 
\[P_\delta(t)=e^{-t\Box^{q,\delta}_{H_n,\Phi}}. \]

From this observation and \eqref{e-gue210521ycd}, we deduce 

\begin{thm}\label{t-gue210521yyda}
We do not assume that $Y(q)$ holds. With the notations used above, for every $\delta>0$, we have 
\begin{equation}\label{e-gue210521yydb}
\begin{split}
&e^{-t\Box^{q,\delta}_{H_n,\Phi}}(x,y)\\
&=\frac{1}{2\pi}\int e^{i<x_{2n+1}-y_{2n+1},\eta>+\frac{\beta}{2}\bigr((x_{2n+1}-y_{2n+1})+i\lambda(-\abs{z}^2+\abs{w}^2)\bigr)+\frac{-\Phi_0(w)+\Phi_0(z)}{2}}\\
&\quad\quad\quad\quad\times e^{-t\Box_\eta}(z,w)1_{[-\delta,\delta]}(\eta)d\eta\in\cali{C}^\infty(\mathbb R_+\times H_n\times H_n,T^{*0,q}H_n\boxtimes(T^{*0,q}H_n)^*),
\end{split}
\end{equation}
where $z=(x_1,\ldots,x_{2n})$, $w=(y_1,\ldots,y_{2n})$, $e^{-t\Box_\eta}(z,w)$ is given by \eqref{e-gue210521yyd}. In particular, we have 
\begin{equation}\label{e-gue210523yyda}
	e^{-t\Box_{H_n,\Phi}^{q,\delta}}(0,0)=\frac{1}{(2\pi)^{n+1}}\int_\R\dfrac{\det\dot R^\eta}{\det\big(1-e^{-t \dot R^\eta}\big)}e^{-t\omega^\eta_{\mathbb C^n}}1_{[-\delta,\delta]}(\eta)d\eta,
	\end{equation}
	where $\dot R^\eta$ and $\omega^\eta_{\mathbb C^n}$ are as in \eqref{Reta} and \eqref{omega} respectively.

Moreover, we have 
\begin{equation}\label{e-gue210523yyd}
\begin{split}
&\lim_{\delta\To+\infty}\frac{1}{2\pi}\int e^{i<x_{2n+1}-y_{2n+1},\eta>+\frac{\beta}{2}\bigr((x_{2n+1}-y_{2n+1})+i\lambda(-\abs{z}^2+\abs{w}^2)\bigr)+\frac{-\Phi_0(w)+\Phi_0(z)}{2}}\\
&\quad\quad\quad\quad\times e^{-t\Box_\eta}(z,w)1_{[-\delta,\delta]}(\eta)d\eta\\
&=e^{-t\Box^q_{H_n,\Phi}}(x,y)\  \ \mbox{in $\mathscr D'(\mathbb R_+\times H_n\times H_n, T^{*0,q}H_n\boxtimes(T^{*0,q}H_n)^*)$}.
\end{split}
\end{equation} 
\end{thm}

Now, assume that $Y(q)$ holds. From \eqref{e-gue210521yyd} and \eqref{e-gue210521ycdII}, it is straightforward to check that for every $\ell\in\mathbb N$, every compact set $K\subset\mathbb R_+\times H_n\times H_n$, there are constants $C>0$, $\varepsilon>0$, such that 
\begin{equation}\label{e-gue210521yydc}
\|P_{\delta_1}(t,x,y)-P_{\delta_2}(t,x,y)\|_{\cali{C}^\ell(K,T^{*0,q}H_n\boxtimes(T^{*0,q}H_n)^*)}\leq Ce^{-\varepsilon\delta_1},\qquad \delta_2>\delta_1>> 1.
\end{equation}
From \eqref{e-gue210521yydc}, we conclude that 
\[\begin{split}
&\lim_{\delta\To+\infty}P_\delta(t,x,y)\\
&=\frac{1}{2\pi}\int e^{i<x_{2n+1}-y_{2n+1},\eta>+\frac{\beta}{2}\bigr((x_{2n+1}-y_{2n+1})+i\lambda(-\abs{z}^2+\abs{w}^2)\bigr)+\frac{-\Phi_0(w)+\Phi_0(z)}{2}}\\
&\quad\quad\quad\quad\times e^{-t\Box_\eta}(z,w)d\eta\\
&\mbox{locally uniformly on $\mathbb R_+\times H_n\times H_n$ in $\cali{C}^\infty$ topology}.
\end{split}\]
Summing, we obtain 

\begin{thm}\label{kernelhn}
Assume that $Y(q)$ holds. We have 
	\begin{equation}\label{e-gue210523yydI}
	\begin{split}
	&e^{-t\Box_{H_n,\Phi}^q}(x,y)\\
	&=\frac{1}{2\pi}\int e^{i<x_{2n+1}-y_{2n+1},\eta>+\frac{\beta}{2}\bigr((x_{2n+1}-y_{2n+1})+i\lambda(-\abs{z}^2+\abs{w}^2)\bigr)+\frac{-\Phi_0(w)+\Phi_0(z)}{2}}\\
&\quad\quad\quad\quad\times e^{-t\Box_\eta}(z,w)d\eta,
\end{split}	\end{equation}
where $z=(x_1,\ldots,x_{2n})$, $w=(y_1,\ldots,y_{2n})$, $e^{-t\Box_\eta}(z,w)$ is given by \eqref{e-gue210521yyd}. 

In particular, we have 
\begin{equation}\label{e-gue210523yydII}
	e^{-t\Box_{H_n,\Phi}^q}(0,0)=\frac{1}{(2\pi)^{n+1}}\int_\R\dfrac{\det\dot R^\eta}{\det\big(1-e^{-t \dot R^\eta}\big)}e^{-t\omega^\eta_{\mathbb C^n}}d\eta,
	\end{equation}
	where $\dot R^\eta$ and $\omega^\eta_{\mathbb C^n}$ are as in \eqref{Reta} and \eqref{omega} respectively.
	\end{thm}

\subsection{Heat kernel asymptotics on CR manifolds} 
\

\vspace{0.25cm}
Our main task is to obtain the asymptotics of $e^{-\frac{t}{k}\Box_{b,k}^q}(x,x)$ as $k\to\infty$ for all $x\in X$. We assume that $Y(q)$ holds. Before proceeding to do so, we shall digress for the moment to illustrate the relationship between $\dot R^\eta$, the curvature of $L$ and the Levi form. 

\begin{defin}\label{d-gue210524yyd}
	Let $L$ be a CR line bundle over $X$ and $h^L$ be the Hermitian fiber metric on $L$ with local weight $\phi$. The curvature of $(L,h^L)$ at $x\in D$ with respect to $\phi$ is the Hermitian quadratic form $\mathcal R_x^{\phi}$ on $T_x^{1,0}X$ defined by
	\begin{equation}\label{curvature}
	\mathcal R_x^\phi(U,\ol V)=\frac{1}{2}\left\langle d(\dbar_b\phi-\pa_b\phi)(x), U\wedge\ol V\right\rangle,\quad U,V\in T^{1,0}_xX,
	\end{equation}
	where $d$ is the usual exterior derivative.
\end{defin}
\noindent 

Let 
\begin{equation}\label{e-gue210524yyds}
\begin{split}
&\dot{\mathcal{R}}^\phi_x: T^{1,0}_xX\To T^{1,0}_xX,\\
&\dot{\mathcal{L}}_x: T^{1,0}_xX\To T^{1,0}_xX,
\end{split}
\end{equation}
be the linear maps given by $\langle\,\dot{\mathcal{R}}^\phi_xU\mid V\,\rangle=\mathcal{R}^\phi_x(U,\ol V)$, 
$\langle\,\dot{\mathcal{L}}_xU\mid V\,\rangle=\mathcal{L}_x(U,\ol V)$, for all $U, V\in T^{1,0}_xX$. For every $\eta\in\mathbb R$, let 
\[{\rm det\,}(\dot{\mathcal{R}}^\phi_x-2\eta\dot{\mathcal{L}}_x)=\mu_1(x)\cdots\mu_n(x),\]
where $\mu_j(x)$, $j=1,\ldots,n$, are the eigenvalues of $\dot{\mathcal{R}}^\phi_x-2\eta\dot{\mathcal{L}}_x$ with respect to $\langle\,\cdot\mid\cdot\,\rangle$, 
and put 
\begin{equation}\label{e-gue210524yydt}
\omega_x^\eta=\sum_{j,l=1}^n(\mathcal{R}^\phi_x-2\eta\mathcal L_x)(U_l,\ol U_j)\ol\omega^j\wedge(\ol\omega^l\wedge)^\star: T^{*0,q}_xX\To T^{*0,q}_xX,
\end{equation}
where $\{U_j\}_{j=1}^n$ is an orthonormal frame of $T^{1,0}_xX$ with dual frame $\{\omega^j\}_{j=1}^n\subset T^{*1,0}X$. 
It should be mentioned that the definition of $\mathcal R_x^{\phi}$ depends on the choice of local trivializations (local weight $\phi$). 
From \cite[Proposition 4.2]{HM12}, it is easy to see that for every $x\in X$, the map
\[\int_\R\dfrac{\det(\dot{\mathcal R}^\phi_x-2\eta\dot{\mathcal L}_x)}{\det\big(1-e^{-t(\dot{\mathcal R}^\phi_x-2\eta\dot{\mathcal L}_x)}\big)}e^{-t\omega_x^\eta}d\eta: T^{*0,q}_xX\To T^{*0,q}_xX\]
is independent of the choice of local weight $\phi$ and hence globally defined.  

We now can prove our first main results:

\begin{proof}[Proof of Theorem~\ref{main2}]
		For $p\in X$, as above,  we choose local coordinates $(z,\theta)$ satisfies \eqref{local1}-\eqref{phi} on some neighborhood of $p$. 
		Note that $(z(p),\theta(p))=(0,0)$. It is straightforward to check that 
		\begin{equation}\label{e-gue210524ycdq}
		\int_\R\dfrac{\det\dot R^\eta}{\det\big(1-e^{-t \dot R^\eta}\big)}e^{-t\omega^\eta_{\mathbb C^n}}d\eta=\int_\R\dfrac{\det(\dot{\mathcal R}^\phi_p-2\eta\dot{\mathcal L}_p)}{\det\big(1-e^{-t(\dot{\mathcal R}^\phi_p-2\eta\dot{\mathcal L}_p)}\big)}e^{-t\omega_p^\eta}d\eta,
		\end{equation}
		where $\dot R^\eta$ and $\omega^\eta_{\mathbb C^n}$ are as in \eqref{Reta} and \eqref{omega} respectively. From Theorem~\ref{t-gue210503yyd}, Theorem~\ref{kernelhn} and \eqref{e-gue210524ycdq}, we have
		\[
		\begin{split}
    	\lim_{k\to\infty}k^{-(n+1)}e^{-\frac{t}{k}\Box^q_{b,k}}(0,0)&=\frac{1}{(2\pi)^{n+1}}\int_\R\dfrac{\det\dot R^\eta}{\det\big(1-e^{-t \dot R^\eta}\big)}e^{-t\omega^\eta_{\mathbb C^n}}d\eta\\
    	&=\frac{1}{(2\pi)^{n+1}}\int_\R\dfrac{\det(\dot{\mathcal{R}}^\phi_p-2\eta\dot{\mathcal L}_p)}{\det\big(1-e^{-t(\dot{\mathcal R}^\phi_p-2\eta\dot{\mathcal L}_p)}\big)}e^{-t\omega_p^\eta}d\eta.
    	\end{split}
		\]
		Apply this procedure for each point $x\in X$ with replace $0$ by $x$, we obtain the desired result.
	\end{proof}

\section{Morse inequalities on CR manifolds}\label{morseinequality}

In this section, we will establish the weak and strong Morse inequalities on CR manifolds from our heat kernel asymptotics. 

Now, we fix $q\in\set{0,1,\ldots,n}$ and assume that $Y(j)$ holds on $X$, for every $j=0,1,\ldots,q$. Let
\begin{equation}\label{e-gue210529yydI}
\Tr_q\left(e^{-\frac{t}{k}\Box_{b,k}^q}(x,x)\right):=\sum^d_{j=1}\langle\,e^{-\frac{t}{k}\Box_{b,k}^q}(x,x)v_j(x)\mid v_j(x)\,\rangle,
\end{equation}
where $\set{v_j}^d_{j=1}$ is an orthonormal basis for $T^{*0,q}_xX$. Let 
\begin{equation}\label{trace}
\Tr_q\left(e^{-\frac{t}{k}\Box_{b,k}^q}\right):=\int_X \Tr_q\left(e^{-\frac{t}{k}\Box_{b,k}^q}(x,x)\right)dv_X(x).
\end{equation}

It is well-known that 
\begin{equation}\label{e-gue210614yyd}
\dim H_b^{q}(X,L^k) \le\Tr_q\left(e^{-\frac{t}{k}\Box_{b,k}^q}\right),\ \ \mbox{for every $t>0$},
\end{equation}
and 
\begin{equation}\label{morse}
\sum_{j=0}^q(-1)^{q-j}\dim H_b^{j}(X,L^k) \le\sum_{j=0}^q(-1)^{q-j}\Tr_j\left(e^{-\frac{t}{k}\Box_{b,k}^j}\right),\ \ \mbox{for every $t>0$}. 
\end{equation}
For $x\in X$ and every $j=0,1,\ldots,q$, set
\begin{equation}\label{rq}
\begin{aligned}
\R_x(j)=\{\eta&\in\R\mid \dot{\mathcal R}^\phi_x-2\eta\dot{\mathcal L}_x\,\text{has exactly $j$ negative eigenvalues} 
	\\&\text{and $n-j$ positive eigenvalues}\}.
\end{aligned}
\end{equation}
Remark that since $Y(j)$ holds at each point of $X$, $\R_x(j)$ is bounded for all local weight of $L$. 
We are now in position to prove Theorem~\ref{thmmorse}.

\begin{proof}[Proof of Theorem~\ref{thmmorse}]
From Theorem~\ref{main2}, Proposition~\ref{p-gue210522yyd}, \eqref{morse} and Lebesgue dominate theorem, 
we have for every $t>0$, 
\begin{equation}\label{e-gue210530yyd}
	\begin{split}
	&\lim_{k\to\infty}k^{-(n+1)}\sum_{j=0}^q(-1)^{q-j}\dim H_b^{j}(X,L^k)\\
	&\leq\lim_{k\to\infty}k^{-(n+1)}\sum^q_{j=0}(-1)^{q-j}\Tr_j(e^{-\frac{t}{k}\Box^j_{b,k}})\\
	&=\sum^q_{j=0}(-1)^{q-j}\int_X\lim_{k\to\infty}k^{-(n+1)}\Tr_j(e^{-\frac{t}{k}\Box^j_{b,k}}(x,x))dv_X(x)\\
	&=\frac{1}{(2\pi)^{n+1}}\sum_{j=0}^q(-1)^{q-j}\int_X\int_\R\dfrac{\det(\dot{\mathcal{R}}^\phi_x-2\eta\dot{\mathcal L}_x)}{\det\big(1-e^{-t(\dot{\mathcal R}^\phi_x-2\eta\dot{\mathcal L}_x)}\big)}\Tr_j e^{-t\omega_x^\eta}d\eta dv_X(x).
		\end{split}
	\end{equation}
It is straightforward to check that for every $j=0,1,\ldots,q$, 
\begin{equation}\label{negative}
\begin{split}
	&\lim_{t\to\infty}\dfrac{\det(\dot{\mathcal{R}}^\phi_x-2\eta\dot{\mathcal L}_x)}{\det\big(1-e^{-t(\dot{\mathcal R}^\phi_x-2\eta\dot{\mathcal L}_x)}\big)}\Tr_j e^{-t\omega_x^\eta}\\
	&=(-1)^j1_{\R_x(j)}(\eta)\det(\dot{\mathcal{R}}^\phi_x-2\eta\dot{\mathcal L}_x)\\
	&=1_{\R_x(j)}(\eta)\abs{\det(\dot{\mathcal{R}}^\phi_x-2\eta\dot{\mathcal L}_x)},
	\end{split}
	\end{equation}
	where $1_{\R_x(q)}(\eta)$ is the characteristic function of $\R_x(q)$. From \eqref{e-gue210530yyd} and \eqref{negative}, we get 
	\begin{equation}\label{e-gue210530yydI}
	\begin{split}
	&\lim_{k\to\infty}k^{-(n+1)}\sum_{j=0}^q(-1)^{q-j}\dim H_b^{j}(X,L^k)\\
	&\leq\lim_{t\to\infty}\frac{1}{(2\pi)^{n+1}}\sum_{j=0}^q(-1)^{q-j}\int_X\int_\R\dfrac{\det(\dot{\mathcal{R}}^\phi_x-2\eta\dot{\mathcal L}_x)}{\det\big(1-e^{-t(\dot{\mathcal R}^\phi_x-2\eta\dot{\mathcal L}_x)}\big)}\Tr_j e^{-t\omega_x^\eta}d\eta dv_X(x)\\
	&=\frac{1}{(2\pi)^{n+1}}\sum_{j=0}^q(-1)^{q-j}\int_X\int_{\R_x(j)}\abs{\det(\dot{\mathcal{R}}^\phi_x-2\eta\dot{\mathcal L}_x)}d\eta dv_X(x).
		\end{split}
	\end{equation}
From \eqref{e-gue210530yydI}, we get \eqref{strong}. 

From \eqref{e-gue210614yyd}, we can repeat the procedure above and get \eqref{weak}. 
\end{proof}

\section{With $\mathbb R$-action}\label{S1}

In the last section, we have established the heat kernel asymptotics of Kohn Laplacian with values in $L^k$ when $Y(q)$ holds on each point of CR manifolds $X$. Nevertheless, in some important problems in CR geometry, we may need to study things without any assumption of the Levi form.
In this section, our task is to show that the same conclusion can be drawn for a CR manifold $X$ of dimensional $2n+1$, $n\ge 1$ with a transversal CR  $\mathbb R$-action.

We assume that $X$ admits a $\mathbb R$-action $\eta$, $\eta\in\mathbb R$: $\eta: X\to X$, $x\mapsto\eta\circ x$. Suppose that $X$ admits a $\mathbb R$-invariant complete Hermitian metric $\langle\,\cdot\mid\cdot\,\rangle$ on $\mathbb CTX$ so that we have the orthogonal decomposition \[
\C TX=T^{1,0}X\oplus T^{0,1}X\oplus\{\lambda T: \lambda\in\C\}
\]
and $|T|^2=\langle\,T\mid T\,\rangle=1$, where $T\in\cali{C}^\infty(X, TX)$ is the infinitesimal generator of  the $\mathbb R$-action which  is given by 
\eqref{e-gue150808}. Let $(L,h^L)$ be a rigid CR line bundle over $X$ with a $\mathbb R$-invaraint Hermitian metric $h^L$ on $X$. We will use the same notations as Section~\ref{s-gue210530yyd}. The following result due to Baouendi–Rothschild–Treves\cite[Proposition I.2]{BRT85} will be used to set up the local coordinates.
\begin{thm}
	For $p\in X$, there exist local coordinates $x=(x_1,\cdots,x_{2n+1})=(z,\theta)=(z_1,\cdots,z_n,\theta)$, $z_j=x_{2j-1}+ix_{2j}$, $j=1,\cdots,n$, defined in some small neighborhood $D=\{(z,\theta): |z|<\eps, |\theta|<\delta\}$ of $p$ such that
	\begin{equation}
	\begin{aligned}
	& T=-\frac{\pa}{\pa\theta}\\
	& Z_j=\frac{\pa}{\pa z_j}-i\frac{\pa\varphi(z)}{\pa z_j}\frac{\pa}{\pa\theta}
	\end{aligned}
	\end{equation}
	where $\{Z_j\}_{j=1}^n$ form a basis of $T_x^{1,0}X$ for each $x\in D$, $\varphi(z)\in C^\infty(D,\R)$ is independent of $\theta$. Moreover, we can take $(z,\theta)$ and $\varphi(z)$ such that $(z(p),\theta(p))=(0,0)$ and $\varphi(z)=\sum_{j=1}^n\lambda_j|z_j|^2+O(|z|^3)$ on $D$, where $\{\lambda_j\}_{j=1}^n$ are the eigenvalues of Levi form on $p$ with respect to the given $\mathbb R$-invariant Hermitian metric. We call $x=(x_1,\ldots,x_{2n+1})=(z,\theta)$ canonical local coordinates of $X$.
\end{thm} 

For $p\in X$, let $s$ be a rigid CR trivializing section of $L$ on an open set $D$ of $p$ and $|s|^2_{h^L}=e^{-\phi}$. We take canonical local coordinates $x$ and rigid CR trivializing section $s$ such that 
\eqref{local1}, \eqref{local2}, \eqref{U}  hold and 
\begin{equation}\label{e-gue210531yyds}
\phi(z)=\sum_{j,l=1}^{n}\mu_{j,l}z_j\ol z_l+O(|z|^3).
\end{equation}

On $D$, by using $s$, we will identify sections with functions in the natural way.
Follow the setting in section \ref{estimate}, we let 
\[\Box^q_{b,(k)}:=\Box^q_{\rho,(k)},\]
where $\Box^q_{\rho,(k)}$ is as in \eqref{e-gue210614yydI}. Note that $\rho=0$ and $\beta=0$ in this case. Let
\begin{equation}\label{e-gue210305yydy}
A_{k\phi,\delta}(\frac{t}{k},x,y):=e^{\frac{k\phi(x)}{2}}A_{k,s,\delta}(\frac{t}{k},x,y)e^{-\frac{k\phi(y)}{2}},
\end{equation}
where $A_{k,s,\delta}(\frac{t}{k},x,y)$ is as in \eqref{e-gue210303yydIa}. We will also use the same notations as in Section~\ref{estimate}. 
Let 
\begin{equation}\label{e-gue210325yydq}
\begin{split}
&A_{(k),\delta}(t,x,y)\\
&:=k^{-(n+1)}A_{k\phi,\delta}(\frac{t}{k},F_kx,F_ky)\in\cali{C}^\infty(\mathbb R_+\times B_{\log k}\times B_{\log k}, F^*_kT^{*0,q}X\boxtimes(F^*_kT^{*0,q}X)^*).
\end{split}
\end{equation}
Let 
\[A_{(k),\delta}(t): F^*_k\Omega^{0,q}_c(B_{\log k})\To F^*_k\Omega^{0,q}(B_{\log k})\] 
be the continuous operator given by 
\begin{equation}\label{e-gue210325yydIq}
(A_{(k),\delta}(t)u)(x)=\int A_{(k),\delta}(t,x,y)u(y)m(F_ky)dy,\ \ u\in F^*_k\Omega^{0,q}_c(B_{\log k}).
\end{equation}
It is clear that 
\begin{equation}\label{e-gue210531ycdp}
\left(\frac{\pa}{\pa t}+\Box_{b,(k)}^q\right)A_{(k),\delta}(t)u=0,\ \ \mbox{for every $u\in F^*_k\Omega^{0,q}_c(B_{\log k})$}. 
\end{equation} 

Since $T=-\frac{\partial}{\partial\theta}$, we have for every $s\in\mathbb N$, 
\begin{equation}\label{e-gue210602yyd}
\mbox{$T^sF^*_ku=k^{-s}F^*_k(T^su)$ on $B_{\log k}$}\  \ \mbox{for all $u\in\Omega^{0,q}(F_k(B_{\log k}))$}.
\end{equation}
According to Kohn's $L^2$ estimate~\cite[Theorem 8.4.2]{K65} and Remark \ref{constant}, we deduce the following 

\begin{prop}\label{p-gue210531yyd}
Let $s\in\mathbb N_0$ and let $r>0$ with $B_{2r}\subset B_{\log k}$. Then, there exists a constant $C_{r,s}>0$ independent of $k$, such that for all $u\in F_k^*\omz^{0,q}(B_{\log k})$, 
\begin{equation}
\|u\|_{kF^*_{k}\phi,s+1,B_r}\le C_{r,s}\left(\|\Box_{b,(k)}^q u\|_{kF^*_{k}\phi,s,B_{2r}}+\|T^{s+1}u\|_{kF^*_{k}\phi,B_{2r}}+\|u\|_{kF^*_{k}\phi,B_{2r}}\right).
\end{equation}
\end{prop}

\begin{lem}\label{l-gue210602yyd}
Let $s\in\mathbb N_0$. Let $r>0$ with $B_{r}\subset B_{\log k}$. We have 
\begin{equation}\label{e-gue210602yydI}
\|T^sA_{(k),\delta}(t)u\|_{kF^*_{k}\phi,B_r}\leq \delta^s\|u\|_{kF^*_k\phi,B_r},\ \ \mbox{for all $u\in F_k^*\omz^{0,q}_c(B_{\log k})$}.
\end{equation}
\end{lem}
 
 \begin{proof} 
 Fix $r>0$ and $s\in\mathbb N_0$. Let $u\in F^*_k\Omega^{0,q}_c(B_r)$. Let $v\in\Omega^{0,q}_c(F_k(B_r))$ such that $F^*_kv=u$ on $B_r$. On $D$, we identify $\Box^q_{b,k}$ with $\Box^q_{b,k\phi}$ and sections of $L^k$ with functions. From \eqref{e-gue210602yyd}, we have on $B_r$, 
\begin{equation}\label{e-gue210410yyda}
\begin{split}
&T^sA_{(k),\delta}(t)u=T^s\bigr(F^*_k(A_{k\phi,\delta}(\frac{t}{k})v)\bigr)\\
&=\frac{1}{k^s}F^*_k\Bigr(T^s A_{k\phi,\delta}(\frac{t}{k})v\Bigr).
\end{split}
\end{equation} 
From \eqref{e-gue210410yyda}, we have 
\begin{equation}\label{e-gue210412yydab}
\begin{split}
&\|T^s A_{(k),\delta}(t)u\|_{kF^*_k\phi,B_r}=\|k^{-s}F^*_k(T^s A_{k\phi,\delta}(\frac{t}{k})v)\|_{kF^*_k\phi,B_r}\\
&=\|k^{-s+n+1}T^s A_{k\phi,\delta}(\frac{t}{k})v\|_{k\phi,F_k(B_r)}
\leq k^{-s+n+1}\|T^s e^{-\frac{t}{k}\Box^q_{b,k,\leq k\delta}}v\|_{h^{L^k}}\\
&\leq k^{n+1}\delta^s\|v\|_{h^{L^k}}=k^{n+1}\delta^s\|v\|_{k\phi,F_k(B_r)}
=\delta^s\|u\|_{kF^*_k\phi,B_r}.
\end{split}
\end{equation} 
 \end{proof}

From Proposition~\ref{p-gue210531yyd}, Lemma~\ref{l-gue210602yyd} and changing $\Box^q_{\rho,(k)}$ in the proof of 
Proposition~\ref{uniformly} to $\Box^q_{\rho,(k)}-T^2$, we can repeat the proofs of Proposition~\ref{uniformly}, Proposition~\ref{p-gue210522yyd} and Theorem~\ref{t-gue210503yyd} 
with minor changes and deduce 

\begin{thm}\label{t-gue210604yyd}
Let $I\subset\mathbb R_+$ be a compact interval and let $K\Subset X$ be a compact set. Then, there is a constant $C>0$ independent of $k$ such that 
\begin{equation}\label{e-gue210604yyd}
\abs{e^{-t\Box^q_{b,k,\leq k\delta}}(x,x)}_{\mathscr L(T^{*0,q}_xX,T^{*0,q}_xX)}\leq Ck^{n+1},\ \ \mbox{for all $x\in K$ and $t\in I$}. 
\end{equation} 

Moreover, 
\begin{equation}\label{e-gue210604yydI}
\lim_{k\To+\infty}A_{(k),\delta}(t,x,y)=B_\delta(t,x,y)
\end{equation}
locally uniformly on $\mathbb R_+\times H_n\times H_n$ in $\cali{C}^\infty$ topology, where 
$B_\delta(t,x,y)\in\cali{C}^\infty(\mathbb R_+\times H_n\times H_n,T^{*0,q}H_n\boxtimes(T^{*0,q}H_n)^*)$ and $B_\delta(t,x,y)$ satisfies
\begin{equation}\label{e-gue210604yydII}
B'_\delta(t)u+\Box^q_{H_n,\Phi}B_\delta(t)u=0\ \ \mbox{for all  $u\in\Omega^{0,q}_c(H_n)$ and for all $t>0$},
\end{equation}
\begin{equation}\label{e-gue210606yyda}
B_\delta(t)u\in{\rm Dom\,}\Box^q_{H_n,\Phi},\ \ \mbox{for every $u\in\Omega^{0,q}_c(H_n)$}
\end{equation}
and 
\begin{equation}\label{e-gue210604yydIII}
\|T^s_{H_n}B_\delta(t)u\|_{\Phi}\leq\delta^s\|u\|_{\Phi},\ \ \mbox{for all $u\in\Omega^{0,q}_c(H_n)$, $t>0$ and every $s\in\mathbb N_0$}.
\end{equation}
Here $B_\delta(t)$ is the continuous operator
$B_\delta(t): \Omega^{0,q}_c(H_n)\To\Omega^{0,q}(H_n)$ given by 
\[(B_\delta(t)u)(x)=\int B_\delta(t,x,y)u(y)dv_{H_n}(y),\ \ u\in\Omega^{0,q}_c(H_n).\]
\end{thm}

We will show that $B_\delta(t)=e^{-t\Box^{q,\delta}_{H_n,\Phi}}$, where $e^{-t\Box^{q,\delta}_{H_n,\Phi}}$ is given by \eqref{e-gue210604yydu}. We need more work. 
Since $X$ admits a $\mathbb R$ invariant Hermitian metric, from~\cite[Theorem 3.5]{HHL20},  we have one of the following two cases: 
\begin{equation}\label{e-gue201109yyd}
\begin{split}
&\mbox{(a) The $\mathbb R$-action is free},\\
&\mbox{(b) The $\mathbb R$-action comes from a CR torus action 
$\mathbb T^d$ on $X$ and $\omega_0$ is $\mathbb T^d$ invariant}. 
\end{split}
\end{equation}
Let us first assume that the $\mathbb R$-action is not free and hence the $\mathbb R$-action 
comes from a CR torus action 
$\mathbb T^d=(e^{i\theta_1},\ldots,e^{i\theta_d})$ on $X$ and $\omega_0$ is $\mathbb T^d$ invariant. Since the $\mathbb R$-action comes 
from the $\mathbb T^d$-action, 
there exist $\beta_1,\ldots,\beta_d\in\mathbb R$, such that 
\begin{equation}\label{e-gue201116yyda}
T=\beta_1T_1+\ldots+\beta_dT_d, 
\end{equation}
where $T_j$ is the vector field on $X$ given by $T_ju:=
\frac{\partial}{\partial\theta_j}((1,\ldots,1,e^{i\theta_j},1,\ldots,1)^*u)|_{\theta_j=0}$, 
$u\in\Omega^{0,q}(X)$, $j=1,\ldots,d$. Recall that on $D$, we identify sections with functions. Put 
\begin{equation}\label{e-gue201115ycdp}
\hat Q_{\leq k\delta}:=(2\pi)^{-1}\int e^{i<x_{2n+1}-y_{2n+1},\eta_{2n+1}>}1_{[-k\delta,k\delta]}(\eta_{2n+1})d\eta_{2n+1}. 
\end{equation}
We consider $\hat Q_{\leq k\delta}$ as a continuous operator
\begin{equation}\label{e-gue210604yydv}
\begin{split}
&\hat Q_{\leq k\delta}: \Omega^{0,q}_c(D,L^k)\To\Omega^{0,q}(D,L^k),\\
&(\hat Q_{\leq k\delta}u)(x)=(2\pi)^{-1}\int e^{i<x_{2n+1}-y_{2n+1},\eta_{2n+1}>}1_{[-k\delta,k\delta]}(\eta_{2n+1})u(x',y_{2n+1})dy_{2n+1}d\eta_{2n+1},
\end{split}
\end{equation}
where $x'=(x_1,\ldots,x_{2n})$. We will also write $x'$ to denote $(x',0)\in\mathbb R^{2n+1}$. 
We can repeat the proof of Lemma 4.21 in~\cite{HMW} and get 

\begin{lem}\label{l-gue201116yyd} 
Assume that the $\mathbb R$-action is not free. 
Fix $D_0\Subset D$. For $u\in\Omega^{0,q}_c(D,L^k)$, we have 
\begin{equation}\label{e-gue201116yydI}
\begin{split}
&Q_{X,\leq k\delta}u=\hat Q_{\leq k\delta} u+\hat R_{\leq k\delta}u\ \ \mbox{on $D_0$},\\
&(\hat R_{\leq k\delta}u)(x)=\frac{1}{2\pi}\sum_{(m_1,\ldots,m_d)\in\mathbb Z^d}\:
\int e^{i\langle x_{2n+1}-y_{2n+1},
\eta_{2n+1}\rangle+i(\sum^d_{j=1}m_j\beta_j)y_{2n+1}-im_1\theta_1-
\ldots-im_d\theta_d}\\
&\times1_{[-k\delta,k\delta]}(\eta_{2n+1}) (1-\chi(y_{2n+1}))u((e^{i\theta_1},
\ldots,e^{i\theta_d})\circ x')d\mathbb T_dd\eta_{2n+1}dy_{2n+1}\ \ \mbox{on $D_0$}, 
\end{split}
\end{equation}
where $\chi\in\cali{C}^\infty_c(I)$, $\chi(x_{2n+1})=1$ 
for every $(x',x_{2n+1})\in D_0$ and 
$\beta_1\in\mathbb R,\ldots,\beta_d\in\mathbb R$ are as in 
\eqref{e-gue201116yyda}.
\end{lem} 

Put 
\begin{equation}\label{e-gue210604yydaz}
\tilde Q_{\delta}:=(2\pi)^{-1}\int e^{i<x_{2n+1}-y_{2n+1},\eta_{2n+1}>}1_{[-\delta,\delta]}(\eta_{2n+1})d\eta_{2n+1}. 
\end{equation}
We consider $\tilde Q_{\delta}$ as a continuous operator
\begin{equation}\label{e-gue210604yydva}
\begin{split}
&\tilde Q_{\delta}: F^*_k\Omega^{0,q}_c(B_{\log k})\To F^*_k\Omega^{0,q}(B_{\log k}),\\
&(\tilde Q_{\delta}u)(x)=(2\pi)^{-1}\int e^{i<x_{2n+1}-y_{2n+1},\eta_{2n+1}>}1_{[-\delta,\delta]}(\eta_{2n+1})u(x',y_{2n+1})dy_{2n+1}d\eta_{2n+1},
\end{split}
\end{equation}
where $x'=(x_1,\ldots,x_{2n})$. From Lemma~\ref{l-gue201116yyd} and the fact that $\lim_{t\To0}e^{-t\Box^q_{b,k,\leq k\delta}}=Q_{X,\leq k\delta}$, it is straightforward to check that 

\begin{lem}\label{l-gue210604yydz} 
Assume that the $\mathbb R$-action is not free. 
Fix $r>0$. For  $u\in F^*_k\Omega^{0,q}_c(B_{\log k})$, we have 
\begin{equation}\label{e-gue201116yydIs}
\begin{split}
&A_{(k),\delta}(0)u=\tilde Q_{\delta} u+\tilde R_{k,\delta}u\ \ \mbox{on $B_r$},\\
&(\tilde R_{k,\delta}u)(x)=\frac{1}{2\pi}\sum_{(m_1,\ldots,m_d)\in\mathbb Z^d}\:
\int e^{i\langle x_{2n+1}-y_{2n+1},
\eta_{2n+1}\rangle+i(\sum^d_{j=1}m_j\beta_j)\frac{y_{2n+1}}{k}-im_1\theta_1-
\ldots-im_d\theta_d}\\
&\times1_{[-\delta,\delta]}(\eta_{2n+1}) \left(1-\chi(\frac{y_{2n+1}}{k})\right)u((e^{i\theta_1},
\ldots,e^{i\theta_d})\circ x')d\mathbb T_dd\eta_{2n+1}dy_{2n+1}\ \ \mbox{on $B_r$}, 
\end{split}
\end{equation}
where $\chi\in\cali{C}^\infty_c(I)$ and 
$\beta_1\in\mathbb R,\ldots,\beta_d\in\mathbb R$ are as in Lemma~\ref{l-gue201116yyd}.
\end{lem} 

We need 

\begin{lem}\label{l-gue210604yydIz}
Assume that the $\mathbb R$-action is not free.
Let $f, g\in\Omega^{0,q}_c(H_n)$. Let $f_k, g_k\in F^*_k\Omega^{0,q}_c(B_{\log k})$ be as in \eqref{e-gue210504yydbc}. We have 
\begin{equation}\label{e-gue210604yydx}
\lim_{k\To+\infty}(\,A_{(k),\delta}(0)f_k\mid g_k\,)_{kF^*_k\phi}=(\,Q_{[-\delta,\delta]}f\mid g\,)_\Phi,
\end{equation}
where $Q_{[-\delta,\delta]}$ is given by \eqref{e-gue210604yydIx}. 
\end{lem} 

\begin{proof}
For simplicity, we assume that $q=0$ and hence $f_k=f$, $g_k=g$. The proof for general $(0,q)$ forms case is similar. 
From Lemma~\ref{l-gue210604yydz}, we have 
\begin{equation}\label{e-gue210604yydax}
(\,A_{(k),\delta}(0)f\mid g\,)_{kF^*_k\phi}=(\,\tilde Q_\delta f\mid g\,)_{kF^*_k\phi}+(\,\tilde R_{k,\delta} f\mid g\,)_{kF^*_k\phi}.
\end{equation}
From \eqref{e-gue210521yydI} and \eqref{e-gue210604yydva}, we can check that 
\begin{equation}\label{e-gue210604yydbx}
\lim_{k\To+\infty}(\,\tilde Q_\delta f\mid g\,)_{kF^*_k\phi}=(\,Q_{[-\delta,\delta]} f\mid g\,)_{\Phi}.
\end{equation}
We claim that 
\begin{equation}\label{e-gue210604yydcx}
\lim_{k\To+\infty}(\,\tilde R_{k,\delta}f\mid g\,)_{kF^*_k\phi}=0.
\end{equation}
From \eqref{e-gue201116yydIs} and Fourier inversion formula, we have 
\begin{equation}\label{e-gue210604yyddx}
\begin{split}
(\,&\tilde R_{k,\delta}f\mid g\,)_{kF^*_k\phi}\\
=&\frac{1}{2\pi}\sum_{(m_1,\ldots,m_d)\in\mathbb Z^d}\:
\int e^{i\langle x_{2n+1}-y_{2n+1},
\eta_{2n+1}\rangle+i(\sum^d_{j=1}m_j\beta_j)\frac{y_{2n+1}}{k}-im_1\theta_1-
\ldots-im_d\theta_d}1_{[-\delta,\delta]}(\eta_{2n+1})\\
&\times\left(1-\chi(\frac{y_{2n+1}}{k})\right)f((e^{i\theta_1},
\ldots,e^{i\theta_d})\circ x')\overline g(x)(F^*_km)(x)e^{-k(F^*_k\phi)(x)}d\mathbb T_dd\eta_{2n+1}dy_{2n+1}dx\\
=&\frac{1}{2\pi}\sum_{(m_1,\ldots,m_d)\in\mathbb Z^d}\:
\int e^{i(\sum^d_{j=1}m_j\beta_j)\frac{y_{2n+1}}{k}-im_1\theta_1-
\ldots-im_d\theta_d}\overline g(x',y_{2n+1}-\eta_{2n+1})\\
&\times (F^*_km)(x',y_{2n+1}-\eta_{2n+1})e^{-k(F^*_k\phi)(x',y_{2n+1}-\eta_{2n+1})}\widehat{1_{[-\delta,\delta]}}(\eta_{2n+1})\\
&\times  \left(1-\chi(\frac{y_{2n+1}}{k})\right)f((e^{i\theta_1},
\ldots,e^{i\theta_d})\circ x')d\mathbb T_dd\eta_{2n+1}dy_{2n+1}dx',
\end{split}
\end{equation}
where $m(x)dx=dv_X(x)$ on $D$, $x'=(x_1,\ldots,x_{2n})$, $\widehat{1_{[-\delta,\delta]}}(\eta_{2n+1})$ denotes the Fourier transform of $1_{[-\delta,\delta]}(\eta_{2n+1})$.
Note that $(F^*_k\phi)(x',y_{2n+1}-\eta)=(F^*_k\phi)(x',0)$. 
From \eqref{e-gue210604yyddx}, $\lim_{k\To+\infty}(1-\chi(\frac{y_{2n+1}}{k}))=0$ pointwise and Lebesgue dominate theorem, we get the claim \eqref{e-gue210604yydcx}. 

From \eqref{e-gue210604yydax}, \eqref{e-gue210604yydbx} and \eqref{e-gue210604yydcx}, the lemma follows. 
\end{proof}

Now, we assume that the $\mathbb R$-action is free. Let $D=U\times I$ 
be a canonical local coordinate patch with canonical local coordinates $x=(x_1,\ldots,x_{2n+1})$, where $U$ is an open set of $\mathbb C^n$ and $I$ is an open interval of 
$\mathbb R$. 
Since the $\mathbb R$-action is 
free, we can extend $x=(x_1,\ldots,x_{2n+1})$ to 
$\hat D:=U\times\mathbb R$. We identify $\hat D$ with an open set in $X$. The following is known (see~\cite[Lemma 4.7]{HMW})

\begin{lem}\label{l-gue201110yyd}
Assume that the $\mathbb R$-action is free. With the notations used above, 
for $u\in\Omega^{0,q}_c(D,L^k)$, we have 
\begin{equation}\label{e-gue201110yyd}
\begin{split}
&(Q_{X,\leq k\delta} u)(x)\\
&=\frac{1}{2\pi}\int e^{i<x_{2n+1}-y_{2n+1},\eta_{2n+1}>}1_{[-k\delta,k\delta]}(\eta_{2n+1})u(x',y_{2n+1})d\eta_{2n+1}\in\Omega^{0,q}(\hat D)
\end{split}
\end{equation}
and ${\rm supp\,}Q_{X,\leq k\delta} u\subset\hat D$, where $x'=(x_1,\ldots,x_{2n})$. 
\end{lem}

Now, we can prove 
\begin{prop}\label{p-gue210604yydex}
We have 
\begin{equation}\label{e-gue210611yyd}
\lim_{t\To0}B_\delta(t)u=Q_{[-\delta,\delta]}u,\ \ \mbox{for all $u\in\Omega^{0,q}_c(H_n)$}, 
\end{equation}
where $B_\delta(t)$ is as in Theorem~\ref{t-gue210604yyd}. 
\end{prop}

\begin{proof}
Assume first that the $\mathbb R$-action is not free. Then the $\mathbb R$-action comes from a CR torus action 
$\mathbb T^d$ on $X$ and $\omega_0$ is $\mathbb T^d$ invariant. 
Let $f, g\in\Omega^{0,q}_c(H_n)$ and let $f_k, g_k\in F^*_k\Omega^{0,q}_c(B_{\log k})$ be as in \eqref{e-gue210504yydbc}. We have for every $t>0$, 
\begin{equation}\label{e-gue210504yydbz}
\begin{split}
&(\,A_{(k),\delta}(t)f_k\mid g_k\,)_{kF^*_k\phi}-
(\,A_{(k),\delta}(0)f_k\mid g_k\,)_{kF^*_k\phi}\\
&=\int^t_0(\,A'_{(k),\delta}(s)f_k\mid g_k\,)_{kF^*_k\phi}ds\\
&=-\int^t_0\left(\,\Box^q_{b,(k)}(A_{(k),\delta}(s)f_k)\,\big|\, g_k\,\right)_{kF^*_k\phi}ds.
\end{split}
\end{equation}
From \eqref{e-gue210412yyd} and \eqref{e-gue210325yydq}, we can apply Lebesgue dominate theorem and by using \eqref{e-gue210604yydx} to obtain
\begin{equation}\label{e-gue210604ycdx}
\begin{split}
&(\,B_\delta(t)f\mid g\,)_\Phi-(\,Q_{[-\delta,\delta]}f\mid g\,)_\Phi\\
&=\lim_{k\To+\infty}\int^t_0(\,A'_{(k),\delta}(s)f_k\mid g_k\,)_{kF^*_k\phi}ds\\
&=\int^t_0\lim_{k\To+\infty}(\,A'_{(k),\delta}(s)f_k\mid g_k\,)_{kF^*_k\phi}ds\\
&=\int^t_0(\,B'_{\delta}(s)f\mid g\,)_\Phi ds\\
&=(\,B_\delta(t)f\mid g\,)_\Phi-\lim_{t\To0}(\,B_\delta(t)f\mid g\,)_\Phi,
\end{split}\end{equation}
which implies \eqref{e-gue210611yyd}. 

Assume that the $\mathbb R$-action is free. Let $f\in\Omega^{0,q}_c(H_n)$ and let $f_k\in F^*_k\Omega^{0,q}_c(B_{\log k})$ be as in \eqref{e-gue210504yydbc}. From \eqref{e-gue201110yyd}, we can check that 
\begin{equation}\label{e-gue210611yydI}
\lim_{k\To+\infty}A_{(k),\delta}(0)f_k=Q_{[-\delta,\delta]}f\ \ \mbox{on every $B_r$, $r>0$}. 
\end{equation}
Repeat the procedure above, we then get \eqref{e-gue210611yyd}. 
\end{proof}

We need

\begin{prop}\label{p-gue210604yydfx}
We have 
\[B_\delta(t)u\in E_{-iT_{H_n}}([-\delta,\delta]),\ \ \mbox{for all $t>0$ and all $u\in\Omega^{0,q}_c(H_n)$}.\]
\end{prop}

\begin{proof}
Fix $u\in\Omega^{0,q}_c(H_n)$ and $t>0$. Let $\varepsilon>0$, $\varepsilon\ll1$. The map
\[-iT_{H_n}: E_{-iT_{H_n}}([-\frac{1}{\varepsilon},-\delta-\varepsilon]\bigcup[\delta+\varepsilon,\frac{1}{\varepsilon}])\To E_{-iT_{H_n}}([-\frac{1}{\varepsilon},-\delta-\varepsilon]\bigcup[\delta+\varepsilon,\frac{1}{\varepsilon}])\]
is one to one and onto. 

Write $(-iT_{H_n})^{-1}$ to denote the inverse of $-iT_{H_n}$ on $E_{-iT_{H_n}}([-\frac{1}{\varepsilon},-\delta-\varepsilon]\bigcup[\delta+\varepsilon,\frac{1}{\varepsilon}])$. It is clear that 
\begin{equation}\label{e-gue210606yyd}
\|(-iT_{H_n})^sg\|_{\Phi}\leq\Bigr(\frac{1}{\delta+\varepsilon}\Bigr)^s\|g\|_{\Phi},
\end{equation}
for all $s\in\mathbb N$ and every $g\in 
E_{-iT_{H_n}}([-\frac{1}{\varepsilon},-\delta-\varepsilon]\bigcup[\delta+\varepsilon,\frac{1}{\varepsilon}])$. Let $v\in E_{-iT_{H_n}}([-\frac{1}{\varepsilon},-\delta-\varepsilon]\bigcup[\delta+\varepsilon,\frac{1}{\varepsilon}])$. 
From \eqref{e-gue210604yydIII} and \eqref{e-gue210606yyd}, we have 
\begin{equation}\label{e-gue210606yydI}
\begin{split}
&\abs{(\,B_\delta(t)u\mid v\,)_\Phi}=\abs{(\,B_\delta(t)u\mid(-iT_{H_n})^s(-iT_{H_n})^{-s}v\,)_\Phi}\\
&=\abs{(\,(-iT_{H_n})^sB_\delta(t)u\mid(-iT_{H_n})^{-s}v\,)_\Phi}\leq\Bigr(\frac{\delta}{\delta+\varepsilon}\Bigr)^s\|u\|_\Phi\|v\|_\Phi,
\end{split}
\end{equation}
for every $s\in\mathbb N$. Let $s\To+\infty$ in \eqref{e-gue210606yydI}, we get $(\,B_\delta(t)u\mid v\,)_\Phi=0$. Hence, 
$B_\delta(t)u$ is orthogonal to  $E_{-iT_{H_n}}([-\frac{1}{\varepsilon},-\delta-\varepsilon]\bigcup[\delta+\varepsilon,\frac{1}{\varepsilon}])$, for every $\varepsilon>0$, $\varepsilon\ll1$. Hence, $B_\delta(t)u\in E_{-iT_{H_n}}([-\delta,\delta])$. 
\end{proof}

\begin{thm}\label{t-gue210606yyd}
We have 
\[B_\delta(t)=e^{-t\Box^{q,\delta}_{\Phi,H_n}},\]
where $e^{-t\Box^{q,\delta}_{\Phi,H_n}}$ is given by \eqref{e-gue210604yydu}. 
\end{thm}

\begin{proof}
Let $u, v\in\Omega^{0,q}_c(H_n)$. From Propostion~\ref{p-gue210604yydex} and Proposition~\ref{p-gue210604yydfx}, we have 
\begin{equation}\label{e-gue210504ycdz}
\begin{split}
&(\,u\mid e^{-t\Box^{q,\delta}_{H_n,\Phi}}v\,)_\Phi-(\,B_\delta(t)u\mid v\,)_\Phi\\
&=\int^t_0\frac{\partial}{\partial s}\Bigr((\,B_{\delta}(t-s)u\mid e^{-s\Box^{q,\delta}_{H_n,\Phi}}v\,)_\Phi\Bigr)ds\\
&=\int^t_0(\,-B'_\delta(t-s)u\mid e^{-s\Box^{q,\delta}_{H_n,\Phi}}v\,)_\Phi ds+\int^t_0(\,B_\delta(t-s)u\mid-\Box^{q}_{H_n,\Phi}(e^{-s\Box^{q,\delta}_{H_n,\Phi}}v)\,)_\Phi ds.
\end{split}
\end{equation}
Form \eqref{e-gue210606yyda}, we see that $B_\delta(t-s)u\in{\rm Dom\,}\Box^q_{H_n,\Phi}$ and hence 
\begin{equation}\label{e-gue210504ycdIz}
(\,B_\delta(t-s)u\mid-\Box^q_{H_n,\Phi}(e^{-s\Box^{q,\delta}_{H_n,\Phi}}v)\,)_\Phi =(\,-\Box^{q}_{H_n,\Phi}B_\delta(t-s)u\mid e^{-s\Box^{q,\delta}_{H_n,\Phi}}v\,)_\Phi. 
\end{equation}
From \eqref{e-gue210604yydII}, \eqref{e-gue210504ycdz} and \eqref{e-gue210504ycdIz}, we deduce that 
\[(\,u\mid e^{-t\Box^{q,\delta}_{H_n,\Phi}}v\,)_\Phi-(\,B_\delta(t)u\mid v\,)_\Phi=0.\]
Since $(\,u\mid e^{-t\Box^{q,\delta}_{H_n,\Phi}}v\,)_\Phi=(\,e^{-t\Box^{q,\delta}_{H_n,\Phi}}u\mid v\,)_\Phi$, we conclude that $B_\delta(t)=e^{-t\Box^{q,\delta}_{H_n,\Phi}}$. 
\end{proof}

From Theorem~\ref{t-gue210521yyda} and Theorem~\ref{t-gue210606yyd}, we can repeat the proofs of Theorem~\ref{main2} and Theorem~\ref{thmmorse} and get Theorem~\ref{main3} and Theorem~\ref{thms1}.



\end{document}